\documentclass[11pt]{amsart}

\usepackage{amsmath,amssymb}
\usepackage{bm}
\usepackage{graphicx}
\usepackage{ascmac}
\usepackage{amsthm}
\usepackage{amsfonts}
\usepackage{mathrsfs}

\theoremstyle{definition}
\newtheorem{thm}{Theorem}[section]
\newtheorem{dfn}[thm]{Definition}
\newtheorem{ex}[thm]{Example}
\newtheorem{cor}[thm]{Corollary}
\newtheorem{prop}[thm]{Proposition}
\newtheorem{lem}[thm]{Lemma}
\newtheorem{rem}[thm]{Remark}

\newtheorem*{sotp}{Strategy of our proof}
\newtheorem*{oftp}{Organization of this paper}
\newtheorem*{ack}{Acknowledgment}

\numberwithin{thm}{section}

\newcommand{\Zpn}{\mathbb{Z}_{>0}}
\newcommand{\Znn}{\mathbb{Z}_{\geq 0}}
\newcommand{\Z}{\mathbb{Z}}
\newcommand{\Zp}{{\mathbb{Z}}_p}
\newcommand{\Zpt}{{\mathbb{Z}}_{p}^{\times}}
\newcommand{\Zps}{{\mathbb{Z}}_{p^2}}
\newcommand{\Zpst}{{\mathbb{Z}}_{p^2}^{\times}}
\newcommand{\Q}{\mathbb{Q}}
\newcommand{\Qp}{{\mathbb{Q}}_p}
\newcommand{\Qpt}{{\mathbb{Q}}_{p}^{\times}}
\newcommand{\Qps}{{\mathbb{Q}}_{p^2}}

\renewcommand{\O}{\mathcal{O}}

\newcommand{\R}{\mathbb{R}}
\newcommand{\C}{\mathbb{C}}
\newcommand{\Fp}{\mathbb{F}_p}

\newcommand{\Fpbar}{\overline{\mathbb{F}}_p}
\newcommand{\Fpt}{\mathbb{F}^{\times}_p}
\newcommand{\Fpst}{\mathbb{F}^{\times}_{p^2}}
\newcommand{\plim}[1][]{\mathop{\varprojlim}\limits_{#1}}

\newcommand{\A}{\mathbb{A}}
\newcommand{\D}{\mathbb{D}}
\newcommand{\G}{\mathbb{G}}
\renewcommand{\H}{\mathbb{H}}
\renewcommand{\P}{\mathbb{P}}

\newcommand{\X}{\mathbb{X}}

\newcommand{\bD}{\mathbf{D}}
\newcommand{\E}{\mathbf{E}}
\newcommand{\bG}{\mathbf{G}}
\renewcommand{\L}{\mathbf{L}}

\newcommand{\bV}{\mathbf{V}}
\newcommand{\bX}{\mathbf{X}}
\newcommand{\bi}{\mathbf{i}}
\newcommand{\bv}{\mathbf{v}}

\newcommand{\B}{\mathcal{B}}

\newcommand{\cL}{\mathcal{L}}
\newcommand{\M}{\mathcal{M}}

\newcommand{\cZ}{\mathcal{Z}}

\newcommand{\sC}{\mathscr{C}}
\newcommand{\sM}{\mathscr{M}}
\newcommand{\sS}{\mathscr{S}}

\newcommand{\Sbar}{\overline{S}}
\newcommand{\ebar}{\overline{e}}
\newcommand{\fbar}{\overline{f}\!}
\newcommand{\gbar}{\overline{g}}
\renewcommand{\hbar}{\overline{h}}
\newcommand{\sbar}{\overline{s}}
\newcommand{\Lambdabar}{\overline{\Lambda}}
\newcommand{\Omegabar}{\overline{\Omega}}
\newcommand{\pibar}{\overline{\pi}}

\DeclareMathOperator{\Ker}{Ker}
\DeclareMathOperator{\Ima}{Im}
\DeclareMathOperator{\Hom}{Hom}
\DeclareMathOperator{\End}{End}
\DeclareMathOperator{\id}{id}

\DeclareMathOperator{\Lie}{Lie}
\DeclareMathOperator{\Frac}{Frac}
\DeclareMathOperator{\tr}{Tr}
\DeclareMathOperator{\trd}{Trd}
\DeclareMathOperator{\nrd}{Nrd}
\DeclareMathOperator{\ord}{ord}

\DeclareMathOperator{\GL}{GL}
\DeclareMathOperator{\GSp}{GSp}
\DeclareMathOperator{\GU}{GU}
\DeclareMathOperator{\GSpin}{GSpin}

\DeclareMathOperator{\Sp}{Sp}

\DeclareMathOperator{\SO}{SO}
\DeclareMathOperator{\PGSp}{PGSp}

\DeclareMathOperator{\N}{N}
\DeclareMathOperator{\nilp}{Nilp}

\DeclareMathOperator{\spec}{Spec}
\DeclareMathOperator{\spf}{Spf}

\DeclareMathOperator{\loc}{loc}
\DeclareMathOperator{\sml}{sim}
\DeclareMathOperator{\red}{red}
\DeclareMathOperator{\length}{length}
\DeclareMathOperator{\OGr}{OGr}
\DeclareMathOperator{\inv}{inv}
\DeclareMathOperator{\Res}{Res}
\DeclareMathOperator{\si}{ss}

\DeclareMathOperator{\Proj}{Proj}
\DeclareMathOperator{\diag}{diag}
\DeclareMathOperator{\rk}{rank}

\DeclareMathOperator{\Vrt}{VL}
\DeclareMathOperator{\Vtx}{Vert}
\DeclareMathOperator{\Edg}{Edge}
\DeclareMathOperator{\VE}{VE}
\DeclareMathOperator{\BT}{BT}

\DeclareMathOperator{\ad}{ad}
\DeclareMathOperator{\Irr}{Irr}
\DeclareMathOperator{\Sh}{Sh}
\DeclareMathOperator{\an}{an}
\DeclareMathOperator{\ns}{nsm}
\DeclareMathOperator{\nfs}{nfs}

\DeclareMathOperator{\bs}{basic}
\DeclareMathOperator{\stab}{Stab}
\DeclareMathOperator{\disc}{disc}

\DeclareMathOperator{\pr}{pr}
\DeclareMathOperator{\Ram}{Ram}
\DeclareMathOperator{\KMPS}{KMPS}

\DeclareMathOperator{\Tor}{Tor}
\DeclareMathOperator{\sr}{sr}
\DeclareMathOperator{\nsr}{nsr}
\DeclareMathOperator{\hs}{hs}
\DeclareMathOperator{\nsp}{ns}
\DeclareMathOperator{\Fil}{Fil}
\DeclareMathOperator{\Lift}{Lift}
\DeclareMathOperator{\Isot}{Isot}
\DeclareMathOperator{\Lag}{Lag}
\DeclareMathOperator{\GM}{GM}

\usepackage{geometry}
\title[The supersingular locus]{On supersingular loci of Shimura varieties for quaternionic unitary groups of degree $2$}
\author[Y.\ Oki]{Yasuhiro Oki}
\address{Graduate School of Mathematical Sciences, 
the University of Tokyo, 3-8-1 Komaba, Meguro-ku, Tokyo 153-8914, Japan.}
\email{oki@ms.u-tokyo.ac.jp}

\begin{document}
\maketitle

\begin{abstract}
We describe the structure of the supersingular locus of a Shimura variety for a quaternionic unitary similitude group of degree $2$ over a ramified odd prime $p$ if the level at $p$ is given by a special maximal compact open subgroup. More precisely, we show that such a locus is purely $2$-dimensional, and every irreducible component is birational to the Fermat surface. Furthermore, we have an estimation of the numbers of connected and irreducible components. To prove these assertions, we completely determine the structure of the underlying reduced scheme of the Rapoport--Zink space for the quaternionic unitary similitude group of degree $2$, with a special parahoric level. We prove that such a scheme is purely $2$-dimensional, and every irreducible component is isomorphic to the Fermat surface. We also determine its connected components, irreducible components and their intersection behaviors by means of the Bruhat--Tits building of $\PGSp_4(\Qp)$. In addition, we compute the intersection multiplicity of the GGP cycles associated to an embedding of the considering Rapoport--Zink space into the Rapoport--Zink space for the unramified $\GU_{2,2}$ with hyperspecial level for the minuscule case. 
\end{abstract}

\tableofcontents

\section{Introduction}\label{intr}

Let $(\bG,\bX)$ be a Shimura datum, that is, a pair consisting of a reductive connected group $\bG$ over $\Q$ and a finite disjoint union of hermitian symmetric domains $\bX$ satisfying certain conditions. Let $\A_f$ be the finite ad{\`e}le ring of $\Q$. For a compact open subgroup $K$ of $\bG(\A_f)$, we associate a complex manifold
\begin{equation*}
\Sh_K(\bG,\bX)^{\an}=\bG(\Q)\backslash (\bX\times \bG(\A_f)/K), 
\end{equation*}
called the Shimura variety. If $K$ is sufficiently small, the Shimura variety $\Sh_K(\bG,\bX)^{\an}$ descends to a quasi-projective variety $\Sh_K(\bG,\bX)$ over a number field $E$ depending on $(\bG,\bX)$, which is called the reflex field. Next, let $p$ be a prime number, and $\A_f^p$ the finite ad{\`e}le ring without $p$-component. Fix a place $\nu$ of $E$ above $p$, and we denote by $O_{E,\nu}$ the completion of the integer ring of $E$ with respect to $\nu$. Moreover, we assume that $(\bG,\bX)$ is of PEL type and $K=K^pK_p$, where $K^p$ is a subgroup of $\bG(\A_f^p)$, and $K_p$ is a parahoric subgroup of $G(\Qp)$. Then, Rapoport and Zink construct in \cite{rz} a scheme $\sS_{K}$ over $O_{E,\nu}$, which is a moduli space of abelian varieties with additional structures. In particular, if $G\otimes_{\Q}\R$ is isomorphic to $\GSp_{2n}$ for some $n\in \Zpn$, then it is an integral model of $\Sh_K(\bG,\bX)$. Now consider the geometric special fiber $\sS_{K,\Fpbar}$ of $\sS_{K}$, and define $\sS_K^{\si}$ as the reduced closed subscheme of $\sS_{K,\Fpbar}$ consisting of points such that the corresponding abelian varieties are supersingular. This $\sS_K^{\si}$ is called the supersingular locus. 

Making concrete descriptions of supersingular loci is important. For example, for unitary and orthogonal Shimura varieties, we have an application to the arithmetic intersection problem conjectured by Kudla in \cite{kud}. It predicts a relation between the intersection multiplicities of certain cycles (called special cycles) and the Fourier coefficients of the derivatives of an Eisenstein series. Some partial results are known for the conjecture above by using descriptions of supersingular loci, e.g.~\cite{kr} for Siegel $3$-folds and \cite{ter} for Hilbert modular surfaces. 

There are many known results for descriptions of supersingular loci. We give some results which are related to this paper in the following: 
\begin{itemize}
\item The case for $\bG=\GSp_{2n}$ (that is, a Siegel modular variety) and $K_p$ is hyperspecial. For $n=2,3$, Katsura and Oort \cite{ko}, \cite{ko2} describe the structure of $\sS_{K}^{\si}$. More precisely, when $n=2$, they proved that each irreducible component is birational to $\P^1_{\Fpbar}$. For general $n$, Li and Oort \cite{lo} give formulas on dimension and the number of irreducible components of $\sS_{K}^{\si}$. 
\item The case for $\bG=\GU_{1,n-1}$ associated to an imaginary quadratic extension $L/\Q$, $p$ is an odd prime which inerts in $L$, and $K_p$ is hyperspecial. In this case, the structure of $\sS_{K}^{\si}$ is given by Vollaard \cite{vol} for $n=3$, and Vollaard and Wedhorn \cite{vw} in general. Recently, Cho \cite{cho} considered in the case when $K_p$ is a parahoric subgroup given by the stabilizer of a single lattice and any $n$. 
\item The case for $\bG=\GU_{2,2}$ associated to an imaginary quadratic extension $L/\Q$, $p$ is an odd prime which inerts in $L$, and $K_p$ is hyperspecial. In this case, Howard and Pappas \cite{hp} give an explicit description of the structure of $\sS_{K}^{\si}$. Their method relies on an exceptional isomorphism between $\GU_{2,2}$ and $\GSpin_{4,2}$ corresponding to the identity $A_3=D_3$ of the Dyinkin diagrams. 
\item The case for $\bG=\GSpin_{n,2}$, $p$ is an odd prime and $K_p$ is hyperspecial. The structure of $\sS_{K}^{\si}$ is given by Howard and Pappas \cite{hp2} by using the same method as that for $\GU_{2,2}$ as above. 
\end{itemize}
For more history of works, see the introductions of \cite{vol} and \cite{wu}. 

In this paper, we describe the supersingular locus of the integral model of the Shimura variety for a quaternionic unitary similitude group of degree $2$ with coefficient $\bD$, an indefinite quaternion algebra over $\Q$, when $p$ is an odd prime which ramifies in $\bD$ and $K_p$ is a special maximal compact open subgroup. We also consider the related Rapoport--Zink space, which is associated to the quaternionic unitary similitude group of degree $2$. Moreover, as another application of the results on the Rapoport--Zink space, we compute the intersection multiplicity of certain cycles, which is called the GGP cycles. Let us explain our results more precisely in the sequel. 

\emph{Throughout this paper, let $p>2$ be an odd prime number. }

\subsection{Main theorem: local results}\label{mrlv}
Let $\Qp$ be the field of $p$-adic numbers, and $D$ the quaternion division algebra over $\Qp$. Write $D$ as below: 
\begin{equation*}
D=\Qp[\varepsilon ,\Pi],\quad \varepsilon^2\in \Zpt \setminus (\Zpt)^2,\quad \Pi^2=p,\quad \Pi \varepsilon=-\varepsilon \Pi. 
\end{equation*}
We define an involution $*$ on $D$ by 
\begin{equation*}
d^{*}:=\varepsilon(\trd(d)-d)\varepsilon^{-1}. 
\end{equation*}
Then the maximal order $O_D$ of $D$ is stable under $*$. Let $\Fp$ be a finite field of $p$ elements. Fix an algebraic closure $\Fpbar$ of $\Fp$, and denote by $W=W(\Fpbar)$ the ring of Witt vectors over $\Fpbar$. Consider triples $(X,\iota,\lambda)$ over a $W$-scheme $S$ on which $p$ is locally nilpotent, where 
\begin{itemize}
\item $X$ is a 4-dimensional $p$-divisible group over $S$, 
\item $\iota \colon O_D\rightarrow \End(X)$ is a ring homomorphism, 
\item $\lambda \colon X \rightarrow X^{\vee}$ is a polarization, 
\end{itemize}
such that the following conditions are fulfilled for any $d\in O_D$: 
\begin{itemize}
\item $\det(T-\iota(d) \mid \Lie(X))=(T^2-\trd_{D/\Qp}(d)T+\nrd_{D/\Qp}(d))^2$, 
\item $\lambda \circ \iota(d)=\iota(d^{*})^{\vee}\circ \lambda$. 
\end{itemize}
Fix an object $(\X_0,\iota_0,\lambda_0)$ over $\spec \Fpbar$ such that $\X_0$ is isoclinic of slope $1/2$ and $\lambda_0$ is an isomorphism, and define $\M_G$ (the subscript $G$ will be introduced in Section \ref{rzdt}) as the moduli space of $O_D$-linear quasi-isogenies $\rho \colon X\times_S \Sbar \rightarrow \X_0\times_{\spec \Fpbar} \Sbar$ such that the pull-back of the polarization of $X$ equals $c(\rho)\lambda_0$ for some $c(\rho)\in \Qpt$. Here, $\Sbar$ is the closed subscheme of $S$ defined by the ideal sheaf $p\O_S$. It is a formal scheme over $\spf W$, which is locally formally of finite type. Note that there is an action of $\GSp_4(\Qp)$ on $\M_G$. Moreover, there is a decomposition into open and closed formal subschemes
\begin{equation*}
\M_G=\coprod_{i\in \Z}\M_G^{(i)},
\end{equation*}
where $\M_G^{(i)}$ is the locus of $\M_G$ where $\ord_p(c(\rho))=i$. The action of $\GSp_4(\Qp)$ on $\M_G$ implies that each $\M_G^{(i)}$ are isomorphic to each other. 

We denote by $\B$ the Bruhat--Tits building of $\PGSp_4(\Qp)$. The main theorem in this section is describing the underlying reduced subscheme $\M_G^{\red}$ of $\M_G$ by means of $\B$. Let $\Vtx^{\hs}$ be the set of all hyperspecial vertices of $\B$, $\Vtx^{\nsp}$ the set of all non-special vertices of $\B$, and $\Edg^{\hs}$ the set of all edges connecting two adjacent hyperspecial vertices. Moreover, put
\begin{equation*}
\VE:=\Vtx^{\hs}\sqcup \Vtx^{\nsp}\sqcup \Edg^{\hs}. 
\end{equation*}
We introduce an order $\leq$ on $\VE$. For distinct $x,y\in \VE$, we have $x<y$ if one of the following hold: 
\begin{itemize}
\item $x\in \Vtx^{\nsp}$, $y\in \Vtx^{\hs}$ and $x,y$ are adjacent, 
\item $x\in \Vtx^{\nsp}$, $y\in \Edg^{\hs}$ and $\{x\}\cup y$ is the set of all vertices of a $2$-simplex in $\B$, 
\item $x\in \Edg^{\hs}$, $y\in \Vtx^{\hs}$ and $y\in x$. 
\end{itemize}

For $x\in \VE$, we associate a reduced closed subscheme $\M_{G,x}$ of the underlying reduced scheme $\M_G^{\red}$ of $\M_G$. See Section \ref{btsn}. Moreover, put $\M_{G,x}^{(0)}:=\M_{G,x}\cap \M_G^{(0)}$. 

\begin{thm}\label{thlc}
\begin{enumerate}
\item(Theorem \ref{mtl1} (i)) \emph{For $x,y\in \VE$, we have $\M_{G,y}\subset \M_{G,x}$ if and only if $y\leq x$. }
\item(Theorem \ref{mtl1} (ii)) \emph{Let $x\in \VE$. 
\begin{itemize}
\item If $x \in \Vtx^{\nsp}$, then $\M_{G,x}^{(0)}$ is a single $\Fpbar$-rational point. 
\item If $x \in \Edg^{\hs}$, then $\M_{G,x}^{(0)}$ is isomorphic to $\P^1_{\Fpbar}$. 
\item If $x \in \Vtx^{\hs}$, then $\M_{G,x}^{(0)}$ is isomorphic to the Fermat surface defined by
\begin{equation*}
x_0^{p+1}+x_1^{p+1}+x_2^{p+1}+x_3^{p+1}=0
\end{equation*}
in $\Proj \Fpbar [x_0,x_1,x_2,x_3]$. 
\end{itemize}
In particular, $\M_{G,x}^{(0)}$ is projective, smooth and irreducible of dimension $d(x)$, where
\begin{equation*}
d(x)=
\begin{cases}
2 &\text{if }x\in \Vtx^{\hs},\\
1 &\text{if }x\in \Edg^{\hs},\\
0 &\text{if }x\in \Vtx^{\nsp}. 
\end{cases}
\end{equation*}}
\end{enumerate}
\end{thm}

For $x\in \VE$, put $\BT_{G,x}^{(0)}:=\M_{G,x}^{(0)}\setminus \bigcup_{y<x}\M_{G,y}^{(0)}$. Then the closure of $\BT_{G,x}^{(0)}$ in $\M_G^{(0),\red}$ equals $\M_{G,x}^{(0)}$ by Theorem \ref{thlc} (ii). Moreover, Theorem \ref{thlc} (i) implies that it is equal to $\coprod_{y\leq x}\BT_{G,y}^{(0)}$. 

\begin{thm}\label{mthm}
\begin{enumerate}
\item(Theorem \ref{mtl1} (iii)) \emph{We have a locally closed stratification (called the Bruhat--Tits stratification)
\begin{equation*}
\M_G^{(0),\red}=\coprod_{x\in \VE}\BT_{G,x}^{(0)}. 
\end{equation*}
Hence $\M_G^{(0)}$ is connected and is purely $2$-dimensional. }
\item(Corollary \ref{dbst}) \emph{For $x\in \VE$, $\BT_{G,x}^{(0)}$ is isomorphic to the Deligne--Lusztig variety for $\GSp_{2d(x)}$ associated to the Coxeter element, where $d(x)$ is an integer defined in (ii). }
\end{enumerate}
\end{thm}

In addition to Theorem \ref{mthm}, we can also describe the non-formally smooth locus of the formal scheme $\M_{G}$ by using the Bruhat--Tits strata. Note that $\M_G$ is regular and flat over $\spf W$. See Corollary \ref{rzsg}. 

\begin{thm}\label{thsg}
\item(Theorem \ref{mtl1} (iv)) \emph{The non-formally smooth locus of $\M_G^{(0)}$ equals $\coprod_{\Vtx^{\nsp}}\BT_{G,x}^{(0)}$. }
\end{thm}

\subsection{Main theorem: global results}\label{mrgv}

Let $\bD$ be an indefinite quaternion algebra over $\Q$ which is ramified at $p$, and $\bV$ a skew-hermitian $\bD$-module of rank $2$ in the sense of \cite{kot3} as in Section \ref{shqt}. Put $\bG:=\GU(\bV)$, which is an algebraic group over $\Q$. Then, $\bG$ is a non-trivial inner form of $\GSp_4$ over $\Q$. In particular, the same assertion holds after the base change from $\Q$ to $\Qp$. On the other hand, we have $\bG \otimes_{\Q}\R \cong \GSp_4$. Let $\bX$ be a $\bG(\R)$-conjugacy class inside the set of all homomorphisms $\Res_{\C/\R}\G_m\rightarrow \bG \otimes_{\Q}\R$ which contains the homomorphism induced by
\begin{equation*}
\Res_{\C/\R}\G_m\rightarrow \GSp_4;a+b\sqrt{-1}\mapsto 
{\begin{pmatrix}
aE_2&-bE_2\\
bE_2&aE_2
\end{pmatrix}
}. 
\end{equation*}
Here $E_2$ is the unit matrix of size $2$. Now fix an order $O_{\bD}$ of $\bD$ which is maximal at $p$. We write $K=K^pK_p$, where $K^p\subset \bG(\A_f^p)$ and $K_p\subset \bG(\Qp)$. Then, we further assume that $K_p$ is the stabilizer of a self-dual $O_{\bD}\otimes_{\Z}\Zp$-lattice, see Section \ref{shqt}. Then the algebraic variety $\Sh_K(\bG,\bX)$ is defined over $\Q$, and it is $3$-dimensional. Moreover, $\sS_K$ is defined over $\Z_{(p)}$ as a moduli space of abelian $4$-folds with $O_{\bD}$-linear actions, polarizations and $K^p$-level structures. See Section \ref{ssdf}. 

\begin{thm}\label{thgl}
\begin{enumerate}
\item (Theorem \ref{nsss}) \emph{The scheme ${\sS}_{K,W}:=\sS_K\times_{\spec \Zp}\spec W$ is regular and flat over $\spec W$. In particular, ${\sS}_{K,\Fpbar}$ is $3$-dimensional. Moreover, it is smooth over $\spec W$ outside a finite set of $\Fpbar$-rational points. }
\item (Theorem \ref{mtsh} (i)) \emph{The scheme $\sS_{K}^{\si}$ is purely $2$-dimensional. Every irreducible component is projective and birational to the Fermat surface defined by
\begin{equation*}
x_0^{p+1}+x_1^{p+1}+x_2^{p+1}+x_3^{p+1}=0
\end{equation*}
in $\Proj \Fpbar [x_0,x_1,x_2,x_3]$. }
\item (Theorem \ref{mtsh} (ii)) \emph{Let $F$ be an irreducible component of $\sS_{K}^{\si}$. Then the following hold:
\begin{itemize}
\item There are at most $(p+1)(p^2+1)$-irreducible components of $\sS_{K}^{\si}$ whose intersections with $F$ are birational to $\P^1_{\Fpbar}$. Here we endow the intersections with reduced structures. 
\item There are at most $(p+1)(p^2+1)$-irreducible components of $\sS_{K}^{\si}$ which intersect $F$ at a single $\Fpbar$-rational point. 
\item Other irreducible components of $\sS_{K}^{\si}$ do not intersect $F$. 
\end{itemize}}
\item (Theorem \ref{mtsh} (iii)) \emph{Each non-formally smooth point in $\sS_{K}^{\si}$ is contained in at most $2(p+1)$-irreducible components. }
\item (Theorem \ref{mtsh} (iv)) \emph{Each irreducible component of $\sS_{K}^{\si}$ contains at most $(p+1)(p^2+1)$-non-formally smooth points. }
\end{enumerate}
\end{thm}

\begin{rem}
Very recently, Wang proved in \cite{wan} the same results as Sections \ref{mrlv} and \ref{mrgv}. His research is independent of the author's one, and the author recognized it after completing this paper. Wang's method is direct, that is, it does not use an exceptional isomorphism. Moreover, he also consider a relation between his Bruhat--Tits stratification and the affine Deligne--Lusztig variety for $G$. On the other hand, our method can be expected to generalize for certain spinor similitude groups of arbitrary degrees. See the strategy of our proof after Section \ref{icmr} for details. 
\end{rem}

\subsection{Application: computation of the intersection multiplicity of the GGP cycles}\label{icmr}

Let $\Qps$ be the unramifed quadratic extension of $\Qp$, and denote by $\tau$ the non-trivial Galois automorphism of $\Qps$ over $\Qp$. Consider triples $(X,\iota,\lambda)$ over a $W$-scheme $S$ on which $p$ is locally nilpotent, where 
\begin{itemize}
\item $X$ is a 4-dimensional $p$-divisible group over $S$, 
\item $\iota \colon \Zps \rightarrow \End(X)$ is a ring homomorphism, 
\item $\lambda \colon X \rightarrow X^{\vee}$ is a polarization, 
\end{itemize}
such that the following conditions are fulfilled for any $a\in \Zps$: 
\begin{itemize}
\item $\det(T-\iota(a) \mid \Lie(X))=(T^2-\tr_{\Qps/\Qp}(a)T+\N_{\Qps/\Qp}(a))^2$, 
\item $\lambda \circ \iota(a)=\iota(\tau(a))^{\vee}\circ \lambda$. 
\end{itemize}
Note that $(\X_0,\iota_0,\lambda_0)$ be the object as in Section \ref{mrlv}. Define $\M_H$ (the subscript $H$ will be $\GU_{2,2}$, the unramified unitary similitude group of signature $(2,2)$) as the moduli space of quasi-isogenies $\rho \colon X\times_S \Sbar \rightarrow \X_0\times_{\spec \Fp} \Sbar$ which commutes with the additional structures. It is a formal scheme over $\spf W$, which is locally formally of finite type. We have a relation between $\M_G$ and $\M_H$ as follows: 
\begin{prop}\label{mhit}
(Proposition \ref{mgeb}) \emph{There is a closed immersion 
\begin{equation*}
i_{G,H}\colon\M_G\rightarrow \M_H
\end{equation*}
whose image consists of $(X,\iota,\lambda,\rho)\in \M_H$ satisfying $\rho^{-1}\circ \iota_0({\Pi})\circ \rho \in \End(X)$ (see Section \ref{mrlv} for the definition of $\Pi$). }
\end{prop}

Now put $\End^0(\X_0):=\End(\X_0)\otimes_{\Z}\Q$ and $\End_{O_D}^0(\X_0):=\End_{O_D}(\X_0)\otimes_{\Z}\Q$. We define a $6$-dimensional quadratic space as
\begin{equation*}
\L_{\Q}^{\Phi}:=\iota_0(\varepsilon \Pi)^{-1}\circ \{f\in \End_{O_D}^0(\X_0)\mid \trd_{\End_{O_D}^0(\X_0)/\Qp}(f)=0,f^{\vee}=f\}\oplus \Qp \iota_0(\Pi^{-1})\subset \End^0(\X_0)
\end{equation*}
with quadratic form $v\mapsto v\circ v$ over $\Qp$. Moreover, put $J_H^0:=\GSpin(\L_{\Q}^{\Phi})$, which is an algebraic group over $\Qp$. Then $J_H^0(\Qp)$ acts on $\L_{\Q}^{\Phi}$ and $\M_H$. We define $\Delta$ as the image of 
\begin{equation*}
(\id,i_{G,H})\colon p^{\Z}\backslash \M_G\rightarrow p^{\Z}\backslash (\M_G\times_{\spf W}\M_H), 
\end{equation*}
called the \emph{GGP cycle}. Moreover, for $g\in J_H^0(\Qp)$, put $g\Delta:=(\id \times g)(\Delta)$. We consider the intersection multiplicity 
\begin{equation*}
\langle \Delta,g\Delta \rangle:=\chi(p^{\Z}\backslash (\M_G\times_{\spf W}\M_H), \O_{\Delta}\otimes^{\mathbb{L}}\O_{g\Delta}). 
\end{equation*}
We call that $g$ is regular semi-simple and minuscule if the $\Zp$-submodule
\begin{equation*}
L(g):=\sum_{i=0}^{5}\Zp(g^i\cdot \iota_0(\Pi^{-1}))
\end{equation*}
of $\L_{\Q}^{\Phi}$ is a lattice, and satisfies $pL(g)^{\natural}\subset L(g)\subset L(g)^{\natural}$, where $L(g)^{\natural}$ is the dual lattice of $L(g)$ in $\L_{\Q}^{\Phi}$. If $g$ is regular semi-simple and minuscule and $\M_H^g\neq \emptyset$, then $g$ induces an action on $L(g)^{\natural}/L(g)$. Let $P_g$ be the characteristic polynomial of $g$ on $L(g)^{\natural}/L(g)$. 

For a non-zero polynomial $R\in \Fp[T]$, we define the reciprocal of $R$ by 
\begin{equation*}
R^{*}(T):=T^{\deg(R)}R(T^{-1}), 
\end{equation*}
and we call that $R$ is self-reciprocal if $R^{*}=R$. Then $P_g$ is self-reciprocal. Let $\Irr(P_g)$ be the set of all monic irreducible factors of $P_g$, and $\Irr^{\sr}(P_g)$ the set of all self-reciprocal monic irreducible factors of $P_g$. Moreover, put $\Irr^{\nsr}(P_g):=\Irr(P_g)\setminus \Irr^{\sr}(P_g)$. Let $\Irr^{\nsr}(P_g)/\sim$ be the quotient of the set $\Irr^{\nsr}(P_g)$ by the relation $R\sim cR^*$ for some $c\in \Fpt$. Moreover, for $R\in \Irr(P_g)$, let $m(R)$ be the multiplicity of $R$ in $P_g$. Then, the function $m$ on $\Irr^{\nsr}(P_g)$ factors through $\Irr^{\nsr}(P_g)/\sim$. 

\begin{thm}\label{imgg}
(Theorem \ref{mnin}) \emph{Assume that $g\in J_H^0(\Qp)$ is regular semi-simple, minuscule and satisfies $\M_H^g\neq \emptyset$. 
\begin{enumerate}
\item The following are equivalent: 
\begin{itemize}
\item We have $\Delta \cap g\Delta \neq \emptyset$. 
\item There is a unique $Q_{g}\in \Irr^{\sr}(P_{g})$ such that $m(Q_{g})$ is odd.
\end{itemize}
If the conditions above hold, then we have
\begin{equation*}
\#(\Delta \cap g\Delta)(\Fpbar)=\deg Q_{g}\prod_{[R]\in \Irr^{\nsr}(P_{g})/\sim}(1+m(R))<\infty. 
\end{equation*}
\item We have an equality
\begin{equation*}
\langle \Delta,g\Delta \rangle=\deg Q_g\frac{m(Q_g)+1}{2}\prod_{[R]\in \Irr^{\nsr}(P_g)/\sim}(1+m(R)). 
\end{equation*}
\end{enumerate}}
\end{thm}

\begin{sotp}
First, we explain the strategy of the proof of Theorems \ref{thlc} and \ref{thgl}. For results on singularities of $\M_G$ and $\sS_{K,W}$, we determine the structure of the local model explicitly. On the other hand, for results on $\M_G^{\red}$ and $\sS_{K}^{\si}$, the strategy is based on the case for $\GU_{1,n-1}$ over an inert prime established by \cite{vol} and \cite{vw}. The method is as follows: 
\begin{enumerate}
\item construct a locally closed stratification of the corresponding Rapoport--Zink space by means of a certain Bruhat--Tits building, 
\item apply the $p$-adic uniformization theorem of Rapoport--Zink \cite[Theorem 6.30]{rz} for our Shimura variety, and reduce to the result in (i). 
\end{enumerate}
In our case, the Rapoport--Zink space appearing in (i) is $\M_G$, and the involving Bruhat--Tits building is $\B$. Our approach to (i) is as follows. 

We regard $\M_G$ as a closed formal subscheme of $\M_H$ by the closed immersion $i_{G,H}$ in Proposition \ref{mhit}. Under this situation, we can construct a closed immersion $\GSp_4\rightarrow J_H^0$ which is compatible with the actions on $\M_G$ and $\M_H$. This induces an embedding of the Bruhat--Tits building $i\colon \B \rightarrow \B'$, where $\B'$ is the Bruhat--Tits building of $(J_H^0)^{\ad}(\Qp)=\SO(\L_{\Q}^{\Phi})(\Qp)$. On the other hand, Howard and Pappas constructed in \cite{hp} a stratification $\{\BT_{H,\bullet}\}$ of $\M_H$ by means of $\B'$. We denote by $\M_{H,\bullet}$ the closure of $\BT_{H,\bullet}$ in $\M_H$. Now we define $\M_{G,x}:=\M_{H,i(x)}\cap \M_G$ for $x\in \VE$. Then we can prove the equality $\M_{G,x}=\M_{H,i(x)}$ and the desired assertions in Theorem \ref{thlc}. 

Note that our method relies on the exceptional isomorphism $\PGSp_{4}\cong \SO_{5}$ corresponding to the identity $C_2=B_2$ of Dynkin diagrams. Hence we cannot expect to generalize it to quaternionic unitary similitude groups of arbitrary degrees. However, we can hope it for a family of spinor similitude groups. More precisely, we expect such a generalization which associates an embedding $\SO_{n+1}\rightarrow \SO_{n+2}$. This is one of our future problems. 

Next, we explain the strategy of the proof of Theorem \ref{imgg}. We use the argument introduced in \cite{lz}, which proves a similar result for Rapoport--Zink spaces for unramified spinor similitude groups with hyperspecial level structures. The most different point from \cite{lz} is the proof of the reducedness of the Bruhat--Tits strata of $\M_H$. Our proof is based on generalizing some linear algebraic results in \cite[\S 2.4]{hp}, which does not involve the integral models of Shimura varieties. 
\end{sotp}

\begin{oftp}
In Section \ref{rzgu}, we specify a Rapoport--Zink space $\M_G$ that we consider in this paper. In Section \ref{wkhp}, we recall and refine the argument in \cite[\S 2]{hp} to study $\M_G$. In Section \ref{lmfl}, we consider the flatness and the non-formally smooth locus of $\M_G$. In Section \ref{vlbt}, we introduce the notion of vertex lattices and construct the Bruhat--Tits stratification. In Section \ref{dlvs}, we interpret the Bruhat--Tits strata in terms of Deligne--Lusztig varieties. In Section \ref{shvr} we specify a Shimura variety and its supersingular locus that we consider, and prove the global results as the first application of our local results. In Section \ref{arit}, we define the GGP cycle, and compute the intersection multiplicity for the minuscule case as the second application of our local results. In appendix \ref{hpcr}, we give a complement of the proof of \cite[Corollary 2.14]{hp}. In Appendix \ref{rdhp}, we give a proof of the reducedness of the Bruhat--Tits strata of $\M_H$. 
\end{oftp}

\begin{ack}
This paper is the author's master thesis. I would like to thank my advisor Yoichi Mieda for his constant support and encouragement. He carefully read the draft version of this paper, and pointed out many mistakes and typos. 

This work was supported by the Program for Leading Graduate Schools, MEXT, Japan.
\end{ack}

\subsection{Notation}\label{nota}
In this paper, we use the following terminologies: 
\begin{itemize}
\item \textbf{$p$-adic fields and their extensions. } For an algebraically closed field $k$ of characteristic $p$, let $W(k)$ be the ring of Witt vectors over $k$. Denote by $\sigma$ the $p$-th power Frobenius on $k$ or its lift to $W(k)$. If $k=\Fpbar$, write $W$ and $K_0$ for $W(k)$ and $\Frac(W(k))$ respectively. We normalize the $p$-adic valuation $\ord_p$ on $\Q$, $\Qp$ or $\Frac(W(k))$ so that $\ord_p(p)=1$. 
\item \textbf{Dual lattices. } Let $O_F$ be one of $\Z_{(p)}$, $\Zp$ and $W$, where $\Z_{(p)}$ is the valuation ring in $\Q$ with respect to $\ord_p$. Put $F:=\Frac(O_F)$. Further, let $(V,(\,,\,))$ be a finite-dimensional vector space over $F$ with a symmetric or a symplectic form, and $\Lambda \subset V$ an $O_F$-lattice. Then we denote the dual lattice of $\Lambda$ by
\begin{equation*}
\Lambda^{\vee}:=\{v\in V\mid (v,\Lambda)\subset O_F \}, 
\end{equation*}
except for $V=\L_{\Q}^{\Phi}$ appered in Section \ref{icmr}. In this case, we denote the above dual lattice by $\Lambda^{\natural}$. 
\item \textbf{Perpendiculars. } Let $k$ be a field, $V$ a symplectic or a quadratic space over $k$ of finite-dimensional, and $W$ a $k$-subspace of $V$. Then we denote by $U^{\perp}$ the perpendicular of $U$ in $V$.  
\item \textbf{Hermitian forms. } Let $k$ be a field, and $B$ either a quadratic extension of $k$ or a quaternion algebra over $k$. We put $\overline{b}:=\trd(b)-b$ for $b\in B$. Let $V$ be a finite free \emph{left} $B$-module. In this paper, a $B/k$-hermitian form on $V$ means a \emph{non-degenerate} $B$-valued bilinear form $\langle \,,\,\rangle$ satisfying $\langle x_1,x_2\rangle=\overline{\langle x_2,x_1\rangle}$ and
\begin{equation*}
\langle b_1 x_1,b_2 x_2\rangle=b_1\langle x_1,x_2\rangle \overline{b_2}
\end{equation*}
for $b_1,b_2\in B$ and $x_1,x_2\in V$. We can apply the results in \cite{sch} to it, since we may regard $V$ as a right $B$-module by restricting scalars by $b\mapsto \overline{b}$. We say that a pair $(V,\langle\,,\,\rangle)$ where $\langle\,,\,\rangle$ is a $B/k$-hermitian form on $V$ as a $B/k$-hertmitian space. 
\item \textbf{Symplectic similitude group. } For $n\in \Zpn$, let 
\begin{equation*}
J_{2n}:={\begin{pmatrix}
0& I_n\\
-I_n &0
\end{pmatrix}}, 
\end{equation*}
where $I_n$ is the anti-diagonal matrix of size $n$ that has $1$ at every non-zero entry. We define an algebraic group $\GSp_{2n}$ over $\Z$ as
\begin{equation*}
\GSp_{2n}(R)=\{(A,c)\in \GL_{2n}(R)\times \G_m(R)\mid {}^tAJ_{2n}A=cJ_{2n}\}
\end{equation*}
for any ring $R$. 
\end{itemize}

\section{Rapoport--Zink space for $\GU_2(D)$}\label{rzgu}

In this section, we recall a Rapoport--Zink space for $\GU_2(D)$. We use the notation in Section \ref{mrlv} for a while. 

\subsection{Rapoport--Zink datum}\label{rzdt}
First, we recall a \emph{Rapoport--Zink datum} of PEL type defined in \cite[Definition 3.18]{rz}. See also \cite[Definition 2.1]{mie}. It is a tuple $(B,*,O_B,V,(\,,\,),b,\mu,\mathcal{L})$, where
\begin{itemize}
\item $B$ is a finite-dimensional semi-simple algebra over $\Qp$, 
\item $*\colon B\rightarrow B$ is an involution, 
\item $O_B$ is a maximal order of $B$ which is stable under $*$, 
\item $V$ is a left $B$-module which is finite-dimensional over $\Qp$, 
\item $(\,,\,)\colon V\times V\rightarrow \Qp$ is a non-degenerate symplectic form such that $(dx,y)=(x,d^*y)$ for any $d\in B$ and $x,y\in V$. 
\end{itemize}
By the data above, we define an algebraic group $G$ over $\Qp$ by
\begin{equation}\label{utsl}
G(R)=\{(g,c)\in \GL_{B\otimes_{\Qp} R}(V\otimes_{\Qp} R)\times \G_m(R)\mid (g(v),g(w))=c(v,w)\text{ for all }v,w\in V\otimes_{\Qp} R\}
\end{equation}
for each $\Qp$-algebra $R$. Then we can explain the remaining objects: 
\begin{itemize}
\item $b\in G(K_0)$, 
\item $\mu \colon \G_{m}\otimes_{\Z} K\rightarrow G\otimes_{\Qp}K$ is a cocharacter defined over a finite extension $K$ of $K_0$, 
\item $\mathcal{L}$ is a self-dual lattice chain, see \cite[Definition 3.1, Definition 3.13]{rz}. 
\end{itemize}
We require the following conditions: 
\begin{itemize}
\item the isocrystal $(V\otimes_{\Qp}K_0,b\circ \sigma)$ has slopes in the interval $[0,1]$, 
\item $\sml(b)=p$, where $\sml \colon G\rightarrow \G_m;(g,c)\mapsto c$ is the similitude character of $G$, 
\item the weight decomposition of $V\otimes_{\Qp}K$ with respect to $\mu$ contains only the weights $0$ and $1$. 
\end{itemize}

We introduce a Rapoport--Zink datum considered in this paper. Note that a partial datum is given in \cite[1.42]{rz}. 

Let $D:=\Qps [\Pi]$ be the quaternion division algebra over $\Qp$ such that $\Pi^2=p$ and $\Pi a=\tau(a)\Pi$ for any $a\in \Qps$. Fix an element $\varepsilon \in \Zpst$, and define an involution $*$ on $D$ by
\begin{equation*}
d^{*}:=\varepsilon \overline{d} \varepsilon^{-1}. 
\end{equation*}
Then we have ${\Pi}^{*}=\Pi$ and $a^{*}=\tau(a)$ for any $a\in \Qps$. Moreover, the unique maximal order $O_D$ of $D$ is stable under $*$. 

Next, let $V:=D^{\oplus 2}$ be a left $D$-module, and $\Lambda^0:=O_D^{\oplus 2}\subset V$ an $O_D$-lattice. Let us define a $\Qp$-valued bilinear form $(\,,\,)$ on $V$ by
\begin{equation*}
((x_1,x_2),(y_1,y_2)):=\trd_{D/\Qp}(\Pi^{-1}(x_1^*y_2-x_2^*y_1))
\end{equation*} 
for $(x_1,x_2),(y_1,y_2)\in V$. 

\begin{lem}\label{sdlt}
\emph{The bilinear form $(\,,\,)$ on $V$ is non-degenerate, alternating and satisfies the following properties:
\begin{enumerate}
\item $(dx,y)=(x,d^{*}y)$ for all $d\in D$ and $x,y\in V$. 
\item $(\Lambda^{0})^{\vee}=\Lambda^0$. 
\end{enumerate}}
\end{lem}

\begin{proof}
Before the proof, we recall some properties of the reduced trace. The following hold: 
\begin{itemize}
\item For any $d_1,d_2\in D$, we have $\trd_{D/\Qp}(d_1d_2)=\trd_{D/\Qp}(d_2d_1)$. 
\item For any $d\in D$, we have $\trd_{D/\Qp}(\overline{d})=\trd_{D/\Qp}(d)$. 
\end{itemize}
Using these properties, for $d\in D$ we have
\begin{equation*}
\trd_{D/\Qp}(d^{*})=\trd_{D/\Qp}(\varepsilon \overline{d}\varepsilon^{-1})=\trd_{D/\Qp}(\overline{d})=\trd_{D/\Qp}(d). 
\end{equation*}

First, we prove that $(\,,\,)$ is alternating. Take $(x_1,x_2),(y_1,y_2)\in V$. Then we have
\begin{equation*}
((x_1,x_2),(y_1,y_2))=\trd_{D/\Qp}((\Pi^{-1}(x_1^*y_2-x_2^*y_1))^{*})=\trd_{D/\Qp}((\Pi^{*})^{-1}(y_2^*x_1-y_1^*x_2)). 
\end{equation*}
Since $\Pi^{*}=\Pi$, we obtain
\begin{equation*}
\trd_{D/\Qp}((\Pi^{*})^{-1}(y_2^*x_1-y_1^*x_2))
=-\trd_{D/\Qp}(\Pi^{-1}(y_1^*x_2-y_2^*x_1))=-((y_1,y_2),(x_1,x_2)), 
\end{equation*}
as desired. 

Second, we prove that $(\,,\,)$ is non-degenerate. Take $x=(x_1,x_2)\in V\setminus \{(0,0)\}$. We may assume $x_1\neq 0$. Then $y=(0,(x_1^{*})^{-1}\Pi)$ satisfies $(x,y)=2\neq 0$ by definition. 

We show the remaining two properties. The assertion (i) follows from the definition of $(\,,\,)$. For the assertion (ii), take $(x,y)\in V$. Then we have
\begin{equation*}
((\Pi^{i},0),(x,y))=\trd_{D/\Qp}(\Pi^{i-1}y),\quad ((0,\Pi^{i}),(x,y))=\trd_{D/\Qp}(-\Pi^{i-1}x)
\end{equation*}
for $i\in \{0,1\}$. Since
\begin{equation*}
\trd_{D/\Qp}(a+\Pi b)=\trd_{\Qps/\Qp}(a)
\end{equation*}
for $a,b\in \Qps$, we have $\trd_{D/\Qp}(\Pi^{i-1}y), \trd_{D/\Qp}(-\Pi^{i-1}x)\in \Zp$ for all $i\in \{0,1\}$ if and only if $x,y\in O_{D}$. Hence the assertion (ii) follows. 
\end{proof}

Let $G$ be the algebraic group over $\Qp$ defined by the tuple $(D,*,O_D,V,(\,,\,))$. We claim that $G$ is isomorphic to $\GU_2(D)$, which is a non-trivial inner form of $\GSp_4\otimes_{\Z}\Qp$ and splits over $\Qps$. This is pointed out in \cite[1.42]{rz}. However we give a proof here to use the argument in Section \ref{shqt}. 

We use the following lemma:
\begin{lem}\label{skhm}
\emph{Let $k$ be a field whose characteristic is not equal to $2$, $B$ a quaternion algebra over $k$. Suppose that there is a subfield $k_0$ of $k$ and a quaternion division algebra $B_0$ over $k_0$ such that $B=B_0\otimes_{k_0}k$. Let $e\in B_0^{\times}$ be an element such that $\overline{e}=-e$ and $e^2\in k_0^{\times}\setminus (k_0^{\times})^2$. Define an involution $*$ of $B$ by $b^{*}:=e\overline{b}e^{-1}$ for any $b\in B$. Moreover, let $V$ be a finite free left $B$-module, and $(\,,\,)$ a $k$-valued symplectic form on $V$ such that $(bx,y)=(x,b^{*}y)$ for any $b\in B$ and $x,y\in V$. Then there is a unique $B/k$-hermitian form $\langle \,,\, \rangle$ on $V$ such that $\trd_{B/k}(e^{-1}\langle x,y\rangle)=(x,y)$ for any $x,y\in V$. Moreover, if $R$ is a $k$-algebra, $(g,c)\in \GL_{B\otimes_{k}R}(V\otimes_{k}R)\times \G_{m}(R)$ and $x,y\in V\otimes_{k}R$. Then $(g(x),g(y))=c(x,y)$ if and only if $\langle g(x),g(y)\rangle=c\langle x,y\rangle$. }
\end{lem}

\begin{proof}
Since $e^2\in k_0^{\times}\setminus (k_0^{\times})^2$ and the characteristic of $k_0$ not equals $2$, the subalgebra $k_0(e)$ of $B_0$ is a separable quadratic field extension of $k_0$. Moreover, the involution $b\mapsto \overline{b}$ induces the non-trivial Galois automorphism of $k_0(e)/k_0$. Hence there is an element $d\in B_0^{\times}$ such that $d^2\in k_0^{\times}$, $de=-ed$ and $B_0=k_0[e,d]$ by \cite[8.12.2.~Theorem]{sch}. Hence we obtain $B=k[e,d]$. 

Fix an element $d$ as above. First, we prove the uniqueness. Assume that two $B/k$-hermitian forms $\langle \,,\,\rangle_{i}$ for $i\in \{1,2\}$ on $V$ satisfy $(x,y)=\trd_{B/k}(e^{-1}\langle x,y\rangle_i)$ for $x,y\in V$. Then we have
\begin{equation}\label{hfun}
\langle x,y\rangle_1-\langle x,y\rangle_2 \in k+kd+kde. 
\end{equation}
In particular, for fixed $x_0,y_0\in V\setminus \{0\}$, we obtain $\langle x_0,y_0\rangle_1-\langle x_0,y_0\rangle_2=a_0+a_1d+a_2ed$ for some $a_0,a_1,a_2\in k$. On the other hand, the assertion (\ref{hfun}) for $x=ex_0,dex_0$ and $dx_0$ imply $a_0=a_1=a_2=0$. Hence we obtain the equality of $B/k$-hermitian forms $\langle\,,\,\rangle_1=\langle \,,\,\rangle_2$. 

To prove the existence, define a $B$-valued bilinear form $\langle\,,\,\rangle$ on $V$ by
\begin{equation*}
\langle x,y\rangle:=\frac{1}{2}(ex,y)+\frac{e}{2}(x,y)+\frac{d}{2}(ed^{-1}x,y)-\frac{ed}{2}(d^{-1}x,y). 
\end{equation*}
Then $\langle \,,\,\rangle$ is a $B/k$-hermitian form satisfying the desired condition. The rest of the assertion follows from the definition of $\langle \,,\,\rangle$ as above. 
\end{proof}

\begin{cor}\label{unhm}
\emph{Keep the notation in Lemma \ref{skhm}, and put $n:=\rk_B V$. We further assume that one of the following hold: 
\begin{itemize}
\item $k$ is a non-archimedean local field, 
\item $k\cong \R$ and $B$ is split, 
\item $k$ is a number field, and $B$ is totally indefinite. 
\end{itemize}
Then the following hold: 
\begin{enumerate}
\item Let $V'$ be another finite free left $D$-module, and $(\,,\,)'$ a $k$-valued symplectic form on $V'$ such that $(bx,y)'=(x,b^{*}y)'$ for any $b\in B$ and $x,y\in V'$. Then there is a $B$-linear isometry of symplectic spaces between $(V,(\,,\,))$ and $(V',(\,,\,)')$ if and only if $\rk_B V'=n$. 
\item The algebraic group $G$ over $k$ defined by the formula (\ref{utsl}) for any $k$-algebra $R$ equals $\GU_n(B)$. In particular $G$ is isomorphic to $\GSp_{2n}\otimes_{\Z}k$ if $B$ is split. 
\end{enumerate}}
\end{cor}

\begin{proof}
First, if $k$ is a non-archimedean local field, then \cite[Theroem 25.5]{shi} (or its proof) and \cite[10.1.8 (ii)]{sch} imply that a hermitian form over $B$ of dimension $n$ is unique up to isometry. Therefore, (i) and (ii) follows from Lemma \ref{skhm}. Second, if $k\cong \R$ and $B$ is split, then (i) and (ii) follow from the same argument by using \cite[10.1.8 (i)]{sch}. Third, if $k$ is a number field, and $B$ is totally indefinite, then (i) and (ii) follows from the same argument by using \cite[10.1.8 (iii)]{sch}.
\end{proof}

Applying Corollary \ref{unhm} to $k=\Qp$, $B=D$, $V=D^{\oplus 2}$ and $e=\varepsilon$, We have $G=\GU_2(D)$. 

We now introduce the notion of polarized isocrystals, which is called \emph{$p$-polarized $D$-isocrystals} in \cite[Definition 2.1]{mie}. 

\begin{dfn}\label{pdic}
Let $k$ be an algebraically closed field of characteristic $p$. 
\begin{enumerate}
\item A \emph{polarized $D$-isocrystal over $k$} is a triple $(N,\Phi,(\,,\,))$, where
\begin{itemize}
\item $(N,\Phi)$ is an isocrystal over $k$, 
\item $(\,,\,)\colon N\times N\rightarrow \Frac(W(k))$ is a non-degenerate symplectic form over $\Frac(W(k))$,
\end{itemize}
such that the following conditions hold for any $x,y\in N$: 
\begin{itemize}
\item $(dx,y)=(x,d^{*}y)$ for any $d\in D$,
\item $(\Phi(x),\Phi(y))=p\sigma((x,y))$. 
\end{itemize}
\item Let $(N_i,\Phi_i,(\,,\,)_i)\,(i=1,2)$ be two polarized $D$-isocrystals. A \emph{morphism of polarized $D$-isocrystals from $(N_1,\Phi_1,(\,,\,)_1)$ to $(N_2,\Phi_2,(\,,\,)_2)$} is a morphism $f\colon (N_1,\Phi_1)\rightarrow (N_2,\Phi_2)$ of isocrystals which satisfies the following: 
\begin{itemize}
\item $f$ commutes with $D$-actions on $N_1$ and $N_2$,
\item there is $c\in \Qpt$ such that $(f(x),f(y))_2=c(x,y)_1$ for any $x,y\in N_1$. 
\end{itemize}
\end{enumerate}
\end{dfn}

Let
\begin{equation*}
b\colon V\otimes_{\Qp}K_0\rightarrow V\otimes_{\Qp}K_0\,;\,(d_1\otimes a_1,d_2\otimes a_2)\mapsto (d_1\Pi \otimes a_1,d_2\Pi \otimes a_2). 
\end{equation*}
Then we have $b\in G(K_0)$ and $\sml(b)=p$ by Lemma \ref{sdlt} (i). Let $F:=b\circ \sigma$ and $\D_{\Q}:=V\otimes_{\Qp}K_0$, and extend the $D$-action and $(\,,\,)$ to $V\otimes_{\Qp} K_0$ over $K_0$. Then $(\D_{\Q},F,(\,,\,))$ is a polarized $D$-isocrystal over $\Fpbar$. Moreover, $(\D_{\Q},F)$ is an isocrystal which is isoclinic of slope $1/2$, and $\D:=\Lambda^0\otimes_{\Zp}W$ is a $W$-lattice in $\D_{\Q}$ which is stable under $F$ and $pF^{-1}$. 

Let $\mu \colon \G_{m}\otimes_{\Z} K_0 \rightarrow G\otimes_{\Qp} K_0$ be the composite of the cocharacter of $\GSp_4$
\begin{equation*}
\G_{m}\otimes_{\Z}K_0\rightarrow \GSp_{4}\otimes_{\Z}K_0;z\mapsto \diag(z,z,1,1)
\end{equation*}
and an inner twist $\GSp_{4}\otimes_{\Z} K_0\cong \GU_2(D)\otimes_{\Qp} K_0$. Then $\mu$ satisfies the condition for weight decomposition. Finally, we can easily check that $\{\Pi^n\Lambda^0\}_{n\in \Z}$ is a self-dual multi-chain of $O_D$-lattices in $V$. Therefore we obtain a Rapoport--Zink datum 
\begin{equation*}
\mathcal{D}=(D,*,O_D,V,(\,,\,),b,\mu,\{\Pi^n\Lambda^0\}_{n\in \Z}). 
\end{equation*}

\subsection{Rapoport--Zink space for $G$}\label{rzsp}

We denote by $\nilp_{W}$ the category of $W$-schemes on which $p$ is locally nilpotent. 

For $S\in \nilp_{W}$, a \emph{$p$-divisible group with $G$-structure} over $S$ is a triple $(X,\iota,\lambda)$ consisting of the following data: 
\begin{itemize}
\item $X$ is a 4-dimensional $p$-divisible group over $S$, 
\item $\iota \colon O_D\rightarrow \End(X)$ is a ring homomorphism, 
\item $\lambda \colon X \rightarrow X^{\vee}$ is a polarization in the sense of \cite[3.20]{rz} (that is, a quasi-isogeny satisfying $\lambda^{\vee}=-\lambda$), 
\end{itemize}
such that the following conditions are fulfilled for any $d\in O_D$: 
\begin{itemize}
\item (Kottwitz condition) $\det(T-\iota(d) \mid \Lie(X))=(T^2-\trd_{D/\Qp}(d)T+\nrd_{D/\Qp}(d))^2$, 
\item $\lambda \circ \iota(d)=\iota(d^{*})^{\vee}\circ \lambda$. 
\end{itemize}

There is a $p$-divisible group with $G$-structure $(\X_0,\iota_0,\lambda_0)$ over $\Fpbar$ such that $\X_0$ is isoclinic of slope $1/2$ and $\lambda_0$ is an isomorphism. Indeed, we can construct such an object from the polarized $D$-isocrystal $(\D_{\Q},F,(\,,\,))$ and the self-dual lattice $\D$. See the argument after \cite[Definition 2.1]{mie}. 

Now we define $\M_G$, \emph{the Rapoport--Zink space for $G$}, to be the functor that parameterizes the equivalence classes of $(X,\iota,\lambda,\rho)$ for $S\in \nilp_{W}$, where $(X,\iota,\lambda)$ is a $p$-divisible group with $G$-structure over $S$, and 
\begin{equation*}
\rho \colon X\times_{S} \Sbar \rightarrow \X_0\times_{\spec \Fpbar} \Sbar
\end{equation*}
is an $O_D$-linear quasi-isogeny such that 
\begin{equation*}
\rho^{\vee}\circ \lambda_0 \circ \rho=c(\rho)\lambda
\end{equation*}
for some locally constant function $c(\rho) \colon \Sbar \rightarrow \Qp^{\times}$. Here $\Sbar$ is the closed subscheme of $S$ which is defined by $p\O_S$. 
Two tuples $(X_1,\iota_1,\lambda_1,\rho_1)$ and $(X_2,\iota_2,\lambda_2,\rho_2)$ are equivalent if $\rho_2^{-1} \circ \rho_1$ lifts to an isomorphism $X_1\rightarrow X_2$ over $S$. 

By \cite[Theorem 3.25]{rz}, the functor $\M_G$ is representable by a formal scheme over $\spf W$, which is formally locally of finite type. Moreover, we have $\M_G=\coprod_{i\in \Z}\M_G^{(i)}$, where $\M_G^{(i)}$ is the locus of $(X,\iota,\lambda,\rho)$ such that $c(\rho)(\Sbar)\subset p^i\Zp^{\times}$. 

\subsection{The algebraic group $J$}\label{grpj}

We define an algebraic group $J$ over $\Qp$ by
\begin{equation*}
J(R)=\{g\in G(K_0\otimes_{\Qp} R)\mid g\circ F=F\circ g\}
\end{equation*}
for any $\Qp$-algebra $R$. The representability of $J$ is obtained in \cite[Proposition 1.12]{rz}. We show that $J$ is isomorphic to $\GSp_4$ over $\Qp$. This is pointed out in \cite[1.42]{rz} for the $\Qp$-rational points. However, we recall a proof here since the argument for the assertion will also appear in Section \ref{flpf}. 

\begin{dfn}\label{y1pa}
We denote by $y_1$ the action $\iota_0({\Pi})$ on $\X_0$. We identify it as an action of $\Pi$ on $\D$ under the isomorphism $\D(\X_0)\cong \D$. 
\end{dfn}

\begin{dfn}\label{e0e1}
For $i\in \{0,1\}$, we define $\varepsilon_i,\D_i$ and $\D_{\Q,i}$ as below: 
\begin{gather*}
\varepsilon_i:=1\otimes \frac{1}{2}+(-1)^i\varepsilon \otimes \frac{1}{2\varepsilon}\in \Zps \otimes_{\Zp} W,\\
\D_i:=\varepsilon_i(\D)=\{x\in \Lambda^0\otimes_{\Zp}W \mid (a\otimes 1)x=(1\otimes \tau^i(a))x\text{ for all }a\in \Zps \}, \\
\D_{\Q,i}:=\D_i\otimes_W K_0=\varepsilon_i(V \otimes_{\Qp}K_0).
\end{gather*}
\end{dfn}
We have equalities 
\begin{equation*}
\D=\D_0\oplus \D_1,\quad \D_{\Q}=\D_{\Q,0}\oplus \D_{\Q,1},
\end{equation*}
and both $\D_{\Q,0}$ and $\D_{\Q,1}$ are totally isotropic (cf.~\cite[1.42]{rz}). Moreover, we have equalities 
\begin{gather*}
y_1(\D_i)\subset \D_{1-i},\quad y_1 (\D_{\Q,i})=\D_{\Q,1-i}, \\
F(\D_i)\subset \D_{1-i},\quad F(\D_{\Q,i})=\D_{\Q,1-i}. 
\end{gather*}

\begin{dfn}\label{rlbl}
Let $i\in \{0,1\}$. 
\begin{itemize}
\item Let $(\,,\,)_0$ be a $K_0$-valued bilinear form on $\D_{\Q,0}$ as follows for any $x,y\in \D_{\Q,0}$:
\begin{equation*}
(x,y)_0:=(x,y_1 (y)). 
\end{equation*}
\item Put $F_0:=y_1^{-1}\circ F$ on $\D_{\Q,0}$. 
\item Define $V_0$ a $\Qp$-vector space as the $F_0$-fixed part of $\D_{\Q,0}$. It is $4$-dimensional. 
\end{itemize}
\end{dfn}
The bilinear form $(\,,\,)_0$ is a non-degenerate symplectic form on $\D_{\Q,0}$ by Lemma \ref{sdlt}, and $(\D_{\Q,0},F_0)$ is isoclinic of slope $0$. Moreover, we have
\begin{equation*}
(F_0(x),F_0(y))_0=\sigma((x,y)_0)
\end{equation*}
for any $x,y\in \D_{\Q,0}$. See \cite[1.42]{rz}. Therefore the symplectic form $(\,,\,)_0$ induces a symplectic form on $V_0$, and there is an isomorphism $(V_0\otimes_{\Qp}K_0,\id \otimes \sigma)\cong (\D_{\Q,0},F_0)$ by the definition of $V_0$. 

We recall the map $\Phi$ defined in \cite[\S 2.2]{hp}. 
\begin{dfn}\label{phdf}
We define a $\sigma$-linear endomorphism $\Phi$ on $\End(\D_{\Q})$ as follows: 
\begin{equation*}
\Phi \colon \End(\D_{\Q})\rightarrow \End(\D_{\Q});f\mapsto F\circ f\circ F^{-1}. 
\end{equation*}
\end{dfn}
Then $(\End(\D_{\Q}),\Phi)$ is isoclinic of slope $0$. Moreover, there are isomorphisms 
\begin{equation*}
\End^0(\X_0)\cong \End(\D_{\Q})^{\Phi},\quad \End_{O_D}^0(\X_0)\cong \End_{D}(\D_{\Q})^{\Phi}. 
\end{equation*}
Here $\End(\D_{\Q})^{\Phi}$ and $\End_{D}(\D_{\Q})^{\Phi}$ are the $\Phi$-fixed parts of $\End(\D_{\Q})$ and $\End_{D}(\D_{\Q})$ respectively. On the other hand, put
\begin{equation*}
\Phi_0 \colon \End(\D_{\Q,0})\rightarrow \End(\D_{\Q,0});f\mapsto F_0\circ f\circ F_0^{-1}. 
\end{equation*}
Then we have an isomorphism 
\begin{equation*}
\End(\D_{\Q,0})^{\Phi_0}=\End(V_0)
\end{equation*}
by the definition of $V_0$. 

\begin{prop}\label{endd}
\emph{There are isomorphisms $\End_{D}(\D_{\Q})\cong \End(\D_{\Q,0})$ and $\End_{D}(\D_{\Q})^{\Phi}\cong \End(\D_{\Q,0})^{\Phi_0}$. }
\end{prop}

\begin{proof}
First, we construct a homomorphism $\End_{D}(\D_{\Q}) \rightarrow \End(\D_{\Q,0})$. Take $f\in \End_{D}(\D_{\Q})$. Since $f$ commutes with $\Qps$-action, we have $f(\D_{\Q,0})\subset \D_{\Q,0}$. Therefore we obtain the map 
\begin{equation*}
\phi \colon \End_{D}(\D_{\Q}) \rightarrow \End(\D_{\Q,0})^{\Phi_0};f\mapsto f\vert_{\D_{\Q,0}}. 
\end{equation*}
Moreover, if $f\in \End_{D}(\D_{\Q})^{\Phi}$, then $\phi(f)$ commutes with $F_0$ since it commutes with $y_1$ and $F$. 

Next, we construct the inverse morphism of $\phi$. Take $f_0\in \End(\D_{\Q,0})^{\Phi_0}$. Put
\begin{equation*}
\psi(f_0)\colon \D_{\Q}\rightarrow \D_{\Q};x+y_1(y)\mapsto f_0(x)+y_1(f_0(y))
\end{equation*}
for $x,y\in \D_{\Q,0}$. Then, $\psi(f_0)$ commutes with the $D$-action by definition. Therefore we obtain the map
\begin{equation*}
\phi \colon \End(\D_{\Q,0})\rightarrow \End_{D}(\D_{\Q}). 
\end{equation*}
Moreover, if $f_0\in \End(\D_{\Q,0})^{\Phi_0}$, then $\psi(f_0)$ commutes with $F$. By the definitions of $\phi$ and $\psi$, they are inverse to each other. Hence the assertion follows. 
\end{proof}

Let $R$ be a $\Qp$-algebra. Then we have equalities
\begin{gather*}
V\otimes_{\Qp}(K_0\otimes_{\Qp} R)=(\D_{\Q,0}\otimes_{K_0}R)\oplus (V_{K_0,1}\otimes_{K_0}R),\\
(V_0\otimes_{\Qp} R)\otimes_{\Qp} K_0=\D_{\Q,0}\otimes_{\Qp}R. 
\end{gather*}
Therefore we have
\begin{equation*}
V\otimes_{\Qp}(K_0\otimes_{\Qp} R)=((V_0\otimes_{\Qp} R)\otimes_{\Qp} K_0)\oplus y_1((V_0\otimes_{\Qp} R)\otimes_{\Qp} K_0). 
\end{equation*}

\begin{prop}\label{jgsp}
\emph{
\begin{enumerate}
\item The isomorphism $\End_{D}(\D_{\Q})\cong \End(\D_{\Q,0})$ in Proposition \ref{endd} induces an isomorphism $G(K_0\otimes_{\Qp}R)\cong \GSp(\D_{\Q,0})(K_0\otimes_{\Qp}R)$ for any $\Qp$-algebra $R$. 
\item The isomorphism $\End_{O_D}^0(\X_0)\cong \End_{D}(\D_{\Q})^{\Phi}\cong \End(\D_{\Q,0})^{\Phi_0} \cong \End(V_0)$ induced by Proposition \ref{endd} induces an isomorphism $J\cong \GSp(V_0)$ of algebraic groups over $\Qp$. 
\end{enumerate}}
\end{prop}

\begin{proof}
(i): Take a $\Qp$-algebra $R$ and $g\in (\End_{D}(\D_{\Q})\otimes_{\Qp} R)^{\times}\cong \GL_R(\D_{\Q,0}\otimes_{\Qp} R)$, where the isomorphism above is induced by the isomorphism in Proposition \ref{endd}. Then, for $c\in R^{\times}$, it suffices to show that we have $(g(x),g(y))=c(x,y)$ for $x\in \D_{\Q}$ if and only if $(g(x),g(y))_0=c(x,y)_0$ for $x\in \D_{\Q,0}$. However, this holds since $g$ commutes with $y_1$. 

(ii): This follows from the same argument for the proof of (i) by using the isomorphism 
\begin{equation*}
(\End_{O_D}(\X_0)\otimes_{\Qp} R)^{\times}\cong \GL_R(V_0\otimes_{\Qp} R)
\end{equation*}
instead of $(\End_{D}(\D_{\Q})\otimes_{\Qp} R)^{\times}\cong \GL_R(\D_{\Q,0}\otimes_{\Qp} R)$. 
\end{proof}

Next, we give a $\Qp$-basis of $V_0$. 
\begin{dfn}\label{efdf}
Define elements $\ebar_i\in \D_{\Q,0}\,(1\leq i\leq 4)$ and $\fbar_j\in \D_{\Q,1}\,(1\leq j\leq 4)$ as follow:
\begin{gather*}
\ebar_1:=(\varepsilon_0,0),\quad \ebar_2:=(0,\varepsilon_0),\quad
\ebar_3:=(\varepsilon_0\Pi,0),\quad \ebar_4:=(0,\varepsilon_0\Pi),\\
\fbar_1:=(\varepsilon_1,0),\quad \fbar_2:=(0,\varepsilon_1),\quad
\fbar_3:=(\varepsilon_1\Pi,0),\quad \fbar_4:=(0,\varepsilon_1\Pi). 
\end{gather*}
\end{dfn}
Then $\ebar_1,\ldots, \ebar_4$ form a $K_0$-basis of $\D_{\Q,0}$, and $\fbar_1,\ldots,\fbar_4$ form a $K_0$-basis of $\D_{\Q,1}$. 

\begin{lem}\label{eifj}
\emph{
\begin{enumerate}
\item We have equalities
\begin{equation*}
y_1(\ebar_1)=\fbar_3,\quad y_1(\ebar_2)=\fbar_4,\quad y_1(\ebar_3)=p\fbar_1,\quad y_1(\ebar_4)=p\fbar_2. 
\end{equation*}
\item For $1\leq i,j\leq 4$, we have $(\ebar_i,\fbar_j)=(-1)^{i-1}\delta_{i,5-j}$, where $\delta_{i,j}$ is Kronecker's delta. 
\end{enumerate}}
\end{lem}

\begin{proof}
(i): This follows from the definitions of $e_i$ and $f_j$. 

(ii): We only prove for $(i,j)=(1,4)$. Other cases are similar. We have
\begin{align*}
(\ebar_1,\fbar_4)=&\left((1,0)\otimes \frac{1}{2}+(\varepsilon,0)\otimes \frac{1}{2\varepsilon},
(0,\Pi)\otimes \frac{1}{2}-(0,\varepsilon \Pi)\otimes \frac{1}{2\varepsilon}\right)\\
=&\,\frac{1}{4}((1,0),(0,\Pi))-\frac{1}{4\varepsilon^2}((\varepsilon,0),(0,\varepsilon \Pi))
+\frac{1}{4\varepsilon}((\varepsilon,0),(0,\Pi))-\frac{1}{4\varepsilon}((1,0),(0,\varepsilon \Pi)). 
\end{align*}
On the other hand, we have $((1,0),(0,\Pi))=\trd(1)=2$ and thus $((\varepsilon,0),(0,\varepsilon \Pi))=-2\varepsilon^2$. Moreover, $((\varepsilon,0),(0,\Pi))=\trd(\varepsilon)=0$ and hence $((1,0),(0,\varepsilon \Pi))=0$. Therefore $(\ebar_1,\fbar_4)=1$. 
\end{proof}

\begin{dfn}\label{epdf}
Let us define elements $e'_i\in \D_{\Q,0}\,(1\leq i\leq 4)$ as follow:
\begin{equation*}
e'_1:=\ebar_1,\quad e'_2:=\ebar_3,\quad e'_3:=p^{-1}\ebar_4,\quad e'_4:=\ebar_2. 
\end{equation*}
\end{dfn}

Then we have $F_0(e'_i)=e'_i$ for $1\leq i\leq 4$, that is, $e'_i\in V_0$. Moreover, $e'_1,e'_2,e'_3,e'_4$ form a basis of $V_0$ whose Gram matrix of $(\,,\,)_0$ is $J_4$ by Lemma \ref{eifj}. The basis will be used in Sections \ref{excp} and \ref{jbts}.

\begin{dfn}\label{gpdf}
\begin{itemize}
\item We define a left-action of $J(\Qp)$ on $\M_G$ by
\begin{equation*}
J(\Qp)\times \M_G(S)\rightarrow \M_G(S);(g,(X,\iota,\lambda,\rho))\mapsto (X,\iota,\lambda,(g\times \id_{\Sbar}) \circ \rho)
\end{equation*}
for any $S\in \nilp_W$. 
\item Let $g_p\in J(\Qp)$ be the element corresponding to 
\begin{equation*}
{\begin{pmatrix}
&&&1\\
&&1&\\
&p&&\\
p&&&
\end{pmatrix}}
\end{equation*}
under the isomorphism $J(\Qp)\cong \GSp_4(\Qp)$ induced by the basis $e'_1,e'_2,e'_3,e'_4$. 
\end{itemize}
\end{dfn}

The following follow from the definition of $g_p$: 

\begin{prop}\label{rzis}
\emph{
\begin{enumerate}
\item We have $\sml(g_p)=-p$ and $g_p^2=p$. 
\item For any $i\in \Z$, the morphism
\begin{equation*}
g_p^i \colon \M_G^{(0)}\rightarrow \M_G^{(i)};\,x\mapsto g_p^i(x)
\end{equation*}
is an isomorphism. 
\item There is an isomorphism $p^{\Z}\backslash \M_G\cong \M_G^{(0)}\sqcup \M_G^{(1)}$. 
\end{enumerate}}
\end{prop}

\section{Rapoport--Zink space for $\GU_{2,2}$}\label{wkhp}

In this section, we recall and refine the method of \cite{hp}, which studies the Rapoport--Zink space for $\GU_{2,2}$, the unramified unitary similitude group of signature $(2,2)$ over $\Qp$ (with hyperspecial level structure). The results in this section will be used in Section \ref{btsn}. 

\subsection{Definition of the Rapoport--Zink space for $\GU_{2,2}$}

Here we regard $V$ as a $\Qps$-vector space by forgetting the action of $\Pi$. Let $H$ be an algebraic group over $\Qp$ defined by
\begin{equation*}
H(R)=\{(g,c)\in \GL_{\Qps\otimes_{\Qp} R}(V\otimes_{\Qp} R)\times \G_m(R)\mid (g(v),g(w))=c(v,w)\text{ for all }v,w\in V\otimes_{\Qp} R\}
\end{equation*}
for each $\Qp$-algebra $R$. Then we have a natural closed immersion
\begin{equation*}
\varphi \colon G\hookrightarrow H. 
\end{equation*}
In the following, we identify $G$ as a closed subgroup of $H$ by $\varphi$. 

For $S\in \nilp_{W}$, a \emph{$p$-divisible group with $H$-structure} over $S$ is a triple $(X,\iota,\lambda)$ consisting of the following data: 
\begin{itemize}
\item $X$ is a 4-dimensional $p$-divisible group over $S$, 
\item $\iota \colon \Zps \rightarrow \End(X)$ is a ring homomorphism, 
\item $\lambda \colon X \rightarrow X^{\vee}$ is a polarization in the sense of \cite[3.20]{rz}, 
\end{itemize}
such that the following conditions are fulfilled for any $a\in \Zps$: 
\begin{itemize}
\item (Kottwitz condition) $\det(T-\iota(a) \mid \Lie(X))=(T^2-\tr_{\Qps/\Qp}(a)T+\N_{\Qps/\Qp}(a))^2$, 
\item $\lambda \circ \iota(a)=\iota(\tau(a))^{\vee}\circ \lambda$. 
\end{itemize}

\begin{ex}\label{gugu}
Let $(\X_0,\iota_0,\lambda_0)$ be the $p$-divisible group with $G$-structure fixed in Section \ref{rzdt}. Then $(\X_0,\iota_0\vert_{\Zps},\lambda_0)$ is a $p$-divisible group with $H$-structure. Note that $\tau=*$ on $\Qps \subset D$. 
\end{ex}

In this paper, as a framing object we use a $p$-divisible group with $H$-structure which follows from a $p$-divisible group with $G$-structure as in Example \ref{gugu}. We define $\M_{H}$, the \emph{Rapoport--Zink space for $H$}, as the functor that parameterizes the equivalence classes of $(X,\iota,\lambda,\rho)$ for $S\in \nilp_{W}$, where $(X,\iota,\lambda)$ is a $p$-divisible group with $H$-structure over $S$, and 
\begin{equation*}
\rho \colon X\times_{S}\overline{S} \rightarrow \X_0\times_{\spec \Fpbar}\overline{S}
\end{equation*}
is a $\Zps$-linear quasi-isogeny such that 
\begin{equation*}
\rho^{\vee}\circ \lambda_0 \circ \rho=c(\rho)\lambda
\end{equation*}
for some locally constant function $c\colon \overline{S}\rightarrow \Qp^{\times}$. 
Two $p$-divisible groups with $H$-structures $(X_1,\iota_1,\lambda_1,\rho_1)$ and $(X_2,\iota_2,\lambda_2,\rho_2)$ are equivalent if $\rho_2^{-1}\circ \rho_1$ lifts to an isomorphism $X_1\rightarrow X_2$ over $S$. 

By the same reason as for $\M_G$, the functor $\M_{H}$ is representable by a formal scheme over $\spf W$, which is formally locally of finite type. Moreover, we have $\M_{H}=\coprod_{i\in \Z}\M_{H}^{(i)}$, where $\M_H^{(i)}$ is the locus of $(X,\iota,\lambda,\rho)$ such that $c(\rho)(\overline{S})\subset p^i\Zp^{\times}$. 

\begin{rem}
\begin{enumerate}
\item Let $(X,\iota,\lambda,\rho)\in \M_{H}(S)$ where $S\in \nilp_{W}$, and denote by $h$ the height of $\lambda$. Then $p^{-h}\lambda$ is a principal polarization on $X$. Moreover, we have $(X,\iota,p^{-h}\lambda,\rho)\in \M_{H}(S)$ which is equivalent to $(X,\iota,\lambda,\rho)$. Hence $\M_{H}$ coincides with the Rapoport--Zink space considered in \cite[\S 2]{hp}. 
\item The formal scheme $\M_H$ is attached to the Rapoport--Zink datum $(\Qps,\tau,\Zps,V,(\,,\,),\varphi(b),\varphi \circ \mu,\{p^n\Lambda^0\}_{n\in \Z})$.
\end{enumerate}
\end{rem}

\begin{prop}\label{mgeb}
\emph{The morphism 
\begin{equation*}
i_{G,H}\colon\M_G \rightarrow \M_{H};\,(X,\iota,\lambda,\rho)\mapsto (X,{\iota}\vert_{\Zps},\lambda,\rho)
\end{equation*}
is a closed immersion. A tuple $(X,\iota,\lambda,\rho)$ in $\M_{H}$ lies in the image of $i$ if and only if ${\rho}^{-1}\circ y_1\circ \rho \in \End(X)$. }
\end{prop}

\begin{proof}
First, we show that $i$ is injective. Suppose that two $p$-divisible groups with $G$-structures $(X_1,\iota_1,\lambda_1,\rho_1)$ and $(X_2,\iota_2,\lambda_2,\rho_2)$ are equivalent as $p$-divisible groups with $H$-structures. Let $\widetilde{\rho}\colon X_1\rightarrow X_2$ be an isomorphism which lifts $\rho_2\circ \rho_1^{-1}$. Then $\widetilde{\rho}$ commutes with the actions of $y_1$ by the rigidity of quasi-isogenies. See the assertion after \cite[Definition 2.8]{rz}. Thus $\widetilde{\rho}$ gives an equivalence between $(X_1,\iota_1,\lambda_1,\rho_1)$ and $(X_2,\iota_2,\lambda_2,\rho_2)$ as $p$-divisible groups with $G$-structures. 

Next, we show the second assertion. It is clear that $(X,\iota,\lambda,\rho)$ in $\M_H$ which lies in the image of $i$ satisfies the property ${\rho}^{-1}\circ y_1\circ \rho \in \End(X)$. On the other hand, for $(X,\iota,\lambda,\rho)$ in $\M_H$ such that ${\rho}^{-1}\circ y_1\circ \rho \in \End(X)$, we can define an action of $\Pi$ on $X$ by ${\rho}^{-1}\circ y_1\circ \rho$. Then it is a $p$-divisible group with $G$-structure by the rigidity of quasi-isogenies, and its image under $i$ is identical to $(X,\iota,\lambda,\rho)$. 

Finally, by the description of the image above and \cite[Proposition 2.9]{rz}, we obtain that $i$ is a closed immersion. 
\end{proof}

\begin{rem}
Some analogues of Proposition \ref{mgeb} for unramified unitary similitude groups of signature $(1,n-1)$ appears in \cite[\S 10]{Li2019} and \cite[\S 10]{rsz}, that pursue the theory of arithmetic intersection. They involve the Rapoport--Zink spaces $\M_{1,n-1}^{(i)}$ with fixed objects $(\X_{1,n-1},\iota,\lambda^{(i)})$ for $i\in \{0,1\}$, where
\begin{itemize}
\item $\X_{1,n-1}$ is a $p$-divisible group of dimension $n-1$ which is isoclinic of slope $1/2$,
\item $\iota \colon \Zps \rightarrow \End(\X_{1,n-1})$ is a ring homomorphism,
\item $\lambda^{(i)} \colon \X_{1,n-1}\rightarrow \X_{1,n-1}^{\vee}$ is a polarization,
\end{itemize}
satisfying the following: 
\begin{itemize}
\item $\det(T-\iota(a))=(T-a)(T-\overline{a})^{n-1}$ for $a\in \Zps$,
\item $\lambda^{(i)}\circ \iota(a)=\iota(\overline{a})\circ \lambda^{(i)}$ for $a\in \Zps$,
\item $\Ker(\lambda^{(i)})\subset \X_{1,n-1}[p]$ and has height $2i$. 
\end{itemize}
More precisely, \cite[\S 10]{Li2019} uses a closed immersion $\M_{1,n-1}^{(1)}\hookrightarrow \M_{1,n}^{(0)}$ to prove the Kudla--Rapoport conjecture for $\M_{1,n-1}^{(1)}$. On the other hand, \cite[\S 10]{rsz} considers a closed immersion $\M_{1,n-1}^{(0)}\hookrightarrow \M_{1,n}^{(1)}$. 
\end{rem}

\subsection{$\pi$-special quasi-endomorphisms}

In this section, we construct a quadratic subspace $\L_{\Q}^{\Phi,\pi}$ in $\End^0(\X_0)$ with a quadratic form $f\mapsto f\circ f$ over $\Qp$, whose elements are called \emph{$\pi$-special quasi-endomorphisms}. We realize it as a subspace of the quadratic space $\L_{\Q}^{\Phi}$ defined in \cite[\S 2.2]{hp}. We use such a space to describe combinatorial properties of the underlying space of $\M_G$, which is similar to the case for $\M_H$ in \cite{hp}. 

Let $(\D_{\Q},F,(\,,\,))$ be the polarized $D$-isocrystal introduced in Definition \ref{pdic}. We define a $\Qps \otimes_{\Qp} K_0$-valued hermitian form $\langle \,,\,\rangle$ on $\D_{\Q}$ by
\begin{equation*}
\langle x,y\rangle:=\frac{1}{2}\otimes (\iota_0(\varepsilon)x,y)+\frac{\varepsilon}{2}\otimes (x,y)
\end{equation*}
for $x,y\in \D_{\Q}$. This is a unique $\Qps/\Qp$-hermitian form such that $\tr_{\Qps/\Qp}((\varepsilon^{-1}\otimes 1)\langle\,,\,\rangle)=(\,,\,)$. We have the following by definition: 

\begin{lem}\label{fhmr}
\emph{For $x,y\in \D_{\Q}$, we have $\langle F(x),F(y)\rangle=p\sigma(\langle x,y\rangle)$. }
\end{lem}

From now on, let us fix an element $\zeta \in \Zps$ with $\N_{\Qps/\Qp}(\zeta)=-1$. 

\begin{lem}\label{base}
\emph{There is a $W$-basis $e_i,f_j$ ($1\leq i,j\leq 4$) of $\D$ with $e_i\in \D_{0}$ and $f_j\in \D_{1}$, such that  
\begin{gather*}
\langle e_i,f_j\rangle=\delta_{ij}\cdot \varepsilon_0, \\
F(e_1)=f_1,\quad F(e_2)=f_2,\quad F(e_3)=pf_3,\quad F(e_4)=pf_4,\\
F(f_1)=pe_1,\quad F(f_2)=pe_2,\quad F(f_3)=e_3,\quad F(f_4)=e_4,
\end{gather*}
and
\begin{gather*}
y_1 (e_1)=-\zeta^{-1}f_2,\quad y_1 (e_2)=\zeta^{-1}f_1,\quad
y_1 (e_3)=p\zeta f_4,\quad y_1 (e_4)=-p\zeta f_3,\\
y_1 (f_1)=p\zeta e_2,\quad y_1 (f_2)=-p\zeta e_1,\quad
y_1 (f_3)=-\zeta^{-1}e_4,\quad y_1 (f_4)=\zeta^{-1}e_3. 
\end{gather*}
Here $\delta_{ij}$ is Kronecker's delta, and $\varepsilon_0\in \Qps \otimes_{\Qp} K_0$ is the element defined in Definition \ref{e0e1}. }
\end{lem}

\begin{proof}
Fix $a,b\in \Zps \subset K_0$ such that $a\sigma(b)-\sigma(a)b=\varepsilon^{-1}$ (for example, put $a:=1$ and $b:=-\varepsilon^{-1}/2$). Let us define $e_i$ and $f_j$ as below: 
\begin{align*}
e_1&:=a\ebar_1+b\ebar_2,&
e_2&:=\zeta^{-1}\sigma(a)\ebar_1+\zeta^{-1}\sigma(b)\ebar_2,&
e_3&:=a\ebar_3+b\ebar_4,&
e_4&:=-\zeta \sigma(a)\ebar_3-\zeta \sigma(b)\ebar_4,\\
f_1&:=\sigma(a)\fbar_3+\sigma(b)\fbar_4,&
f_2&:=-\zeta a\fbar_3-\zeta b\fbar_4,&
f_3&:=\sigma(a)\fbar_1+\sigma(b)\fbar_2,&
f_4&:=\zeta^{-1} a\fbar_1+\zeta^{-1}b\fbar_2. 
\end{align*}
By Lemmas \ref{eifj} and \ref{fhmr}, the elements above satisfy the desired conditions. 
\end{proof}

Let $\bigwedge^2_{\Qps} \D_{\Q}$ be the second exterior power of $\D_{\Q}$ as a $\Qps \otimes K_0$-module. We regard it as a $K_0$-subspace of $\End(\D_{\Q})$ by
\begin{equation*}
(x\wedge y)(z)=\langle x,z\rangle y-\langle y,z\rangle x
\end{equation*}
for $x,y\in \D_{\Q}$. Moreover, we introduce a hermitian form $\langle \,,\,\rangle$ on $\bigwedge^2_{\Qps}\D_{\Q}$ by 
\begin{equation*}
\langle v_1\wedge v_2,w_1\wedge w_2\rangle:=\langle v_1,w_1\rangle \langle v_2,w_2\rangle-\langle v_1,w_2\rangle \langle v_2,w_1\rangle
\end{equation*}
for any $v_1,v_2,w_1,w_2\in \D_{\Q}$. 

We give two lemmas on this subspace. 

\begin{lem}\label{wdhm}
\emph{Let $v\in \bigwedge^2_{\Qps}\D_{\Q}$, and regard it as an element of $\End(\D_{\Q})$ by the injection above. Then, for $x,y\in \D_{\Q}$ we have 
\begin{equation*}
\langle x,v(y)\rangle=-{\langle v(x),y\rangle}^{*}. 
\end{equation*}}
\end{lem}

\begin{proof}
This is pointed out in the proof of \cite[Proposition 2.8]{hp}. However, we make a proof since we use this assertion in Lemma \ref{pipr} below. 

We may assume $v=x_0\wedge y_0$ for some $x_0,y_0\in \D_{\Q}$. 
In this case, we have 
\begin{equation*}
\langle x,v(y)\rangle=\langle x,\langle x_0,y\rangle y_0-\langle y_0,y\rangle x_0 \rangle=\langle x,y_0\rangle \langle x_0,y\rangle^{*}-\langle x,x_0\rangle \langle y_0,y\rangle^{*}. 
\end{equation*}
On the other hand we obtain
\begin{equation*}
-{\langle v(x),y\rangle}^{*}=-\langle \langle x_0,x\rangle y_0-\langle y_0,x\rangle x_0,y \rangle^{*}=\langle x,y_0\rangle \langle x_0,y\rangle^{*}-\langle x,x_0\rangle \langle y_0,y\rangle^{*}, 
\end{equation*}
which concludes the assertion. 
\end{proof}

\begin{lem}\label{piwd}
\emph{We have
\begin{equation*}
y_1=p\zeta e_1\wedge e_2-\zeta^{-1}e_3\wedge e_4-\zeta^{-1}f_1\wedge f_2+p\zeta f_3\wedge f_4
\end{equation*}
as elements in $\End(\D_{\Q})$. }
\end{lem}

\begin{proof}
It suffices to show the equality on $e_i$ and $f_j$, since they form a $K_0$-basis of $\D_{\Q}$. We only prove it for $f_1$. Other cases are similar. 

For the right-hand side, we have 
\begin{equation*}
(p\zeta e_1\wedge e_2-\zeta^{-1}e_3\wedge e_4-\zeta^{-1}f_1\wedge f_2+p\zeta f_3\wedge f_4)(f_1)=(p\zeta e_1\wedge e_2)(f_1)=p\zeta \langle e_1,f_1\rangle e_2. 
\end{equation*}
Moreover, we have
\begin{equation*}
p\zeta \langle e_1,f_1\rangle e_2=p\zeta \varepsilon_0e_2=p\zeta e_2. 
\end{equation*}
by the property on $\langle e_i,f_j\rangle$ in Lemma \ref{base} and $e_2\in \D_{\Q,0}$. On the other hand, we have $y_1(f_1)=p\zeta e_2$ by the last property in Lemma \ref{base}. Hence the equality for $f_1$ follows. 
\end{proof}

Let us recall the Hodge star operator on $\bigwedge^2_{\Qps} \D_{\Q}$ defined in \cite[\S 2.2]{hp}. Put
\begin{equation*}
\omega:=e_1\wedge e_2\wedge e_3\wedge e_4+f_1\wedge f_2\wedge f_3\wedge f_4\in \bigwedge^4_{\Qps} \D_{\Q}. 
\end{equation*}
Then the Hodge star operator $\star$ is the $K_0$-linear map $\bigwedge^2_{\Qps} \D_{\Q}\rightarrow \bigwedge^2_{\Qps} \D_{\Q}$ such that for $x\in \bigwedge^2_{\Qps} \D_{\Q}$, $x^{\star}$ satisfies $y\wedge x^{\star}=\langle y,x\rangle \omega$ for any $y\in \bigwedge^2_{\Qps} \D_{\Q}$. 

\begin{rem}
In \cite[\S 2.2]{hp}, they also define the Hodge star operator with respect to $\alpha \omega$ for any $\alpha \in \Qps$ with $\N_{\Qps/\Qp}(\alpha)=1$. However we use it only for $\alpha=1$. 
\end{rem}

Let us recall the $\sigma$-linear endomorphism $\Phi$ on $\End(\D_{\Q})$ introduced in Definition \ref{phdf}. It is defined by $f\mapsto F\circ f\circ F^{-1}$. We study relations between $\Phi$ and $\bigwedge^2_{\Qps} \D_{\Q}$. 

\begin{lem}\label{phst}
\emph{
\begin{enumerate}
\item For $x,y\in \D_{\Q}$, we have $\Phi(x\wedge y)=p^{-1}F(x)\wedge F(y)$. In particular, $\bigwedge^2_{\Qps} \D_{\Q}$ is stable under $\Phi$, that is, $(\bigwedge^2_{\Qps} \D_{\Q},\Phi)$ is a subisocrystal of $(\End(\D_{\Q}),\Phi)$. 
\item For $v,w\in \bigwedge^2_{\Qps} \D_{\Q}$, we have $\langle \Phi(v),\Phi(w)\rangle=\sigma(\langle v,w\rangle)$. 
\item For $v,w\in \bigwedge^2_{\Qps}\D_{\Q}$, put $\Phi(v\wedge w):=\Phi(v)\wedge \Phi(w)$. Then we have $\Phi(\omega)=\omega$. 
\item The map $\Phi$ commutes with $\star$. 
\end{enumerate}}
\end{lem}

\begin{proof}
Note that the assertion (i) and (iii) are pointed out after \cite[Proposition 2.4]{hp}. 

(i): Take $z\in \D_{\Q}$. Then we have
\begin{equation*}
\Phi(x\wedge y)(z)=F(\langle x,F^{-1}(z)\rangle y-\langle y,F^{-1}(z)\rangle x)
=p^{-1}F(\sigma^{-1}(\langle F(x),z\rangle) y-\sigma^{-1}(\langle F(y),z\rangle) x). 
\end{equation*}
By Lemma \ref{fhmr}, we obtain
\begin{align*}
F(\langle x,F^{-1}(z)\rangle y-\langle y,F^{-1}(z)\rangle x)
=&\,p^{-1}F(\sigma^{-1}(\langle F(x),z\rangle) y-\sigma^{-1}(\langle F(y),z\rangle) x)\\
=&\,p^{-1}(\langle F(x),z\rangle F(y)-\langle F(y),z\rangle F(x))\\
=&\,p^{-1}\langle F(x),F(y)\rangle (z), 
\end{align*}
which implies the desired equality. 

(ii): We may assume $v=x_1\wedge x_2$ and $w=y_1\wedge y_2$ for some $x_1,v_2,w_1,w_2\in \D_{\Q}$. By (i) we have
\begin{align*}
\langle \Phi(x_1\wedge x_2),\Phi(y_1\wedge y_2)\rangle=& \,p^{-2}\langle F(x_1)\wedge F(x_2),F(y_1)\wedge F(y_2)\rangle \\
=&\, p^{-2}\langle F(x_1),F(y_1)\rangle \langle F(x_2),F(y_2)\rangle-p^{-2}\langle F(x_1),F(y_2)\rangle \langle F(x_2),F(y_1)\rangle. 
\end{align*}
On the other hand, we have 
\begin{gather*}
p^{-2}\langle F(x_1),F(y_1)\rangle \langle F(x_2),F(y_2)\rangle-p^{-2}\langle F(x_1),F(y_2)\rangle \langle F(x_2),F(y_1)\rangle \\
=\sigma(\langle x_1,y_1\rangle) \sigma(\langle x_2,y_2\rangle)-\sigma(\langle x_1,y_2\rangle) \sigma(\langle x_2,y_1\rangle)\\
=\sigma(\langle x_1\wedge x_2,y_1\wedge y_2\rangle).  
\end{gather*}
Therefore the assertion follows. 

(iii): By (i), we have 
\begin{equation*}
\Phi(\omega)=p^{-2}(F(e_1)\wedge F(e_2)\wedge F(e_3)\wedge F(e_4)+F(f_1)\wedge F(f_2)\wedge F(f_3)\wedge F(f_4)). 
\end{equation*}
Using the second property of $e_i$ and $f_j$ in Lemma \ref{base}, we obtain
\begin{equation*}
p^{-2}(F(e_1)\wedge F(e_2)\wedge F(e_3)\wedge F(e_4)+F(f_1)\wedge F(f_2)\wedge F(f_3)\wedge F(f_4))=\omega, 
\end{equation*}
which implies the assertion. 

(iv): Take $v,w\in \bigwedge^2_{\Qps} \D_{\Q}$. By (ii), we have
\begin{equation*}
\Phi(w)\wedge \Phi(v)^{\star}=\langle \Phi(w),\Phi(v)\rangle \omega=\sigma(\langle w,v\rangle)\omega. 
\end{equation*}
On the other hand, we have
\begin{equation*}
\Phi(w)\wedge \Phi(v^{\star})=\Phi(w\wedge v^{\star})=\Phi(\langle w,v\rangle \omega)=\sigma(\langle w,v\rangle)\omega
\end{equation*}
by (iii) and the definition of $\star$. Hence we have $\Phi(v)^{\star}=\Phi(v^{\star})$. 
\end{proof}

Now we define a map $\pi$ as follows: 
\begin{equation*}
\pi \colon \End(\D_{\Q})\rightarrow \End(\D_{\Q});\,h\mapsto \iota_0(\varepsilon \Pi)\circ h\circ \iota_0(\varepsilon \Pi)^{-1}
\end{equation*}

We have $\pi^2=\id_{\End(\D_{\Q})}$ by definition. 

\begin{lem}\label{pipr}
\emph{
\begin{enumerate}
\item For any $v\in \bigwedge^2_{\Qps} \D_{\Q}$, we have $\pi(v)= -y_1^{-1} \circ v\circ y_1$. 
\item For $x\wedge y\in \bigwedge^2_{\Qps} \D_{\Q}$ with $x,y\in \D_{\Q}$, we have $\pi(x\wedge y)=p^{-1}y_1 (x)\wedge y_1 (y)$. In particular, $\bigwedge^2_{\Qps} \D_{\Q}$ is stable under $\pi$.
\item For $v,w\in \bigwedge^2_{\Qps}\D_{\Q}$, we have $\langle \pi(v),\pi(w)\rangle=\langle v,w\rangle^{*}$. 
\item For $v,w\in \bigwedge^2_{\Qps}\D_{\Q}$, put $\pi(v\wedge w):=\pi(v)\wedge \pi(w)$. Then, for $a\in \Qps$ we have $\pi(\iota_0(a)\omega)=\iota_0(a^{*})\omega$. 
\item The map $\pi$ commutes with $\Phi$ and $\star$. 
\end{enumerate}}
\end{lem}

\begin{proof}
(i): This follows from the equality $\iota_0(\varepsilon)\circ v=-v \circ \iota_0(\varepsilon)$. 

(ii): Take $z\in \D_{\Q}$. By (i) and Lemmas \ref{piwd}, \ref{wdhm}, we have
\begin{equation*}
\pi(x\wedge y)(z)=-y_1^{-1}(\langle x,y_1 (z)\rangle y-\langle y,y_1 (z)\rangle x)
=y_1^{-1} ({\langle y_1 (x),z\rangle}^{*} y-{\langle y_1 (y),z\rangle}^{*}x). 
\end{equation*}
Since $\Pi a=a^{*}\Pi$ for any $a\in \Qps$, we obtain
\begin{align*}
y_1^{-1} ({\langle y_1 (x),z\rangle}^{*} y-{\langle y_1 (y),z\rangle}^{*}x)
=&p^{-1}(\langle y_1 (x),z\rangle y_1 (y)-\langle y_1 (y),z\rangle y_1 (x))\\
=&p^{-1}(y_1 (x)\wedge y_1 (y))(z), 
\end{align*}
which implies the desired equality. 

(iii): We may assume $v=x\wedge x'$ and $w=y\wedge y'$ for some $x,x',y,y'\in \D_{\Q}$. By (ii) we have
\begin{align*}
\langle \pi(x\wedge x'),\pi(y\wedge y')\rangle
=&\,p^{-2}\langle y_1 (x)\wedge y_1 (x'),y_1(y)\wedge y_1(y')\rangle \\
=&\,p^{-2}\langle y_1 (x),y_1 (y)\rangle \langle y_1 (x'),y_1 (y')\rangle-p^{-2}\langle y_1 (x),y_1 (y')\rangle \langle y_1 (x'),y_1 (y)\rangle. 
\end{align*}
Note that we can apply Lemma \ref{wdhm} by Lemma \ref{piwd}. Hence we have
\begin{gather*}
p^{-2}\langle y_1 (x),y_1 (y)\rangle \langle y_1 (x'),y_1 (y')\rangle-p^{-2}\langle y_1 (x),y_1 (y')\rangle \langle y_1 (x'),y_1 (y)\rangle \\
=\langle x,y\rangle^{*} \langle x',y'\rangle^{*}-\langle x,y'\rangle^{*} \langle x',y\rangle^{*}\\
=\langle x\wedge x',y\wedge y'\rangle^{*}. 
\end{gather*}
Therefore the assertion follows. 

(iv): By (ii), we have
\begin{equation*}
\pi(\iota_0(a)\omega)=p^{-2}(ay_1(e_1)\wedge y_1(e_2)\wedge y_1(e_3)\wedge y_1(e_4)+\sigma(a)y_1(f_1)\wedge y_1(f_2)\wedge y_1(f_3)\wedge y_1(f_4)). 
\end{equation*}
Using the last property of $e_i$ and $f_j$ in Lemma \ref{base}, we obtain
\begin{gather*}
p^{-2}(ay_1(e_1)\wedge y_1 (e_2)\wedge y_1(e_3)\wedge y_1(e_4)+\sigma(a)y_1 (f_1)\wedge y_1 (f_2)\wedge y_1 (f_3)\wedge y_1 (f_4))\\
=\sigma(a)e_1\wedge e_2\wedge e_3\wedge e_4+af_1\wedge f_2\wedge f_3\wedge f_4. 
\end{gather*}
On the other hand, we have
\begin{equation*}
\iota_0(a^{*})\omega=\sigma(a)e_1\wedge e_2\wedge e_3\wedge e_4+af_1\wedge f_2\wedge f_3\wedge f_4, 
\end{equation*}
since $\sigma(a)=a^{*}$. Hence the assertion follows. 

(v): The first assertion follows from the commutativity of $y_1$ and $F$. For the second assertion, take any $v,w\in \bigwedge^2_{\Qps}\D_{\Q}$. By (iii) we have
\begin{equation*}
\pi(w)\wedge \pi(v)^{\star}=\langle \pi(w),\pi(v)\rangle \omega=\langle w,v\rangle^{*} \omega=\langle w,v\rangle^{*} \omega. 
\end{equation*}
On the other hand, by (iv) we have
\begin{equation*}
\pi(w)\wedge \pi(v^{\star})={\pi}(w\wedge v^{\star})={\pi}(\langle w,v\rangle \omega)={\langle w,v\rangle}^{*} \omega. 
\end{equation*}
Thus the assertion follows. 
\end{proof}

We set
\begin{equation*}
\L_{\Q}:=\{v\in \bigwedge^2_{\Qps} \D_{\Q} \mid v^{\star}=v\}, 
\end{equation*}
endowed with a quadratic form over $K_0$ by $Q(v):=v\circ v$ (see \cite[Proposition 2.4 (2)]{hp} for the map $Q$ being $K_0$-valued). Define the associated symmetric bilinear form $[\,,\,]$ on $\L_{\Q}$ by
\begin{equation*}
[v,w]=Q(v+w)-Q(v)-Q(w). 
\end{equation*}
By Lemma \ref{phst} (iv), $(\L_{\Q},\Phi)$ becomes an isocrystal of slope $0$. Hence we can consider a $\Qp$-vector space $\L_{\Q}^{\Phi}$, the $\Phi$-fixed part of $\L_{\Q}$. 
The structure of $\L_{\Q}^{\Phi}$ is determined in \cite[Proposition 2.6]{hp}. Here we refine their assertion. 

\begin{dfn}\label{xidf}(cf.~the proof of \cite[Proposition 2.6]{hp})
We define $x_i\,(1\leq i\leq 6)$ in $\L_{\Q}$ as follow: 
\begin{align*}
x_1&:=e_1\wedge e_2+f_3\wedge f_4,\,&x_2&:=e_3\wedge e_4+f_1\wedge f_2,\\
x_3&:=e_1\wedge e_3+f_4\wedge f_2,\,&x_4&:=e_4\wedge e_2+f_1\wedge f_3,\\
x_5&:=e_1\wedge e_4+f_2\wedge f_3,\,&x_6&:=e_2\wedge e_3+f_1\wedge f_4. 
\end{align*}
\end{dfn}
Then we have
\begin{equation*}
\left([x_i,x_j]\right)_{1\leq i,j\leq 6}={\begin{pmatrix}
0&-1&&&&\\
-1&0&&&&\\
&&0&-1&&\\
&&-1&0&&\\
&&&&0&-1\\
&&&&-1&0
\end{pmatrix}}. 
\end{equation*}
Next, we (re-)define the elements $y_i\,(1\leq i\leq 6)$ in $\L_{\Q}$ by
\begin{align*}
y_1&:=p\zeta x_1-\zeta^{-1}x_2=\iota_0({\Pi}),&
y_2&:=\varepsilon (p\zeta x_1+\zeta^{-1}x_2),\\
y_3&:=x_3+x_4,&y_4&:=\varepsilon(x_3-x_4),\\
y_5&:=x_5+x_6,&y_6&:=\varepsilon(x_5-x_6).
\end{align*}
Then the Gram matrix is given by
\begin{equation*}
\left(\frac{[y_i,y_j]}{2}\right)_{1\leq i,j\leq 6}={\begin{pmatrix}
p&&&&&\\
&-\varepsilon^2p&&&&\\
&&-1&&&\\
&&&\varepsilon^2&&\\
&&&&-1&\\
&&&&&\varepsilon^2
\end{pmatrix}}
\end{equation*}
and hence $y_1,\ldots ,y_6$ form an orthogonal basis of $\L_{\Q}^{\Phi}$. Put 
\begin{gather*}
\L_{\Q}^{\Phi,\pi}:=\{v\in \L_{\Q}^{\Phi}\mid \pi(v)=v\}, \\
\L_{\Q}^{\Phi,-\pi}:=\{v\in \L_{\Q}^{\Phi}\mid \pi(v)=-v\}. 
\end{gather*}
Then $\pi(y_i)=y_i$ for $2\leq i\leq 6$ by the last property in Lemma \ref{base}. On the other hand, we have $\pi(y_1)=-y_1$ by definition. Therefore we have $\L_{\Q}^{\Phi,\pi}=\bigoplus^{6}_{i=2}y_i\Qp$ and $\L_{\Q}^{\Phi,-\pi}=y_1\Qp$. 
Consequently, we obtain the following: 

\begin{prop}\label{lqpi}
\emph{
\begin{enumerate}
\item We have an orthogonal decomposition $\L_{\Q}^{\Phi}=\L_{\Q}^{\Phi,\pi}\oplus \L_{\Q}^{\Phi,-\pi}$. Moreover, $\L_{\Q}^{\Phi,-\pi}=y_1\Qp$ and it is the unique orthogonal complement of $\L_{\Q}^{\Phi,\pi}$ in $\L_{\Q}^{\Phi}$. Hence $\pi$ is the reflection on $\L_{\Q}^{\Phi}$ with respect to $y_1$. 
\item The discriminant of $(\L_{\Q}^{\Phi,\pi},Q)$ is equal to $-{\varepsilon}^{2}p$. The Hasse invariant of $(\L_{\Q}^{\Phi,\pi},Q)$ equals $1$. 
\item The discriminant of $(\L_{\Q}^{\Phi,\pi},p^{-1}\varepsilon^{-2} Q)$ is equal to $1$. The Hasse invariant of $(\L_{\Q}^{\Phi,\pi},p^{-1}\varepsilon^{-2}Q)$ equals $1$. 
\end{enumerate}}
\end{prop}

\subsection{Exceptional isomorphisms}\label{excp}

First, let us recall basic facts of the theory of Clifford algebras and spinor similitude groups. See also \cite{bas}. 

Let $U$ be a finite-dimensional quadratic space over a field $k$, and $C(U)$ the Clifford algebra of $U$. We have a $\Z/2$-grading which is denoted by
\begin{equation*}
C(U)=C^{+}(U)\oplus C^{-}(U). 
\end{equation*}
We regard $U$ as a subspace of $C^{-}(U)$. We have a canonical involution $v\mapsto v'$ on $C(U)$ which is characterized by $(v_1\cdots v_n)'=v_n\cdots v_1$ for $v_1,\ldots,v_n\in U$. 

We define an algebraic group $\GSpin(U)$ over $k$ by
\begin{equation*}
\GSpin(U)(R)=\{g\in (C^{+}(U)\otimes_k R)^{\times}\mid g(U\otimes_k R) g^{-1}=U\otimes_k R,\,g'g\in R^{\times}\}
\end{equation*}
for any $k$-algebra $R$. There is an exact sequence of algebraic groups over $k$: 
\begin{equation*}
1\rightarrow \G_m \rightarrow \GSpin(U) \rightarrow \SO(U)\rightarrow 1.
\end{equation*}
Here the morphism $\GSpin(U)\rightarrow \SO(U)$ is defined as
\begin{equation*}
\GSpin(U)(R) \rightarrow \SO(U)(R);g\mapsto [v\mapsto gvg^{-1}]
\end{equation*}
for any $k$-algebra $R$. 

We apply the results above to $\L_{\Q}^{\pi}$, the $\pi$-invariant part of $\L_{\Q}$. The inclusion $\L^{\pi}_{\Q} \subset \End(\D_{\Q})$ induces an injection 
\begin{equation*}
i\colon C(\L^{\pi}_{\Q}) \rightarrow \End(\D_{\Q})
\end{equation*}
by \cite[Proposition 2.4 (2)]{hp}. It commutes with the actions of $\Phi$ by definition. 

\begin{prop}
\emph{The injection $i$ induces an isomorphism 
\begin{equation*}
C^{+}(\L^{\pi}_{\Q})\cong \End_D(\D_{\Q}). 
\end{equation*}}
\end{prop}

\begin{proof}
If $v\in \L_{\Q}^{\pi}$ then we have $v\circ \iota(d)=-\iota(d)\circ v$ for $d\in \{\varepsilon ,\Pi \}$. Hence the image of $C^{+}(\L_{\Q}^{\pi})$ by the injection $i$ is contained in $\End_D(\D_{\Q})$. On the other hand, we have $\dim_{K_0}C^{+}(\L^{\pi}_{\Q})=2^{5-1}=16$ since $\dim_{K_0}\L_{\Q}^{\pi}=5$. Moreover, we have $\dim_{K_0}\End_D(\D_{\Q})=16$ since $\D_{\Q}$ is free of rank $2$ over $D\otimes_{\Qp} K_0$. Therefore the image of $C^{+}(\L_{\Q}^{\pi})$ by the injection $i$ equals $\End_D(\D_{\Q})$. 
\end{proof}

Let $R$ be a $\Qp$-algebra. We define an action of $H(K_0\otimes_{\Qp} R)$ on $\End(\D_{\Q}\otimes_{\Qp} R)=\End(\D_{\Q})\otimes_{\Qp} R$; cf.~\cite[\S 2.6]{hp} for $R=\Qp$. It is defined by
\begin{equation*}
h\cdot f:=h\circ f\circ h^{-1}
\end{equation*}
for $h\in H(K_0\otimes_{\Qp} R)$ and $f\in \End(\D_{\Q})\otimes_{\Qp} R$. If $x\wedge y\in \bigwedge^{2}_{\Qps}\D_{\Q}$ with $x,y\in \D_{\Q}$, then we have
\begin{equation*}
g\cdot(x\wedge y)=\sml(g)^{-1}g(x)\wedge g(y). 
\end{equation*}
On the other hand, let
\begin{equation*}
H^0(K_0\otimes_{\Qp} R)=\{h\in H(K_0\otimes_{\Qp} R)\mid \sml(h)^2=\det{}_{\Qps \otimes_{\Qp} (K_0\otimes_{\Qp} R)}(h)\}. 
\end{equation*}
Then the action of $h\in H(K_0\otimes_{\Qp} R)$ commutes with $\star$ if and only if $h\in H^0(K_0\otimes_{\Qp} R)$. See \cite[\S 2.3]{hp}. Moreover, \cite[Proposition 2.7]{hp} shows that there is an isomorphism of groups between $H^0(K_0)$ and $\GSpin(\L_{\Q})(K_0)$, which commutes with the actions on $\L_{\Q}$. 

\begin{rem}
The groups $H(K_0)$ and $H^0(K_0)$ are identical to the groups $\GU(\D_{\Q})$ and $\GU^0(\D_{\Q})$ in \cite[\S 2.3]{hp} respectively. 
\end{rem}

We identify $\D_{\Q,1}$ as a dual space of $\D_{\Q,0}$ by 
\begin{equation*}
(\,,\,)\colon \D_{\Q,0}\times \D_{\Q,1}\rightarrow K_0. 
\end{equation*}
We also denote an element of $H(K_0\otimes_{\Qp} R)$ as $(h_0,h_1)$, where $h_i\in \GL(\D_{\Q,i}\otimes_{\Qp} R)$ for $i\in \{0,1\}$. Then we have following: 

\begin{lem}\label{htis}
\emph{
\begin{enumerate}
\item We have two homomorphisms
\begin{align*}
H(K_0\otimes_{\Qp} R)\rightarrow \GL(\D_{\Q,0}\otimes_{\Qp} R)\times \G_{m}(R)&;h\mapsto (h_0,\sml(h)),\\
\GL(\D_{\Q,0}\otimes_{\Qp} R)\times \G_{m}(R)\rightarrow H(K_0\otimes_{\Qp} R)&;(h_0,c)\mapsto (h_0,c(h_0^{\wedge})^{-1}). 
\end{align*}
Here $(h_0^{\wedge})^{-1}$ is the dual map of $h_0$. They are inverse to each other. 
\item For $h=(h_0,h_1)\in H(K_0\otimes_{\Qp} R)$, we have equalities
\begin{equation*}
\det{}_{\Qps \otimes_{\Qp} (K_0\otimes_{\Qp} R)}(h)=(\det{}_{K_0\otimes_{\Qp} R}(h_0),\det{}_{K_0\otimes_{\Qp} R}(h_1))=(\det{}_{K_0\otimes_{\Qp} R}(h_0),\sml(h)^4\det{}_{K_0\otimes_{\Qp} R}(h_0)^{-1}). 
\end{equation*}
in $\Qps \otimes_{\Qp}(K_0\otimes_{\Qp} R)\cong (K_0\otimes_{\Qp} R)\times (K_0\otimes_{\Qp} R)$. 
\end{enumerate}}
\end{lem}

\begin{lem}\label{g0ig}
\emph{The group $G(K_0\otimes_{\Qp} R)$ is contained in $H^0(K_0\otimes_{\Qp} R)$. }
\end{lem}

\begin{proof}
Take $h=(h_0,h_1)\in H(K_0\otimes_{\Qp} R)$. By Lemma \ref{htis}, we have $\sml(h)^2=\det_{\Qps \otimes_{\Qp} (K_0\otimes_{\Qp} R)}(h)$ if and only if $\sml(h)^2=\det_{K_0\otimes_{\Qp} R}(h_0)$. On the other hand, the map above induces an isomorphism $G(K_0\otimes_{\Qp} R)\cong \GSp(\D_{\Q,0})(K_0\otimes_{\Qp}R)$ by Proposition \ref{jgsp} (i). Then the square of the similitude equals the determinant for all elements in $G(K_0\otimes_{\Qp} R)$ since $\dim_{K_0}\D_{\Q,0}=4$. 
\end{proof}

By Lemma \ref{g0ig}, $G(K_0\otimes_{\Qp} R)$ acts on $\L_{\Q}^{\pi}\otimes_{\Qp} R$ as the same formula for the action of $H^0(K_0\otimes_{\Qp} R)$ on $\End(\D_{\Q})\otimes_{\Qp} R$, $(\bigwedge^2_{\Qps}\D_{\Q})\otimes_{\Qp} R$ or $\L_{\Q}\otimes_{\Qp} R$.  

\begin{prop}\label{exis}
\emph{The isomorphism $\End_D(\D_{\Q})\cong C^+(\L_{\Q}^{\pi})$ induces an isomorphism $G(K_0\otimes_{\Qp} R)\cong \GSpin(\L_{\Q}^{\pi})(R)$, which commutes with the actions on $\L_{\Q}^{\pi}\otimes_{\Qp} R$. }
\end{prop}

\begin{proof}
We follow the proof of \cite[Proposition 2.7]{hp}. Take a $\Qp$-algebra $R$ and $g\in (\End_D(\D_{\Q})\otimes_{\Qp} R)\cong (C^+(\L_{\Q}^{\pi})\otimes R)^{\times}$. Note that we have $\langle gx,y\rangle=\langle x,g'y\rangle$ for any $x,y\in \D_{\Q}\otimes_{\Qp} R$ by Lemma \ref{wdhm}. It is equivalent to the condition that $(gx,y)=(x,g'y)$ for any $x,y\in \D_{\Q}\otimes_{\Qp} R$. 

First, suppose $g \in G(K_0\otimes_{\Qp} R)$. By Lemma \ref{g0ig}, the action of $g$ on $\End_D(\D_{\Q})\otimes_{\Qp} R$ preserves the subspace $\L_{\Q}\otimes_{\Qp} R$. Furthermore, the action of $g$ also preserves $\L_{\Q}^{\pi}\otimes_{\Qp} R$ since $g$ is $D$-linear. On the other hand, by the argument above we have
\begin{equation*}
(x,(g'g)\cdot y)=(g\cdot x,g\cdot y)=\sml(g)(x,y)
\end{equation*}
for any $x,y\in \D_{\Q}\otimes_{\Qp} R$. We therefore obtain $g'g=\sml(g)\in R^{\times}$, which concludes $g\in \GSpin(\L_{\Q}^{\pi})(R)$. Second, suppose $g\in \GSpin(\L_{\Q}^{\pi})(R)$. Then we have $g'g\in R^{\times}$ and
\begin{equation*}
(g(x),g(y))=(x,(g'g)(y))=(g'g)(x,y)
\end{equation*}
for any $x,y\in \D_{\Q}\otimes_{\Qp} R$. Hence we have $(g,(g'g))\in G(K_0\otimes_{\Qp} R)$. 

The compatibility with the action on $\L_{\Q}^{\pi}\otimes_{\Qp} R$ is a consequence of the definition. 
\end{proof}

Taking the $\Phi$-invariants of the isomorphism in Proposition \ref{exis} for any $\Qp$-algebra $R$, we obtain the following: 

\begin{cor}\label{eis1}
\emph{There is an isomorphism $J\rightarrow \GSpin(\L_{\Q}^{\Phi,\pi})$ of algebraic groups over $\Qp$. Therefore there is an isomorphism $J^{\ad}\cong \SO(\L_{\Q}^{\Phi,\pi})$ of algebraic groups over $\Qp$. Here $J^{\ad}$ is the adjoint group of $J$, which is isomorphic to $\PGSp(V_0)$. }
\end{cor}

\section{Local model and singularity of $\M_G$}\label{lmfl}

In this section, we prove the flatness of $\M_G$, which is conjectured after the proof of \cite[Corollary 3.35]{rz}. To do this, we consider the local model instead of $\M_G$ itself. The local model is also useful to study the singularity of $\M_G$. 

\subsection{Local model}

We define the \emph{local model} $M^{\loc}_G$ to be the functor that parameterizes $O_D\otimes_{\Zp}\O_S$-modules $\mathcal{F}$ of $\Lambda^0 \otimes_{\Zp} \O_S$ for any $\Zp$-scheme $S$, satisfying the following: 
\begin{itemize}
\item $\mathcal{F}$ is a direct summand of $\Lambda^0 \otimes_{\Zp}\O_S$ as an $\O_S$-module, 
\item $\mathcal{F}^{\perp}=\mathcal{F}$, 
\item (Kottwitz condition) $\det(T-\iota(d) \mid \mathcal{F})=(T^2-\trd_{D/\Qp}(d)T+\nrd_{D/\Qp}(d))^2$ for any $d\in O_D$.
\end{itemize}

The goal of this section is the following: 

\begin{thm}\label{flat}
\emph{
\begin{enumerate}
\item The scheme $M^{\loc}_{G,W}:=M^{\loc}_{G}\times_{\spec \Zp}\spec W$ is flat over $\spec W$. Moreover, it is regular and irreducible of dimension $4$.
\item The scheme $M^{\loc}_{G,W}$ is smooth over $\spec W$ outside the unique $\Fpbar$-rational point $x_0$ such that the corresponding subspace $\mathcal{F}$ of $\Lambda^0 \otimes_{\Zp}\Fpbar$ satisfies $\Pi \mathcal{F}=0$. Moreover, there is an isomorphism 
\begin{equation*}
\O_{M^{\loc}_{G,W},x_0}\cong (W[t_1,t_2,t_3,t_4]/(t_1t_2+t_3t_4+p))_{(t_1,t_2,t_3,t_4)}. 
\end{equation*}
\end{enumerate}}
\end{thm}

By \cite[Proposition 3.33]{rz}, we obtain the following: 

\begin{cor}\label{rzsg}
\emph{The formal scheme $\M_G$ is flat over $\spf W$. Moreover, it is regular of purely $4$-dimensional (that is, the local ring is regular of dimension $4$ for each point) and formally smooth over $\spf W$ outside the discrete set of $\Fpbar$-rational points such that the corresponding tuple $(X,\iota,\lambda,\rho)$ satisfies $\iota(\Pi)=0$ on $\Lie(X)$. }
\end{cor}

\begin{rem}
There is a point in $\M_G$ which is not formally smooth over $\spf W$. Indeed, the $\Fpbar$-rational point corresponding to $(\X_0,\iota_0,\lambda_0,\id_{\X_0})$ satisfies the condition since $\Lie(\X_0)=\D/F^{-1}(p\D)=\D/y_1 \D$. 
\end{rem}

\subsection{Proof of Theorem \ref{flat}}\label{flpf}

In this paper, we use a generalization of \cite[2.4]{hp} in Appendix \ref{rdhp}. Note that there is a more direct proof which reduces to \cite[\S 3]{yu}. See \cite[\S 2.3]{wan}. 

Put
\begin{equation*}
\L_0:=\{v\in \L_{\Q}^{\pi}\mid v(\D)\subset \D \}=\bigoplus_{i=2}^{6}Wy_i\subset \L_{\Q}^{\pi}. 
\end{equation*}
(see Section \ref{rzdt} for the definition of $\D$). Note that we have $\L_0\subset \L_0^{\vee}$ and $\length_{W}(\L_0^{\vee}/\L_0)=1$. Let $\Isot_{\L_0}$ be the functor which parametrizes all isotropic lines in $\L_0\otimes_{W}\O_S$ for any $W$-scheme $S$. Here, an isotropic line in $\L_0\otimes_{W}\O_S$ is an $\O_S$-submodule of $\L_0\otimes_{W}\O_S$ which is locally a direct summand of rank $1$ which is totally isotropic with respect to $Q$. 

\begin{prop}\label{gnfb}
\emph{
\begin{enumerate}
\item For a $W$-algebra $R$ and $\mathcal{F}\in M^{\loc}_{G,W}(R)$, 
\begin{equation*}
l^{\pi}(\mathcal{F}):=\{v\in \L_0\otimes_{W}R\mid v(\mathcal{F})=0\}
\end{equation*}
is an isotropic line in $\L_0\otimes_{W}R$. Hence $\mathcal{F}\mapsto l^{\pi}(\mathcal{F})$ induces a morphism of $W$-schemes
\begin{equation*}
f_{\pi}\colon M^{\loc}_{G,W}\xrightarrow{\cong} \Isot_{\L_0}. 
\end{equation*}
\item The morphism $f_{\pi}$ in (i) is an isomorphism. Moreover, the image of $\Pi(\D \otimes_{W}\Fpbar)\in M^{\loc}_{G}$ under $f_{\pi}$ equals the radical of $\L_0\otimes_{W}\Fpbar$. 
\end{enumerate}}
\end{prop}

\begin{proof}
Put $\L:=\{v\in \L_{\Q}\mid v(\D)\subset \D\}$. By Lemma \ref{liat}, $l_R(\mathcal{F}):=\{v\in \L \otimes_{W}R\mid v(\mathcal{F})=0\}$ is an isotropic line in $\L \otimes_{W}R$ for any $\mathcal{F}\in \Lag^{\Zps}_{\D}(R)$ (that is, a Lagrangian subspace of $\D$), and $\mathcal{F}\mapsto l_R(\mathcal{F})$ induces an isomorphism $\Lag^{\Zps}_{\D}\cong \Isot_{\L}$. Moreover, $\mathcal{F}$ is $\Pi$-stable if and only if $l_R(\mathcal{F})\subset \L_0 \otimes_{W}R$. Note that the $\Pi$-stable locus of $\Lag^{\Zps}_{\D}$ equals $M^{\loc}_{G,W}$ by definition. Hence the assertions (i) and the isomorphy of $f_{\pi}$ follow. On the other hand, by the definitions of $f_{\pi}$ and $e_i,f_j$ (see Lemma \ref{base}), we have $f_{\pi}(\Pi(\D \otimes_{W}\Fpbar))=\Fpbar y_2$ and $\Fpbar y_2$ is the radical of $\L_0 \otimes_{W}\Fpbar$. Hence the assertion for $f_{\pi}(\Pi(\D \otimes_{W}\Fpbar))$ follows. 
\end{proof}

By definition, the $W$-scheme $\Isot_{\L_0}$ is the $W$-scheme $Q(\L_0)$ in the sense of \cite[12.7.1]{hpr}. Hence we have the following: 

\begin{thm}\label{hprt} (\cite[Proposition 12.6]{hpr})
\emph{There is an isomorphism of $W$-schemes between $\Isot_{\L_0}$ and the hypersurface defined by
\begin{equation*}
x_1x_4+x_2x_3+px_5^2=0
\end{equation*}
in $\Proj W[x_1,x_2,x_3,x_4,x_5]$. Moreover, the radical of $\L_0\otimes_{W}\Fpbar$ corresponds to the unique non-smooth point of the latter scheme. }
\end{thm}

Hence, all assertions in Theorem \ref{flat} follow from Theorem \ref{hprt}. 

\section{Bruhat--Tits stratification}\label{vlbt}

In this section, we start to study the underlying space of $\M_G$. First, we describe the set of points of $p^{\Z}\backslash \M_G$ by specific self-dual lattices in the certain quadratic space $\L_{\Q}$. Second, we define the notion of vertex lattices. They can be written in terms of the Bruhat--Tits building of $\SO(\L_{\Q}^{\Phi,\pi})(\Qp)$. We introduce the closed formal subschemes attached to vertex lattices and construct the locally closed stratification by them. Finally, using the exceptional isomorphism $J^{\ad}(\Qp)\cong \SO(\L_{\Q}^{\Phi,\pi})(\Qp)$, we rewrite the stratification above the Bruhat--Tits building of $J^{\ad}(\Qp)$. As an application of this, we also prove some properties of the $J(\Qp)$-action on $\M_G^{\red}$. 

\subsection{Description of the set of geometric points of $\M_G$}\label{rtpt}

In this section, let $k/\Fpbar$ be a field extension, $W(k)$ the Cohen ring of $k$ and $K:=\Frac W(k)$. We construct a bijection between the $k$-rational points of $p^{\Z}\backslash \M_G$ and the set of certain lattices in $\L_{\Q}\otimes_{K_0}K$, which is compatible with that of $H$ in \cite[Corollary 2.14]{hp}. First, we recall the bijection in \cite[\S 2.4]{hp}. A Dieudonn{\'e} lattice in $\D_{\Q,K}:=\D_{\Q}\otimes_{K_0} K$ is a $\Zps$-stable $W(k)$-lattice $M$ in $\D_{\Q,K}$ satisfying the following conditions: 
\begin{itemize}
\item $M^{\vee}=p^iM$ for some $i\in \Z$, 
\item $pM\subset F^{-1}(pM)\subset M$, 
\item $\dim_{k}M/F^{-1}(pM)=4$. 
\end{itemize}
There is a relation between Dieudonn{\'e} lattices and $\M_H$ as follows. It is based on the theory of windows, see \cite{zin}. 
\begin{thm}\label{rtdu} (\cite[Corollary 2.11]{hp})
\emph{There is a bijection
\begin{equation*}
\M_{H}(k)\cong \{\text{Dieudonn{\'e} lattices in }\D_{\Q,K}\}. 
\end{equation*}}
\end{thm}
On the other hand, a special lattice in $\L_{\Q,K}:=\L_{\Q}\otimes_{K_0} K$ is a self-dual lattice $L$ in $\L_{\Q,K}$ such that 
\begin{equation*}
\length_{W(k)}(L+\Phi_{*}(L))/L=1, 
\end{equation*}
where $\Phi_{*}(L)$ is the $W(k)$-lattice in $\L_{\Q,K}$ generated by $\Phi(L)$. They are related as follows: 
\begin{thm}\label{hppc} (\cite[Corollary 2.14]{hp})
\emph{The map
\begin{equation*}
p^{\Z}\backslash \{\text{Dieudonn{\'e} lattices in }\D_{\Q,K}\} \rightarrow \{\text{special lattices in }\L_{\Q,K}\}
\end{equation*}
which is defined by
\begin{equation*}
M\mapsto L(M):=\{v\in \L_{\Q,K}\mid v(F^{-1}(pM))\subset F^{-1}(pM)\}
\end{equation*}
is bijective. }
\end{thm}

\begin{rem}
The original proof of Theorem \ref{hppc} needs a little modification. We will consider it more precisely in Appendix \ref{hpcr}. 
\end{rem}

\begin{dfn}
\begin{enumerate}
\item A \emph{$\Pi$-stable Dieudonn{\'e} lattice in $\D_{\Q,K}$} is a Dieudonn{\'e} lattice $M$ in $\D_{\Q,K}$ such that $y_1(M)\subset M$. 
\item A \emph{$\pi$-special lattice in $\L_{\Q,K}$} is a special lattice $L$ in $\L_{\Q,K}$ containing $y_1$. 
\end{enumerate}
\end{dfn}

The following follows from Proposition \ref{mgeb}: 

\begin{prop}\label{mgpd}
The bijection in Theorem \ref{rtdu} induces a bijection
\begin{equation*}
\M_{G}(k)\cong \{\Pi \text{-stable Dieudonn{\'e} lattices in }\D_{\Q,K}\}. 
\end{equation*}
\end{prop}

Next, we describe a correspondence between $\Pi$-stable Dieudonn{\'e} lattices in $\D_{\Q,K}$ and $\pi$-special lattices in $\L_{\Q,K}$, using the bijection in Theorem \ref{hppc}. 

\begin{prop}\label{pscr}
\emph{The bijection in Theorem \ref{hppc} induces a bijection
\begin{equation*}
p^{\Z}\backslash \{\Pi \text{-stable Dieudonn{\'e} lattices in }\D_{\Q,K}\} \rightarrow \{\pi \text{-special lattices in }\L_{\Q,K}\}. 
\end{equation*}}
\end{prop}

\begin{proof}
Let $M$ be a Dieudonn{\'e} lattice in $\D_{\Q,K}$, and denote the corresponding special lattice in $\L_{\Q,K}$ by $L$. By the definition of the bijection in Theorem \ref{hppc}, $M$ is $\Pi$-stable if and only if $y_1\in L$, that is, $L$ is $\pi$-special. 
\end{proof}

Finally, we give a relation between the action of $\Pi$ and the map $\pi$, which will be used in Section \ref{btsn}: 

\begin{prop}\label{picr}
\emph{Let $M$ be a $\Pi$-stable Dieudonn{\'e} lattice in $\D_{\Q,K}$ and $L$ a corresponding $\pi$-special lattice in $\L_{\Q,K}$ under the bijection in Proposition \ref{pscr}. 
\begin{enumerate}
\item If $M$ corresponds to a point of $\M_G^{(i)}(k)$ for some $i\in \Z$. Then both $y_1(M)$ and $F(M)$ are also $\Pi$-stable Dieudonn{\'e} lattices in $\D_{\Q,K}$, and the corresponding points of $\M_G(k)$ lie in $\M_G^{(i-1)}$. 
\item Under the bijection in Proposition \ref{pscr}, $y_1(M)$ and $F(M)$ correspond to $\pi(L)$ and $\Phi(L)$ respectively. 
\end{enumerate}}
\end{prop}

\begin{proof}
(i): We only prove the assertion for $y_1(M)$. Other case is similar. Since the action of $\Pi$ commutes with $F$, we can see that $y_1(M)$ satisfies the conditions of $\Pi$-stable Dieudonn{\'e} lattices except for that of relation with the dual. However, using the assumption $M^{\vee}=p^iM$ and Lemma \ref{sdlt}, we have $(y_1(M))^{\vee}=p^{i-1}y_1(M)$. Therefore the assertion for $y_1(M)$ follows. 

(ii): This follows from the definition of the bijection in Proposition \ref{pscr}. 
\end{proof}

\subsection{Vertex lattices}\label{vtdn}

First, we introduce the notion of vertex lattices, which will be indices of a locally closed stratification of $\M_G$. 

\begin{dfn}
\begin{itemize}
\item A \emph{vertex lattice} is a lattice $\Lambda \subset \L^{\Phi,\,\pi}_{\Q}$ such that $p\Lambda \subset \Lambda^{\vee} \subset \Lambda$. 
We call $t(\Lambda):=\dim_{\Fp}(\Lambda/\Lambda^{\vee})$ the \emph{type} of $\Lambda$. 
\item We denote the set of all vertex lattices by $\Vrt$. 
\end{itemize}
\end{dfn}

\begin{lem}
\emph{For $\Lambda \in \Vrt$, we have $t(\Lambda)\in \{1,3,5\}$. }
\end{lem}

\begin{proof}
Fix an orthogonal $\Zp$-basis $v_1,\ldots,v_5$ of $\Lambda$ (this is possible since $p>2$). Then we have
\begin{equation*}
t(\Lambda)=-\sum_{i=1}^{5}\ord_p(Q(v_i))
\end{equation*}
and $0\leq t(\Lambda)\leq 5$ by the definition of vertex lattices. On the other hand, we have
\begin{equation*}
\ord_p(\disc(\L_{\Q}^{\Phi,\pi}))=\sum_{i=1}^{5}\ord_p(Q(v_i)) \bmod 2, 
\end{equation*}
where $\disc(\L_{\Q}^{\Phi,\pi})$ is the discriminant of $\L_{\Q}^{\Phi,\pi}$. By Proposition \ref{lqpi} (ii), we know that $\ord_p(\disc(\L_{\Q}^{\Phi,\pi}))$ is odd. Therefore $t(\Lambda)$ is also odd. 
\end{proof}

For $t\in \{1,3,5\}$, we set $\Vrt(t):=\{\Lambda \in \Vrt \mid t(\Lambda)=t\}$.

\begin{prop}
\emph{We have $\Vrt(t)\neq \emptyset$ for each $t\in \{1,3,5\}$. }
\end{prop}

\begin{proof}
By Proposition \ref{lqpi} (iii), there is a $\Qp$-basis $w_1,\ldots,w_5$ of $\L_{\Q}^{\Phi}$ whose Gram matrix is given by
\begin{equation*}
\begin{pmatrix}
p^{-1}&&&&0\\
&&&&p^{-1}\\
&&&p^{-1}&\\
&&p^{-1}&&\\
0&p^{-1}&&&
\end{pmatrix}. 
\end{equation*}
Now let $\Lambda'_i\,(1\leq i\leq 3)$ be lattices in $\L_{\Q}^{\Phi,\pi}$ as below: 
\begin{align*}
\Lambda'_1&:=\Zp w_1\oplus \Zp w_2\oplus \Zp w_3\oplus \Zp pw_4\oplus \Zp pw_5, \\
\Lambda'_2&:=\Zp w_1\oplus \Zp w_2\oplus \Zp w_3\oplus \Zp w_4\oplus \Zp pw_5,\\
\Lambda'_3&:=\Zp w_1\oplus \Zp w_2\oplus \Zp w_3\oplus \Zp w_4\oplus \Zp w_5. 
\end{align*}
Then we have $\Lambda'_i\in \Vrt(2i-1)$. 
\end{proof}

Next, we consider combinatorial properties of vertex lattices. 

\begin{dfn}\label{sov5}
For $\Lambda \in \Vrt$, put $\SO(\Lambda):=\{g\in \SO(\L_{\Q}^{\Phi,\pi})(\Qp)\mid g\cdot \Lambda=\Lambda \}$. 
\end{dfn}

\begin{prop}\label{ltic}
\emph{
\begin{enumerate}
\item If $\Lambda \in \Vrt(1)\sqcup \Vrt(3)$. Then 
\begin{equation*}
\#\{\Lambda'\in \Vrt(5) \mid \Lambda \subset \Lambda'\}=
\begin{cases}
2(p+1) & \text{if }t(\Lambda)=1,\\
2 & \text{if }t(\Lambda)=3. 
\end{cases}
\end{equation*}
\item If $\Lambda \in \Vrt(3)$. Then $\#\{\Lambda'\in \Vrt(1) \mid \Lambda' \subset \Lambda \}=p+1$. 
\item Let $\Lambda \in \Vrt(5)$. For $t\in \{1,3\}$, the group $\stab(\Lambda)$ acts transitively on the set $\{\Lambda'\in \Vrt(t) \mid \Lambda' \subset \Lambda \}$, which has exactly $(p+1)(p^2+1)$-elements. 
\item If $\Lambda_1\in \Vrt(5)$ and $\Lambda_2\in \Vrt(3)$. Then $\#\{\Lambda \in \Vrt(5) \mid \Lambda \cap \Lambda_1=\Lambda_2 \}=1$. 
\item If $\Lambda_1\in \Vrt(5)$ and $\Lambda_2\in \Vrt(1)$. Then $\#\{\Lambda \in \Vrt(5) \mid \Lambda \cap \Lambda_1=\Lambda_2 \}=p$. 
\end{enumerate}}
\end{prop}

To show the assertions above, we need some results about quadratic spaces over finite fields. We denote by $\H_{\Fp}$ the hyperbolic plane over $\Fp$. Moreover, we endow a $1$-dimensional space $\Fp$ with the norm form. 

\begin{dfn}\label{omlm}
For $\Lambda \in \Vrt$, let 
\begin{equation*}
\Omega_0(\Lambda):=\Lambda/\Lambda^{\vee},\quad \Omega'_0(\Lambda):=\Lambda^{\vee}/p\Lambda. 
\end{equation*}
These are $\Fp$-vector space. We endow $\Omega_0(\Lambda)$ and $\Omega'_0(\Lambda)$ with quadratic forms $-\varepsilon^{2}pQ\bmod p$ and $Q\bmod p$ respectively. 
\end{dfn}

\begin{lem}\label{qtvx}
\emph{Let $\Lambda \in \Vrt$. 
\begin{enumerate}
\item If $\Lambda \in \Vrt(t)$ for some $t\in \{1,3\}$ then there is an isomorphism $\Omega'_0(\Lambda)\cong \H_{\Fp}^{\oplus (5-t(\Lambda))/2}$ of quadratic spaces over $\Fp$. 
\item If $\Lambda \in \Vrt(t)$ for some $t\in \{3,5\}$ then there is an isomorphism $\Omega_0(\Lambda)\cong \H_{\Fp}^{\oplus (t(\Lambda)-1)/2}\oplus \Fp$ of quadratic spaces over $\Fp$. 
\end{enumerate}}
\end{lem}

\begin{proof}
We will only prove the assertion (i) for $t(\Lambda)=3$. The other assertions follow by the same argument. 

Take an orthogonal basis $v_1,\ldots,v_5$ over $\Zp$ of $\Lambda$. Since $\Lambda \in \Vrt(3)$, after permuting $v_i$, we may write 
\begin{equation*}
Q(v_i)=
\begin{cases}
u_i& \text{if }i=1,2,\\
p^{-1}u_i&\text{if }3\leq i\leq 5
\end{cases}
\end{equation*}
for some $u_i\in \Zpt$. Then we have $\Lambda^{\vee}/p\Lambda=\Fp v_1 \oplus \Fp v_2$. Using the basis $v_i$, the Hasse invariant of $\L_{\Q}^{\Phi,\pi}$ equals $(p,-u_1u_2)_p$, where $(\,,\,)_p$ is the Hilbert symbol with respect to $p$. On the other hand, Proposition \ref{lqpi} (ii) implies $(p,-u_1u_2)_p=1$, that is, $u_1u_2=-1$ in $\Zp^{\times}/(\Zp^{\times})^2$. Therefore we have $\det(\Lambda^{\vee}/p\Lambda)=-1$ in $\Fp^{\times}/(\Fp^{\times})^2$, which implies that $\Lambda^{\vee}/p\Lambda$ is isomorphic to $\H_{\Fp}$. 
\end{proof}

\begin{lem}\label{tinb}
\emph{
\begin{enumerate}
\item For $t\in \Zpn$, the group $\SO(\H_{\Fp}^{\oplus t}\oplus \Fp)(\Fp)$ acts transitively on the set of all isotropic lines in $\H_{\Fp}^{\oplus t}\oplus \Fp$. The number of non-zero isotropic lines in $\H_{\Fp}^{\oplus t}\oplus \Fp$ equals $\sum^{2t-1}_{i=0}p^i$. 
\item The group $\SO(\H_{\Fp}^{\oplus 2}\oplus \Fp)(\Fp)$ acts transitively on the set of all totally isotropic subspaces of dimension $2$ in $\H_{\Fp}^{\oplus 2}\oplus \Fp$. The number of the set above equals $(p+1)(p^2+1)$. 
\item The number of isotropic lines in $\H_{\Fp}$ equals $2$. 
\item The number of totally isotropic subspaces of dimension $2$ in $\H_{\Fp}^{\oplus 2}$ equals $2(p+1)$. 
\item Let $W$ be a totally isotropic subspaces of dimension $2$ in $\H_{\Fp}^{\oplus 2}$. Then the number of Lagrangian subspace $W'\neq W$ in $\H_{\Fp}^{\oplus 2}$ such that $W\cap W'=\{0\}$ equals $p$. 
\end{enumerate}}
\end{lem}

\begin{proof}
These follow from \cite[3.7.2, 3.7.4]{wil}. 
\end{proof}

If $g\in \SO(\Lambda)$, then we have $g(\Lambda^{\vee})=\Lambda^{\vee}$, which concludes that $g\vert_{\Lambda}$ induces an element of $\SO(\Omega_0(\Lambda))(\Fp)$. Hence we obtain a homomorphism
\begin{equation*}
\red_{\Lambda}\colon \SO(\Lambda)\rightarrow \SO(\Omega_0(\Lambda))(\Fp);\,g \mapsto \gbar_{\Lambda}. 
\end{equation*}

\begin{lem}\label{sosj}
\emph{Under the notation above, we further assume that $\Lambda \in \Vrt(5)$. Then the homomorphism $\red_{\Lambda}$ is surjective. }
\end{lem}

\begin{proof}
The lattice $\E(\Lambda):=\iota_0(\varepsilon \Pi)\Lambda$ in $\E_0$ is self-dual. Let $\SO(\E(\Lambda))$ be the algebraic group over $\Zp$. Then we have $\SO(\E(\Lambda))(\Zp)\cong \SO(\Lambda)$. Furthermore it induces isomorphisms
\begin{equation*}
\SO(\E(\Lambda))(\Fp)\cong \SO(\Omega_0(\Lambda))(\Fp)\cong \SO(\H_{\Fp}^{\oplus 2}\oplus \Fp)(\Fp). 
\end{equation*}
Since $p>2$, the algebraic group $\SO(\E(\Lambda))$ is smooth over $\Zp$. Therefore the homomorphism $\SO(\E(\Lambda))(\Zp)\rightarrow \SO(\E(\Lambda))(\Fp)$ is surjective by \cite[Proposition 2.8.13]{fu}, which concludes the assertion. 
\end{proof}

\begin{proof}[Proof of Proposition \ref{ltic}]
(i): We have a bijection below: 
\begin{equation*}
\{\Lambda'\in \Vrt(5)\mid \Lambda \subset \Lambda'\}\xrightarrow{\cong}
\{W\subset \Omega'_0(\Lambda)\colon \Fp\text{-subspace}\mid W=W^{\perp}\};\Lambda' \mapsto p\Lambda'/p\Lambda. 
\end{equation*}
By Lemma \ref{qtvx} (i) and Lemma \ref{tinb} (iii), (iv), the number of the set in the right-hand side equals 
\begin{equation*}
\begin{cases}
2(p+1)& \text{if }t(\Lambda)=1,\\
2& \text{if }t(\Lambda)=2. 
\end{cases}
\end{equation*}
Hence the assertion follows. 

(ii): We have a bijection below, which commutes with actions of $\stab(\Lambda)$ and $\SO(\Omega_0(\Lambda))(\Fp)$: 
\begin{equation*}
\{\Lambda'\in \Vrt(1)\mid \Lambda' \subset \Lambda \}\xrightarrow{\cong}
\{P\subset \Omega_0(\Lambda) \colon \Fp\text{-subspace}\mid 0\neq P\subset P^{\perp}\};
\Lambda' \mapsto (\Lambda')^{\vee}/\Lambda^{\vee}. 
\end{equation*}
By Lemma \ref{qtvx} (ii) and Lemma \ref{tinb} (i), $\SO(\Omega_0(\Lambda))(\Fp)$ acts transitively on the set in the right-hand side, which has exactly $p+1$-elements. Hence the assertion follows from Lemma \ref{sosj}. 

(iii): We have a bijection below: 
\begin{equation*}
\{\Lambda'\in \Vrt(t)\mid \Lambda' \subset \Lambda \}\xrightarrow{\cong}
\{P\subset \Omega_0(\Lambda) \colon \Fp\text{-subspace}\mid 0\neq P\subset P^{\perp},\dim_{\Fp}P=(5-t)/2\};\,
\Lambda' \mapsto (\Lambda')^{\vee}/p\Lambda. 
\end{equation*}
By Lemma \ref{qtvx} (ii) and Lemma \ref{tinb} (ii), the number of the set in the right-hand side equals $(p+1)(p^2+1)$. Applying Lemma \ref{sosj} to $\Lambda$, we obtain the assertion. 

(iv): The bijection in the proof of (i) for $\Lambda=\Lambda_2$ induces a bijection
\begin{equation*}
\{\Lambda \in \Vrt(5)\mid \Lambda_2 \subset \Lambda,\,\Lambda \cap \Lambda_1=\Lambda_2\}\xrightarrow{\cong}
\{W\subset \Omega'(\Lambda_2) \colon \Fp\text{-subspace}\mid W=W^{\perp},W\cap (p\Lambda_1/p\Lambda)=\{0\}\}. 
\end{equation*}
By Lemma \ref{tinb} (iii), the number of the right-hand side of the bijection above equals $1$. Hence the assertion follows. 

(v): The bijection in the proof of (i) for $\Lambda=\Lambda_2$ induces a bijection
\begin{equation*}
\{\Lambda \in \Vrt(5)\mid \Lambda_2 \subset \Lambda,\,\Lambda \cap \Lambda_1=\Lambda_2\}\xrightarrow{\cong}
\{W\subset \Omega'(\Lambda_2) \colon \Fp\text{-subspace}\mid W=W^{\perp},W\cap (p\Lambda_1/p\Lambda)=\{0\}\}. 
\end{equation*}
By Lemma \ref{tinb} (v), the number of the right-hand side of the bijection above equals $p$. Hence the assertion follows. 
\end{proof}

Next, we state the property which concerns with the connectedness of $\M_G^{(0)}$. The proof is exactly the same as that of \cite[Proposition 5.1.5]{hp2}:  

\begin{prop}\label{adjv}
For any $\Lambda,\Lambda'\in \Vrt$, there is a sequence 
\begin{equation*}
\Lambda_0=\Lambda,\Lambda_1,\ldots ,\Lambda_n=\Lambda'
\end{equation*}
such that every $\Lambda_i$ lies in $\Vrt$, and we have either $\Lambda_i\subset \Lambda_{i+1}$ or $\Lambda_i\supset \Lambda_{i+1}$ for each $i\in \{0,\ldots,n-1\}$. 
\end{prop}

In the sequel, we call a vertex lattice in the sense of \cite[Definition 2.17]{hp} an \emph{$H$-vertex lattice}. Denote by $\Vrt_H$ the set of all $H$-vertex lattices. We relate vertex lattices with $H$-vertex lattices. 

\begin{dfn}\label{hvdl}
For any lattice $L$ in $\L_{\Q}$ or $\L_{\Q}^{\Phi}$, we denote the dual lattice of $L$ by $L^{\natural}$. 
\end{dfn}

\begin{prop}\label{vtcr}
\emph{We have two maps
\begin{align*}
\varphi \colon \Vrt \rightarrow \{\widetilde{\Lambda}\in \Vrt_H\mid p^{-1}y_1\in \widetilde{\Lambda}\}&;\,\Lambda \mapsto \Lambda \oplus p^{-1}y_1\Zp, \\
\psi \colon \{\widetilde{\Lambda}\in \Vrt_H\mid p^{-1}y_1\in \widetilde{\Lambda}\}\rightarrow \Vrt &;\,\widetilde{\Lambda} \mapsto \widetilde{\Lambda}\cap \L_{\Q}^{\Phi,\pi}. 
\end{align*}
They are inverse to each other. Moreover, if $\Lambda \in \Vrt(t)$ for some $t\in \{1,3,5\}$, then $\varphi(\Lambda)$ is an $H$-vertex lattice of type $t+1$. }
\end{prop}

To prove the assertion above, we need the following: 

\begin{lem}\label{anis}
\emph{Let $\widetilde{\Lambda}$ be an $H$-vertex lattice containing $p^{-1}y_1$. Then we have
\begin{equation*}
\{y\in \Qp y_1 \mid \text{there is an element $z\in \L_{\Q}^{\Phi,\pi}$ such that }y+z\in \widetilde{\Lambda}\} \subset p^{-1}y_1\Zp. 
\end{equation*}}
\end{lem}

\begin{proof}
We use the correspondence between $H$-vertex lattices and lattices $\widetilde{L}$ in $\H^{\oplus 2}\oplus \Qps \cong (\L_{\Q}^{\Phi},p^{-1}Q)$ with $\widetilde{L} \subset \widetilde{L}^{\vee}\subset p^{-1}\widetilde{L}$, appearing in the proof of \cite[Proposition 2.22]{hp}. Here $L^{\vee}$ is the dual lattice of $L$ with respect to $p^{-1}Q$. Under this correspondence, the assertion is equivalent to the formula
\begin{equation*}
e(\widetilde{\Lambda}):=\{y\in \Qp \mid \text{there is an element $z\in \H^{\oplus 2}\oplus \Qp \varepsilon$ such that }y+z\in p\widetilde{\Lambda}\} \subset \Zp. 
\end{equation*}
By \cite[Lemma 29.2 (1),(3)]{shi}, there is a self-dual lattice $\widetilde{L}$ containing $p\widetilde{\Lambda}$. Note that $1\in \widetilde{L}$ since $p^{-1}y_1\in \widetilde{\Lambda}$. Hence we have a decomposition $\widetilde{L}=\widetilde{L}_0\oplus \Zp$, where $\widetilde{L}_0=\widetilde{L}\cap (\H^{\oplus 2}\oplus \Qp \varepsilon)$. Since $e(\widetilde{\Lambda})$ is the image of $p\widetilde{\Lambda}$ under the second projection $(\H^{\oplus 2}\oplus \Qp \varepsilon)\oplus \Qp \rightarrow \Qp$, the assertion follows. 
\end{proof}

\begin{proof}[Proof of Proposition \ref{vtcr}]
If $\Lambda$ is a vertex lattice of type $t\in \{1,3,5\}$ then $\varphi(\Lambda)$ is an $H$-vertex lattice of type $t+1$ containing $p^{-1}y_1$ since $\varphi(\Lambda)^{\natural}=\Lambda^{\vee}\oplus y_1\Zp$. On the other hand, if $\widetilde{\Lambda}$ is an $H$-vertex lattice containing $p^{-1}y_1$, then we have $\widetilde{\Lambda}=(\widetilde{\Lambda}\cap \L^{\Phi,\pi})\oplus p^{-1}y_1\Zp$ by Lemma \ref{anis}. Since $\widetilde{\Lambda}$ is an $H$-vertex lattice, we obtain that $\psi(\widetilde{\Lambda})$ is a vertex lattice. The assertions $\psi \circ \varphi=\id_{\Vrt}$ and $\varphi \circ \psi=\id$ follow from the definitions of $\varphi$ and $\psi$. 
\end{proof}

For a lattice $L$ in $\L_{\Q}$ and an integer $r\in \Znn$, we set $L^{(r)}:=L+\Phi(L)+\cdots+\Phi^r(L)$. 

\begin{prop}\label{spvt}
\emph{For a $\pi$-special lattice $L$ in $\L_{\Q}$, set $L^{r}:=(L+\pi(L))^{(r)}$ for $r\in \Znn$. There is an integer $d\in \{0,1,2\}$ such that $L=L^{-1}\subsetneq \cdots \subsetneq L^{d}=L^{d+1}$ and $L^{i}/L^{i-1}$ is of length $1$ for $0\leq i\leq d$. Moreover, 
\begin{equation*}
\Lambda(L):=L^d\cap \L_{\Q}^{\Phi,\pi}
\end{equation*}
is a vertex lattice of type $2d+1$ and $\Lambda(L)^{\vee}=L\cap \L_{\Q}^{\Phi,\pi}$. }
\end{prop}

To prove Proposition \ref{spvt}, we use the following: 

\begin{lem}\label{y1il}
\emph{If $L$ is a $\pi$-special lattice in $\L_{\Q}$, then we have $L+\pi(L)=L+p^{-1}y_1W$. }
\end{lem}

\begin{proof}
Let $L$ be a $\pi$-special lattice in $\L_{\Q}$. Taking the duals, it suffices to show $L\cap \pi(L)=L\cap (\L_{\Q}^{\pi}\oplus y_1W)$. 

First, take $v\in L\cap \pi(L)$. Then we have $\pi(v)\in L$, and thus we obtain $(v+\pi(v))/2\in \L_{\Q}^{\pi}$ and $(v-\pi(v))/2\in L\cap \L_{\Q}^{-\pi}=y_1W$. Here we use the assumption that $L$ is $\pi$-special. Hence we obtain $v=(v+\pi(v))/2+(v-\pi(v))/2 \in \L_{\Q}^{\pi}\oplus y_1W$. 

Next, take $v\in L\cap (\L_{\Q}^{\pi}\oplus y_1W)$ with $v=w+ay_1$, where $w\in \L_{\Q}^{\pi}$ and $a\in W$. Since $L$ is $\pi$-special, we have $y_1\in L$. Therefore we obtain $w\in L\cap \L_{\Q}^{\pi}\subset L\cap \pi(L)$, which concludes $v\in L\cap \pi(L)$. 
\end{proof}

\begin{proof}[Proof of Proposition \ref{spvt}]
First, we have $\length_{W}L_0/L=1$ by $p^{-1}y_1\not\in L$ and Lemma \ref{y1il}. By \cite[Proposition 2.19]{hp}, there is the minimum integer $d\in \{1,2,3\}$ such that $L^{(d)}$ is $\Phi$-stable. Note that $p^{-1}y_1\in L^{(d)}$ if $d=3$, since $y_1\in (L^{(d)})^{\natural}=pL^{(d)}$. We have
\begin{equation*}
\length_{W}L^{r}/L^{r-1}\leq \length_{W}L^{(r)}/L^{(r-1)}\leq 1
\end{equation*}
for any $r\in \Zpn$ by Lemma \ref{y1il} and \cite[Proposition 2.19]{hp}. On the other hand, we have an equality for $r\in \Zpn$:
\begin{equation}\label{lneq}
\length_{W}L^{r}/L^{(r)}+\length_{W}L^{(r)}/L^{(r-1)}
=\length_{W}L^{r}/L^{r-1}+\length_{W}L^{r-1}/L^{(r-1)}. 
\end{equation}

First, assume $p^{-1}y_1\not\in L^{(d)}$ (hence $d\leq 2$). We claim that $d$ is the minimum integer among $r\in \Znn$ such that $L^r$ is $\Phi$-stable. The $\Phi$-stability of $L^d$ follows from the equality $L^d=L^{(d)}+p^{-1}y_1W$, which is a consequence of Lemma \ref{y1il}. Next, if $0\leq r\leq d$ then we prove that $L^{r-1}\subsetneq L^r$. By the assumption $p^{-1}y_1\not\in L^{(d)}$, we have $\length_{W}L^{r}/L^{(r)}=\length_{W}L^{r-1}/L^{(r-1)}=1$ by Lemma \ref{y1il}. On the other hand, we have $\length_{W}L^{(r)}/L^{(r-1)}=1$ by the minimality of $d$. Therefore we obtain $\length_{W}L^{r}/L^{r-1}=1$ by (\ref{lneq}). 

Next, assume $p^{-1}y_1\in L^{(d)}$. We claim that $d-1$ is the minimum integer among $r\in \Znn$ such that $L^r$ is $\Phi$-stable. Let $r_0$ be the minimum integer among $r\in \Znn$ such that $p^{-1}y_1\in L^{(r)}$. Then $1\leq r_0\leq d$. Therefore, we have $\length_{W}L^{(r)}/L^{(r-1)}=1$ for $1\leq r\leq r_0$. On the other hand, by the definition of $r_0$ and Lemma \ref{y1il}, we have $\length_{W}L^{r_0}/L^{(r_0)}=0$ and $\length_{W}L^{r}/L^{(r)}=1$ for $0\leq r\leq r_0-1$. Combining them with the formula (\ref{lneq}) for $r=1,\ldots ,r_0$, we obtain $L^{r_0}=L^{r_0-1}$ and $\length_{W}L^r/L^{r-1}=1$ for $1\leq r\leq r_0-1$, that is, $r_0-1$ is the minimum integer among $r\in \Znn$ such that $L^r$ is $\Phi$-stable. Hence it suffices to show $r_0=d$. Now suppose $r_0<d$. Then we have $p^{-1}y_1\in L^{(d-1)}$, which concludes $L^{(d-1)}=L^{d-1}$. By using the formula (\ref{lneq}) for $r=d$, we obtain
\begin{equation*}
\length_{W}L^{(d)}/L^{(d-1)}=\length_{W}L^{d}/L^{d-1}. 
\end{equation*}
Since $L^{r_0}$ is $\Phi$-invariant, we have $L^d=L^{r_0}=L^{d-1}$. Consequently, the left-hand side is $0$, which contradicts the assumption of $d$. Hence we have $r_0=d$. 
\end{proof}

\subsection{Bruhat--Tits stratification}\label{btsn}

First, we recall the closed formal subscheme $\M_{H,\widetilde{\Lambda}}$ of $\M_H$ attached to an $H$-vertex lattice $\widetilde{\Lambda}$. It is defined as the locus of $(X,\iota,\lambda,\rho)$ such that $\rho^{-1}\circ \Lambda^{\natural}\circ \rho \subset \End(X)$. Moreover, put $\M_{H,\widetilde{\Lambda}}^{(i)}:=\M_{H,\widetilde{\Lambda}}\cap \M_H^{(i)}$. 

We define closed formal subschemes of $\M_G$ by the same manner. 

\begin{dfn}\label{clsb}
Let $\Lambda \in \Vrt$ be a vertex lattice. 
\begin{itemize}
\item We define a closed formal subschme $\M_{G,\Lambda}$ of $\M_G$ as the locus of $(X,\iota,\lambda,\rho)$ with $\rho^{-1}\circ \Lambda^{\vee} \circ \rho \subset \End(X)$. 
\item For $i\in \Z$, we set $\M_{G,\Lambda}^{(i)}:=\M_{G,\Lambda}\cap \M_G^{(i)}$. 
\end{itemize}
\end{dfn}

The bijection in Proposition \ref{pscr} induces a bijection as follows: 
\begin{align*}
p^{\Z} \backslash \M_{G,\Lambda}(\Fpbar)
&\cong \{L\colon \pi\text{-special lattice in }\L_{\Q}\mid \Lambda^{\vee} \subset L \}\\
&=\{L\colon \pi\text{-special lattice in }\L_{\Q}\mid \Lambda(L) \subset \Lambda \}. 
\end{align*}

\begin{prop}\label{btsc}
\emph{Let $\Lambda \in \Vrt$. 
\begin{enumerate}
\item We have $\M_{G,\Lambda}=\M_{H,\varphi(\Lambda)}$. 
\item We have $\M_{G,\Lambda}^{(i)}=\M_{H,\varphi(\Lambda)}^{(i)}$ for any $i\in \Z$. 
\end{enumerate}}
\end{prop}

\begin{proof}
(i): By Proposition \ref{mgeb}, $\M_{G,\Lambda}$ is the locus of $(X,\iota,\lambda,\rho)$ in $\M_{H}$ such that $\rho^{-1}\circ y_1 \circ \rho \subset \End(X)$ and $\rho^{-1}\circ \Lambda^{\vee}\circ \rho  \subset \End(X)$. This is equivalent to the condition $\rho^{-1}\circ (\varphi(\Lambda))^{\natural} \circ \rho \subset \End(X)$ since $\varphi(\Lambda)^{\natural}=\Lambda^{\vee}\oplus \Zp y_1$. 

(ii): This follows from (i) and $\M_G^{(i)}\subset \M_H^{(i)}$. 
\end{proof}

\begin{cor}\label{redv}
\emph{Let $\Lambda \in \Vrt$. 
\begin{enumerate}
\item The formal scheme $\M_{G,\Lambda}$ is a reduced scheme of characteristic $p$. 
\item We have the following: 
\begin{itemize}
\item if $\Lambda \in \Vrt(1)$, then $\M_{G,\Lambda}^{(0)}$ is a single point, 
\item if $\Lambda \in \Vrt(3)$, then $\M_{G,\Lambda}^{(0)}$ is isomorphic to $\P^1_{\Fpbar}$, 
\item if $\Lambda \in \Vrt(5)$, then $\M_{G,\Lambda}^{(0)}$ is isomorphic to the Fermat surface defined by
\begin{equation*}
x_0^{p+1}+x_1^{p+1}+x_2^{p+1}+x_3^{p+1}=0
\end{equation*}
in $\Proj \Fpbar [x_0,x_1,x_2,x_3]$. 
\end{itemize}
In particular, $\M_{G,\Lambda}^{(i)}$ is projective, smooth and irreducible of dimension $(t(\Lambda)-1)/2$ for any $i\in \Z$. 
\end{enumerate}}
\end{cor}

\begin{proof}
(i): By Proposition \ref{btsc} (i), it suffices to show that $\M_{H,\varphi(\Lambda)}$ is reduced. This follows from Theorem \ref{redh}. 

(ii): By Propositions \ref{rzis} (ii) and \ref{btsc} (ii), it suffices to show that $\M_{H,\varphi(\Lambda)}^{(0)}$ is projective, smooth and irreducible of dimension $(t(\Lambda)-1)/2$. This follows from \cite[Theorem 3.10]{hp}. 
\end{proof}

\begin{thm}\label{mgld}
\emph{Let $\Lambda_1,\Lambda_2\in \Vrt$. 
\begin{enumerate}
\item We have
\begin{equation*}
\M_{G,\Lambda_1}\cap \M_{G,\Lambda_2}=
\begin{cases}
\M_{G,\Lambda_1\cap \Lambda_2} & \text{if }\Lambda_1\cap \Lambda_2\in \Vrt, \\
\emptyset & \text{otherwise}. 
\end{cases}
\end{equation*}
\item We have $\M_{G,\Lambda_1}\subset \M_{G,\Lambda_2}$ if and only if $\Lambda_1\subset \Lambda_2$. 
\end{enumerate}}
\end{thm}

\begin{proof}
(i): We follow the proof of \cite[Proposition 4.3 (ii)]{rtw}. For $S\in \nilp_W$ and $(X,\iota,\lambda,\rho)\in \M_G(S)$, we have $\rho^{-1}\circ (\Lambda_1\cap \Lambda_2)^{\vee}\circ \rho=\rho^{-1}\circ (\Lambda_1^{\vee}+\Lambda_2^{\vee})\circ \rho \subset \End(X)$ if and only $\rho^{-1}\circ \Lambda_i^{\vee}\circ \rho \subset \End(X)$ for $i\in \{1,2\}$. Therefore, if $\Lambda_1\cap \Lambda_2\in \Vrt$, then we have $\M_{G,\Lambda_1}\cap \M_{G,\Lambda_2}=\M_{G,\Lambda_1\cap \Lambda_2}$. Next, assume $\M_{G,\Lambda_1}\cap \M_{G,\Lambda_2}\neq \emptyset$. Then there is a $\pi$-special lattice $L$ such that $L\in (\M_{G,\Lambda_1}\cap \M_{G,\Lambda_2})(\Fpbar)$. We have $\Lambda(L)\subset \Lambda_1\cap \Lambda_2$, and therefore
\begin{equation*}
p(\Lambda_1\cap \Lambda_2)\subset \Lambda_1^{\vee}+\Lambda_2^{\vee}=(\Lambda_1\cap \Lambda_2)^{\vee}\subset \Lambda(L)^{\vee}\subset \Lambda(L)\subset \Lambda_1\cap \Lambda_2, 
\end{equation*}
that is, $\Lambda_1\cap \Lambda_2\in \Vrt$. 

(ii): By definition, we have $\M_{G,\Lambda_1}\subset \M_{G,\Lambda_2}$ if $\Lambda_1\subset \Lambda_2$. On the other hand, assume $\M_{G,\Lambda_1}\subset \M_{G,\Lambda_2}$. By (i), we have $\Lambda_1\cap \Lambda_2\in \Vrt$ and 
\begin{equation*}
\M_{G,\Lambda_1\cap \Lambda_2}=\M_{G,\Lambda_1}\neq \emptyset. 
\end{equation*}
Hence it suffices to show the equality $\Lambda_1\cap \Lambda_2=\Lambda_1$. We have
\begin{equation*}
\frac{t(\Lambda_1\cap \Lambda_2)-1}{2}=\dim \M_{G,\Lambda_1\cap \Lambda_2}=\dim \M_{G,\Lambda_1}=\frac{t(\Lambda_1)-1}{2}
\end{equation*}
by Corollary \ref{redv} (ii). Hence we have $t(\Lambda_1\cap \Lambda_2)=t(\Lambda_1)$. Since $\Lambda_1\cap \Lambda_2 \subset \Lambda_1$, we obtain the desired equality. 
\end{proof}

Let us define a locally closed subscheme of $\M_G$ attached to $\Lambda \in \Vrt$ by 
\begin{equation*}
\BT_{G,\Lambda}:=\M_{G,\Lambda}\setminus \bigcup_{{\Lambda}'\subsetneq \Lambda}\M_{G,\,{\Lambda}'}. 
\end{equation*}
Then we have a bijection as follows by Proposition \ref{spvt}: 
\begin{equation*}
p^{\Z}\backslash \BT_{G,\Lambda}(\Fpbar)\cong \{L\colon \pi\text{-special lattice in }\L_{\Q}^{\pi}\mid \Lambda(L)=\Lambda \}. 
\end{equation*}
We set $\BT_{G,\Lambda}^{(i)}:=\BT_{G,\Lambda}\cap\,\M_{G}^{(i)}$ for each $i\in \Z$. 

\begin{thm}\label{btm0}
\emph{
\begin{enumerate}
\item We have a locally closed stratification
\begin{equation*}
\M_G^{(0),\red}=\coprod_{\Lambda \in \Vrt}\BT_{G,\Lambda}^{(0)}. 
\end{equation*}
Each $\BT_{G,\Lambda}^{(0)}$ is irreducible of dimension $(t(\Lambda)-1)/2$. 
\item For any $\Lambda \in \Vrt$, the closure of $\BT_{G,\Lambda}^{(0)}$ in $\M_{G}^{(0)}$ equals $\M_{G,\Lambda}^{(0)}=\coprod_{\Lambda' \subset \Lambda}\BT_{G,\Lambda}^{(0)}$.
\item The scheme $\M_G^{(0),\red}$ is connected. 
\item Let $\Irr(\M_G^{(0)})$ be the set of all irreducible components of $\M_G^{(0)}$. Then we have a bijection
\begin{equation*}
\Vrt(5)\xrightarrow{\cong} \Irr(\M_G^{(0)});\Lambda \mapsto \M_{G,\Lambda}^{(0)}. 
\end{equation*}
In particular, $\M_G^{(0)}$ is purely $2$-dimensional. 
\end{enumerate}}
\end{thm}

The stratification in Theorem \ref{btm0} (i) is called the \emph{Bruhat--Tits stratification of $\M_G$}. 

\begin{proof}
(i): This follows from Proposition \ref{spvt} and Corollary \ref{redv} (ii). 

(ii): This follows from the irreducibility of Corollary \ref{redv} (ii). 

(iii): This follows from the same argument as that of \cite[Theorem 6.4.1]{hp2}. Here we use Proposition \ref{adjv}. 

(iv): By (i) and Proposition \ref{ltic} (i), we have an equality of underlying sets
\begin{equation*}
\M_G^{(0),\red}=\bigcup_{\Lambda \in \Vrt(5)}\M_{G,\Lambda}. 
\end{equation*}
Hence $\M_{G,\Lambda}^{(0)}$ is an irreducible component of $\M_G^{(0)}$, and the map $\Lambda \mapsto \M_{G,\Lambda}^{(0)}$ is surjective. On the other hand, the injectivity of the map follows from Theorem \ref{mgld} (ii). 
\end{proof}

We further consider the geometric structure of $\M_G$. Let $\M_G^{\nfs}$ be the non-formally smooth locus of $\M_G$ over $\spf W$, and set $\M_G^{(0),\nfs}:=\M_G^{\nfs}\cap \M_G^{(0)}$. 

\begin{thm}\label{sgvl}
\emph{There is an equality
\begin{equation*}
\M_G^{(0),\nfs}=\coprod_{\Lambda \in \Vrt(1)}\M_{G,\Lambda}^{(0)},
\end{equation*}}
\end{thm}

\begin{proof}
Note that Corollary \ref{rzsg} implies that for any non-formally smooth point $x\in \M_G^{(0)}$ is $\Fpbar$-rational, and the corresponding $\Pi$-stable Dieudonn{\'e} lattice $M$ satisfies $y_1(M)\subset F^{-1}(pM)$. This is equivalent to the condition $(y_1^{-1}\circ F)(M)=M$. Indeed, the equivalence above follows from $\dim_{\Fpbar}M/y_1(M)=4=\dim_{\Fpbar}M/pF^{-1}(M)$. Here the second equality is a consequence of the Kottwitz condition. Moreover, if $L$ is the $\pi$-special lattice corresponding to $M$, then the condition $(y_1^{-1}\circ F)(M)=M$ is translated into $\pi\circ \Phi(L)=L$ by Proposition \ref{picr}. 

First, take $\Lambda \in \Vrt(1)$, and write $\M_{G,\Lambda}^{(0)}(\Fpbar)={L}$, where $L$ is a $\pi$-special lattice in $\L_{\Q}$. Then we have $L+\pi(L)=\Lambda$, and
\begin{equation*}
\pi(L)\in \M_{G,\Lambda}^{(1)}(\Fpbar)\subset p^{\Z}\backslash \M_{G,\Lambda}(\Fpbar)
\end{equation*}
by Proposition \ref{picr}. Moreover, we have 
\begin{equation*}
\pi \circ \Phi(L)=\Phi \circ \pi(L)\in \M_{G,\Lambda}^{(0)}(\Fpbar)\subset p^{\Z}\backslash \M_{G,\Lambda}(\Fpbar)
\end{equation*}
by Lemma \ref{pipr} (iv) and Proposition \ref{picr}. On the other hand, $\M_{G,\Lambda}^{(0)}(\Fpbar)$ consists of a single point by Corollary \ref{redv} (ii). Hence we obtain $\pi \circ \Phi(L)=L$, that is, $\M_{G,\Lambda}^{(0)}\subset \M_G^{(0),\nfs}$. Note that the disjointness of the right-hand side follows from Theorem \ref{mgld} (ii). 

Next, take a $\pi$-special lattice $L$ in $\L_{\Q}$ which corresponds to a point in $\M_G^{(0),\nfs}$. Then we have $\pi(L)=\Phi(L)$, which implies that $L^{0}=L+\pi(L)$ is $\Phi$-invariant since $\pi^2=\id_{\L_{\Q}}$. Hence we have $L^0\in \Vrt(1)$ and $\M_{G,L^0}^{(0)}=\{x\}$. 
\end{proof}

\begin{cor}\label{ctnb}
\emph{
\begin{enumerate}
\item Each non-formally smooth point of $\M_G^{(0)}$ is contained in $2(p+1)$-irreducible components. 
\item Each irreducible component of $\M_G^{(0)}$ contains $(p+1)(p^2+1)$-non-formally smooth points. 
\item For each irreducible component $F$ of $\M_G^{(0)}$, the number of irreducible components of $\M_G^{(0)}$ such that the intersections with $F$ are $1$-dimensional is $(p+1)(p^2+1)$. 
\item For each irreducible component $F$ of $\M_G^{(0)}$, the number of irreducible components of $\M_G^{(0)}$ which intersect at a single point is $p(p+1)(p^2+1)$. 
\end{enumerate}}
\end{cor}

\begin{proof}
(i): This is a consequence of Propositions \ref{ltic} (i), \ref{mgld} (ii) and Theorem \ref{sgvl}. 

(ii): This follows from Propositions \ref{ltic} (iii), \ref{mgld} (ii) and Theorem \ref{sgvl}. 

(iii): This is a consequence of Propositions \ref{ltic} (iii), (iv) and \ref{mgld}. 

(iv): This follows from Propositions \ref{ltic} (iii), (v) and \ref{mgld}. 
\end{proof}

\begin{rem}
\begin{enumerate}
\item We can count numbers as in Corollary \ref{ctnb} for $\M_H^{(0)}$ by the same method.
\item We modify the number of irreducible components of $\M_H^{(0)}$ which intersects at a single point for each irreducible component of $\M_H^{(0)}$ asserted in \cite[Theorem 3.12 (1)]{hp}. The precise value is $p(p^2+1)(p^3+1)$. This follows from the analogue of Proposition \ref{ltic} (v) and the fact that each irreducible component of $\M_H^{(0)}$ contains $(p^2+1)(p^3+1)$-superspecial points. Here a superspecial point is an element of $\M_H(\Fpbar)$ such that the corresponding Dieudonn{\'e} lattice $M$ in $\D_{\Q}$ satisfies $F^2(M)=pM$. 
\end{enumerate}
\end{rem}

We write vertex lattices in terms of the Bruhat--Tits building. Define a simplicial complex $\mathcal{V}$ with an action of $\SO(\L_{\Q}^{\Phi,\pi})$ as follows: 
\begin{itemize}
\item The set of vertices in $\mathcal{V}$ is the set $\Vrt(1)\sqcup \Vrt(5)$. 
\item The adjacency relation $\sim$ is given as below for distinct $\Lambda,\Lambda'\in \Vrt(1)\sqcup \Vrt(5)$: 
\begin{itemize}
\item if $t(\Lambda)=1$ and $t(\Lambda')=5$. Then $\Lambda \sim \Lambda'$ if $\Lambda \subset \Lambda'$, 
\item if $t(\Lambda)=t(\Lambda)=5$. Then $\Lambda \sim \Lambda'$ if $\length_{W} ((\Lambda+\Lambda')/\Lambda)=1$ and $\length_{W}((\Lambda+\Lambda')/\Lambda')=1$. 
\end{itemize}
\item For $m\in \{0,1,2\}$, an $m$-simplex is a subset of $(m+1)$-vertex lattices $\{\Lambda_0,\ldots ,\Lambda_m\}$ which are mutually adjacent. 
\item $\SO(\L_{\Q}^{\Phi,\pi})(\Qp)$ acts simplicially on $\mathcal{V}$ by
\begin{equation*}
\SO(\L_{\Q}^{\Phi,\pi})\times (\Vrt(1)\sqcup \Vrt(5))\rightarrow \Vrt(1)\sqcup \Vrt(5); (g,\Lambda)\mapsto g\Lambda. 
\end{equation*}
\end{itemize}

\begin{prop}\label{v3eg}
\emph{There is a bijection between the set $\Vrt(3)$ and the set of edges connecting two adjacent vertex lattices of type $5$, that is, the set of $1$-simplexes $\{\Lambda_1,\Lambda_2\}$ such that $\Lambda_1,\Lambda_2\in \Vrt(5)$. }
\end{prop}

\begin{proof}
For a $1$-simplex $\{\Lambda_1,\Lambda_2\}$ with $\Lambda_i\in \Vrt(5)$ for $i\in \{1,2\}$, we show that $\Lambda_1\cap \Lambda_2\in \Vrt(3)$. For $i\in \{1,2\}$, since $\Lambda_i\in \Vrt(5)$, we have $\Lambda_i^{\vee}=p\Lambda_i$. Hence we have 
\begin{equation*}
p(\Lambda_1 \cap \Lambda_2)\subset p(\Lambda_1+\Lambda_2)=p\Lambda_1+p\Lambda_2=(\Lambda_1\cap \Lambda_2)^{\vee}. 
\end{equation*}
On the other hand, since $\Lambda_1$ and $\Lambda_2$ are adjacent, we have
\begin{equation*}
p(\Lambda_1+\Lambda_2)\subset \Lambda_1\cap \Lambda_2. 
\end{equation*}
Hence we obtain $\Lambda_1\cap \Lambda_2\in \Vrt$. Moreover, the adjacency relation for $\Lambda_1$ and $\Lambda_2$ also implies that $\Lambda_1\cap \Lambda_2\in \Vrt(3)$. Therefore, we obtain a map
\begin{equation*}
\{\{\Lambda_1,\Lambda_2\}\colon \text{$1$-simplex}\mid \Lambda_1,\Lambda_2\in \Vrt(5)\}\rightarrow \Vrt(3);\{\Lambda_1,\Lambda_2\} \mapsto \Lambda_1\cap \Lambda_2. 
\end{equation*}
The bijectivity of the map above follows from Proposition \ref{ltic} (i). 
\end{proof}

We can also prove the assertion as below, which follows from the same argument as the proof of \cite[Proposition 2.22]{hp} by using Propositions \ref{lqpi} (iii) and \ref{v3eg}. See also \cite[20.3]{gar}. 

\begin{prop}\label{vibo}
There is an $\SO(\L_{\Q}^{\Phi,\pi})$-equivariant isomorphism between the simplicial complex $\mathcal{V}$ and the Bruhat--Tits building of $\SO(\L_{\Q}^{\Phi,\pi})$. 
\end{prop}

\subsection{Bruhat--Tits building of $J^{\ad}(\Qp)$ and Bruhat--Tits strata}\label{jbts}

We interpret the simplicial complex $\mathcal{V}$ by the isomorphism $J^{\ad}\cong \SO(\L_{\Q}^{\Phi,\pi})$ constructed in Corollary \ref{eis1}. 

First, we recall the Bruhat--Tits building of $J^{\ad}(\Qp)\cong \PGSp_4(\Qp)$. See also \cite[20.1]{gar} and \cite[Chapter 2]{fan}. We use the $4$-dimensional symplectic space $(V_0,(\,,\,)_0)$ over $\Qp$ constructed in Section \ref{grpj}. See Definition \ref{rlbl}. For a lattice $T$ in $V_0$, let $[T]$ be the homothety class of lattices in $V_0$ containing $T$. 

We define a simplicial complex $\B$ as follows: 
\begin{itemize}
\item The set of vertices consists of sets of lattices $\triangle$ in $V_0$ such that there is (necessarily unique) $T\in \triangle$ satisfying the conditions as follows: 
\begin{itemize}
\item $\triangle=[T]\cup [T^{\vee}]$,
\item $pT\subset T^{\vee}\subset T$. 
\end{itemize}
We denote by $\Vtx$ the set of all vertices. 
\item The adjacency relation $\sim$ on the set of vertices is defined as follows: for vertices $\triangle_1,\triangle_2$, we have $\triangle_1\sim \triangle_2$ if there are lattices $T_1,T_2$ in $V_0$ such that
\begin{itemize}
\item $T_i\in \triangle_i$ and $pT_i \subset T_{i}^{\vee}\subset T_i$ for $i\in \{1,2\}$,
\item $T_1\subset T_2$ or $T_2\subset T_1$. 
\end{itemize}
\item For $m\in \{0,1,2\}$, an $m$-simplex is a subset of $(m+1)$-vertices $\{\triangle_0,\ldots,\triangle_m\}$ which are mutually adjacent. 
\item The group $J(\Qp)\cong \GSp(V_0)(\Qp)$ acts simplicially on $\B$ by 
\begin{equation*}
J(\Qp)\times \Vtx\rightarrow \Vtx; (g,\triangle)\mapsto g\triangle:=\{gT\mid T\in \triangle \}. 
\end{equation*}
Since the action of the center of $J(\Qp)$ is trivial, the action of $J(\Qp)$ factors through $J^{\ad}(\Qp)$. 
\end{itemize}

For a vertex $\triangle$ of $\B$, Put $t(\triangle):=\dim_{\Fp}(T/T^{\vee})\in \Z$, where $T\in \triangle$ is the unique lattice in $V_0$ satisfying $pT\subset T^{\vee}\subset T$. Note that the number $t(\triangle)$ is independent of the choice of $T$, and we have $t(\triangle)\in \{0,2,4\}$. We call $t(\triangle)$ the \emph{type} of $\triangle$. 

\begin{dfn}
\begin{enumerate}
\item For $t\in \{0,2,4\}$, we denote by $\Vtx(t)$ the set of all vertices $\triangle$ of $\B$ satisfying $t(\triangle)=t$. 
\item Put $\Vtx^{\hs}:=\Vtx(0)\sqcup \Vtx(4)$. An element of $\Vtx^{\hs}$ is called a \emph{hyperspecial vertex}. 
\item Put $\Vtx^{\nsp}:=\Vtx(2)$. An element of $\Vtx^{\nsp}$ is called a \emph{non-special vertex}. 
\item Let $\Edg^{\hs}$ be the set of all edges connecting two adjacent hyperspecial vertices, that is, $1$-simplexes $\{\triangle'_0,\triangle'_2\}$ where $\triangle'_i\in \Vtx(2i)$ for $i=0,2$. 
\end{enumerate}
\end{dfn}

\begin{rem}
\begin{enumerate}
\item Let $\triangle \in \Vtx$ and $T\in \triangle$. Then we have $\triangle=[T]$ if and only if $\triangle \in \Vtx^{\hs}$. If $\triangle \in \Vtx^{\hs}$, then we have $\triangle=[T]\sqcup [T^{\vee}]$.  
\item There is an isomorphism of simplicial complexes between the Bruhat--Tits building of $\Sp_4(\Qp)$ and $\B$ with $\Sp(V_0)(\Qp)$-actions by sending $[T]$ to $[T]\cup [T^{\vee}]$; cf.~\cite[20.1]{gar}. 
\item Let $g\in J^{\ad}(\Qp)$ satisfying $\ord_p(\sml(g))\in \Z\setminus 2\Z$, and $T$ a lattice in $V_0$ satisfying $pT\subset T^{\vee}\subset T$ and $\dim_{\Fp}(T/T^{\vee})=2$. Then the homothety class $[gT]$ does not contain a lattice $T'$ in $V_0$ such that $pT'\subset (T')^{\vee}\subset T'$. However, $[(gT)^{\vee}]$ contains such a lattice. 
\end{enumerate}
\end{rem}

We further consider the lattices in $V_0$. 
\begin{dfn}\label{j0df}
We define a subgroup $J^0$ of $J(\Qp)$ by 
\begin{equation*}
J^0:=\{g \in J(\Qp)\mid \sml(g)\in \Zpt \}. 
\end{equation*}
\end{dfn}

\begin{prop}\label{trjv}
Let $T_0$ be a self-dual lattice in $V_0$, and $T_1$ a lattice in $V_0$ satisfying $pT_1\subsetneq T_1^{\vee}\subsetneq T_1$ (that is, $\dim_{\Fp}(T/T^{\vee})=2$). 
\begin{enumerate}
\item For any lattice $T$ in $V_0$ satisfying $T^{\vee}=p^iT$ for some $i\in \Z$, there is an element $g\in J(\Qp)$ such that $T=g(T_0)$. 
\item Let $i\in \Z$, and $T$ a lattice in $V_0$ satisfying $T^{\vee}=p^iT$. Then we have an equality
\begin{equation*}
\{g(T)\mid g\in J^0\}=\{T'\colon \text{lattice in }V_0\mid (T')^{\vee}=p^iT'\}. 
\end{equation*}
\item For a lattice $T$ in $V_0$ satisfying $pT\subsetneq T^{\vee}\subsetneq T$, there is an element $g\in J^0$ such that $T=g(T_1)$. 
\end{enumerate}
\end{prop}

To prove Proposition \ref{trjv}, let us define an algebraic group $J_{\Zp}$ over $\Zp$ as
\begin{equation*}
J_{\Zp}(R)=\{(g,c)\in \GL_{R}(T_0\otimes_{\Zp} R)\times \G_m(R)\mid (g(v),g(w))=c(v,w)\text{ for all }v,w\in V\otimes_{\Zp} R\}
\end{equation*}
for any $\Zp$-algebra $R$. Then we have $J_{\Zp}\otimes_{\Zp} \Qp \cong J$ and $J_{\Zp}\cong \GSp_4\otimes_{\Z} \Zp$. 

We endow $T_{0,\Fp}:=T_0\otimes_{\Zp}\Fp$ with the non-degenerate symplectic form $(\,,\,)_0\bmod p$ over $\Fp$. 

\begin{lem}\label{spfp}
\emph{The group $J_{\Zp}(\Fp)$ acts transitively on the set of isotropic lines in $T_{0,\Fp}$. }
\end{lem}

\begin{proof}
These follow from \cite[3.5.4]{wil}. 
\end{proof}

\begin{lem}\label{spsj}
\emph{For $n\in \Zpn$, the canonical homomorphism $\GSp_{2n}(\Zp)\rightarrow \GSp_{2n}(\Fp)$ is surjective. }
\end{lem}

\begin{proof}
The algebraic group $\GSp_{2n}\otimes_{\Z}\Zp$ is smooth over $\Zp$. Hence the assertion follows from \cite[Proposition 2.8.13]{fu}. 
\end{proof}

\begin{proof}[Proof of Proposition \ref{trjv}]
(i): This is a consequence of \cite[Corollary 7.3]{kot3}. 

(ii): First, note that $gT$ for $g\in J^0$ satisfies $(gT)^{\vee}=p^{i}gT$, since
\begin{equation*}
(gT,p^i gT)_0=\sml(g)(T,p^iT)_0=(T,p^iT)_0=\Zp. 
\end{equation*}
On the other hand, take a lattice $T'$ in $V_0$ satisfying $(T')^{\vee}=p^iT'$. By (i), there is an element $g\in J(\Qp)$ such that $T'=gT$. Then we have 
\begin{equation*}
\Zp=(T',p^iT')_0=(gT,p^igT)_0=\sml(g)(T,p^iT)_0=\sml(g)\Zp. 
\end{equation*}
Consequently we have $\sml(g)\in \Zpt$, that is, $g\in J^0$. 

(iii): Take a lattice $T$ in $V_0$ satisfying $pT\subsetneq T^{\vee}\subsetneq T$. Then there is a self-dual lattice $T'$ in $V_0$ such that $T'\subsetneq T\subsetneq p^{-1}T'$. We have $\dim_{\Fp}(pT/pT')=\dim_{\Fp}(T/T')=1$. Let us fix $T'$ as above. By (ii), there is an element $g_0\in J^0$ such that $T'=g_0T_0$. Then we have $T_0\subset g_0^{-1}T\subset p^{-1}T_0$. Hence there is an element $\gbar \in J_{\Zp}(\Fp)$ such that $g_0^{-1}(pT)/pT_0=\gbar(pT_1/pT_0)$ by Lemma \ref{spfp}. Furthermore, by Lemma \ref{spsj} there is $g\in J_{\Zp}(\Zp)$ such that $g\bmod p=\overline{g}$. Then we have $g_0^{-1}T=gT_1$, that is, $T=(g_0g)T_1$. 
\end{proof}

Now we give a reinterpretation of the Bruhat--Tits stratification of $\M_{G}^{(0)}$. 

\begin{prop}\label{vtvt}
\emph{There is an isomorphism of simplicial complexes 
\begin{equation*}
\Psi \colon \B \xrightarrow{\cong} \mathcal{V}, 
\end{equation*}
which commutes with the actions of $J^{\ad}(\Qp)\cong \SO(\L_{\Q}^{\Phi,\pi})(\Qp)$. Moreover, $\Psi$ induces correspondences as follow: 
\begin{enumerate}
\item The set $\Vtx^{\hs}$ corresponds to the set $\Vrt(5)$. 
\item The set $\Edg^{\hs}$ corresponds to the set $\Vrt(3)$. 
\item The set $\Vtx^{\nsp}$ corresponds to the set $\Vrt(1)$. 
\end{enumerate}}
\end{prop}

\begin{proof}
By Proposition \ref{vibo} and the isomorphism $J^{\ad}\cong \SO(\L_{\Q}^{\Phi,\pi})$, there is a $J^{\ad}(\Qp)$-equivariant isomorphism of simplicial complexes
\begin{equation*}
\Psi \colon \B \xrightarrow{\cong} \mathcal{V}. 
\end{equation*}
By the same argument as \cite[\S 3.6]{hp}, the assertions (i) and (ii) follows. Since $\Psi$ preserves the adjacent relations by definition, we obtain the assertion (iii). 
\end{proof}

By Proposition \ref{vtvt}, we can restate the Bruhat--Tits stratification by means of $\B$. Put $\VE:=\Vtx^{\hs}\sqcup \Vtx^{\nsp}\sqcup \Edg^{\hs}$. Note that we have a bijection $\Psi \colon \VE\cong \Vrt$ by Proposition \ref{vtvt}. We define an order $\leq$ on $\VE$ as follows for $x,y\in \VE$: 
\begin{itemize}
\item if $x,y\in \Vtx$. Then $x\leq y$ if $x\in \Vtx^{\nsp}$, $y\in \Vtx^{\hs}$ and they are adjacent. 
\item if $x\in \Vtx$ and $y\in \Edg^{\hs}$, then $x\leq y$ if $x\in \Vtx^{\nsp}$ and $\{x\}\cup y$ is a $2$-simplex. 
\item if $x\in \Edg^{\hs}$ and $y\in \Vtx$, then $x\leq y$ if $y\in x$. 
\end{itemize}

The following follows from the definition of $\leq$: 

\begin{prop}\label{odis}
\emph{The bijection $\Psi$ induces an isomorphism $(\VE,\leq)\cong (\Vrt,\subset)$ of ordered sets. }
\end{prop}

\begin{dfn}\label{mgve}
For $x\in \VE$, put 
\begin{align*}
\M_{G,x}&:=\M_{G,\Psi(x)},& \BT_{G,x}&:=\BT_{G,\Psi(x)},\\
\M_{G,x}^{(0)}&:=\M_{G,\Psi(x)}^{(0)},& \BT_{G,x}^{(0)}&:=\BT_{G,\Psi(x)}^{(0)}
\end{align*}
(see Definition \ref{clsb} for the definition of $\M_{G,\Psi(x)}$). 
\end{dfn}

The following is a restatement of the results in the previous section: 

\begin{thm}\label{mtl1}
\emph{
\begin{enumerate}
\item For $x,y\in \VE$, we have $\M_{G,x}\subset \M_{G,y}$ if and only if $x\leq y$. 
\item For $x\in \VE$, we have the following: 
\begin{itemize}
\item if $x \in \Vtx^{\nsp}$, then $\M_{G,x}^{(0)}$ is a single $\Fpbar$-rational point, 
\item if $x \in \Edg^{\hs}$, then $\M_{G,x}^{(0)}$ is isomorphic to $\P^1_{\Fpbar}$, 
\item if $x \in \Vtx^{\hs}$, then $\M_{G,x}^{(0)}$ is isomorphic to the Fermat surface defined by
\begin{equation*}
x_0^{p+1}+x_1^{p+1}+x_2^{p+1}+x_3^{p+1}=0
\end{equation*}
in $\Proj \Fpbar [x_0,x_1,x_2,x_3]$. 
\end{itemize}
In particular, $\M_{G,x}^{(0)}$ is projective, smooth and irreducible. 
\item Put $\BT_{G,x}^{(0)}:=\BT_{G,x}\cap \M_G^{(0)}$ for $x\in \VE$. Then we have a locally closed stratification
\begin{equation*}
\M_G^{(0),\red}=\coprod_{x\in \VE}\BT_{G,x}^{(0)}
\end{equation*}
(we also call the equality above the Bruhat--Tits stratification). In particular, $\M_G^{(0),\red}$ is connected and is purely $2$-dimensional.  
\item There is an equality
\begin{equation*}
\M_G^{(0),\nfs}=\coprod_{x\in \Vtx^{\nsp}}\M_{G,x}^{(0)}. 
\end{equation*}
\item For distinct $x,y\in \Vtx^{\hs}$, the following hold: 
\begin{itemize}
\item if $x\cup y$ is a $2$-simplex, then $\M_{G,x}^{(0)}\cap \M_{G,y}^{(0)}=\M_{G,x\cup y}^{(0)}$,
\item if $x,y$ are not adjacent and there is (necessarily unique) $z\in \Vtx^{\nsp}$ such that $x\cup z,z\cup y$ are $2$-simplexes, then $\M_{G,x}^{(0)}\cap \M_{G,y}^{(0)}=\M_{G,z}^{(0)}$,
\item if $x,y$ do not satisfy the above two conditions, then $\M_{G,x}^{(0)}\cap \M_{G,y}^{(0)}=\emptyset$. 
\end{itemize}
\end{enumerate}}
\end{thm}

\begin{proof}
(i): This follows from Theorem \ref{mgld} (ii) and Proposition \ref{odis}. 

(ii): This follows from Corollary \ref{redv} (ii) and Proposition \ref{odis}. Note that we can also prove the connectedness of $\M_G^{(0)}$ and the assertion on the dimension of $\M_G$ by the Bruhat--Tits stratification and the connectedness of $\B$. 

(iii): This follows from Theorem \ref{btm0} (i), (ii) and Proposition \ref{odis}. 

(iv): This follows from Theorem \ref{sgvl} and Proposition \ref{vtvt} (iii). 

(v): This follows from Theorem \ref{mgld} (i) and Proposition \ref{odis}. 
\end{proof}

Finally, we consider the $J(\Qp)$-action on $\M_{G}^{\red}$. 

\begin{prop}\label{jact}
\emph{
\begin{enumerate}
\item Let $\Irr(\M_G^{\red})$ be the set of all irreducible components of $\M_G^{\red}$, and take $\triangle_i\in \Vtx(2i)$ for $i\in \{0,2\}$. Then $\{\M_{G,\triangle_0}^{(0)},\M_{G,\triangle_2}^{(0)}\}$ is a set of complete representatives of $J(\Qp)\backslash \Irr(\M_G^{\red})$. In particular, we have $\#(J(\Qp)\backslash \Irr(\M_G^{\red}))=2$. 
\item The group $J(\Qp)$ acts transitively on the set $\M_G^{\nfs}$. 
\end{enumerate}}
\end{prop}

\begin{proof}
(i): Take two irreducible components $F_1,F_2$ of $\M_G^{\red}$. For $i=1,2$, if $F_i\subset \M_G^{(n_i)}$ for some $n_i\in \Z$ then $g_p^{-n_i}F_i$ is an irreducible component of $\M_G^{(0),\red}$. Hence we may assume $F_i\subset \M_G^{(0),\red}$. Now write $F_i=\M_{G,\triangle'_i}^{(0)}$, where $\triangle'_i\in \Vtx^{\hs}$. Take $T_i\in \triangle'_i$ satisfying $pT_i\subset T_i^{\vee}\subset T_i$. Then $F_1$ and $F_2$ are in the same orbit if and only if there is $g\in J^0$ such that $g\triangle'_1=\triangle'_2$, which is equivalent to $gT_1=T_2$. Therefore the assertion follows from Proposition \ref{trjv} (ii). 

(ii): By Theorem \ref{mtl1} (iv), we have an $J^0(\Qp)$-equivariant isomorphism $\M_G^{(0),\nfs}\cong \Vtx^{\nsp}$. Hence the assertion follows from Proposition \ref{trjv} (iii). 
\end{proof}

\begin{rem}
The finiteness of $J(\Qp)\backslash \Irr(\M_G^{\red})$ is already obtained by \cite[Proposition 2.8, Theorem 1.1]{mie}.
\end{rem}

\section{Deligne--Lusztig varieties and Bruhat--Tits stratification}\label{dlvs}

In this section, we relate the Bruhat--Tits strata with some Deligne--Lusztig varieties. 

\subsection{Generalized Deligne--Lusztig varieties for odd special orthogonal groups}
Let $G_0$ be a split reductive group over $\Fp$. Fix a split maximal torus $T_0$ of $G_0$ and a Borel subgroup $B_0$ containing $T_0$. Put $G:=G_0\otimes_{\Fp} \Fpbar,\,T:=T_0\otimes_{\Fp} \Fpbar$ and $B:=B_0\otimes_{\Fp} \Fpbar$. We denote the Frobenius of $G$ by $\overline{\Phi}$. Let $W$ be the Weyl group and $\Delta^{*}=\{\alpha_1,\ldots ,\alpha_n\}$ be the simple roots corresponding to $(T,B)$. Let $s_i:=s_{\alpha_i}\in W$ be the simple reflection corresponding to $\alpha_i$. For $I\subset \Delta^{*}$, let $W_I$ be the subgroup of $W$ generated by $\{s_i\in W\mid i\in I\}$ and $P_I:=BW_IB$ the corresponding parabolic subgroup. Then we have an equality
\begin{equation*}
G=\coprod_{w\in W_I\backslash W/W_I}P_IwP_I. 
\end{equation*}
Hence we obtain the relative position map
\begin{equation*}\label{ivmp}
\inv \colon G/{P_I}\times G/{P_I}\rightarrow W_I\backslash W/W_I
\end{equation*}
by the composite of $G/{P_I}\times G/{P_I}\rightarrow P_I\backslash G/P_I$ defined by $(g_1,g_2)\mapsto g_1^{-1}g_2$ and the bijection $P_I\backslash G/P_I\cong W_I\backslash W/W_I$ coming from the equality above. 

\begin{dfn}
For $I\subset \Delta^{*}$ and $w\in W_I\backslash W/W_I$, we define a \emph{generalized Deligne--Lusztig variety} as a locally closed subscheme of $G/P_{I}$ which is defined by
\begin{equation*}
X_{P_I}(w):=\{g\in G/P_I\mid \inv(g,\overline{\Phi}(g))=w\}. 
\end{equation*}
\end{dfn}

We keep the notations above. We consider the case where $G_0$ is an odd special orthogonal group. See also \cite[Example 4.1.3]{wu}. Let $m\in \Znn$ and $(\Omega_0,[\,,\,])$ a $(2m+1)$-dimensional quadratic space over $\Fp$. 
\begin{dfn}
A basis $e_1,\ldots ,e_{2m+1}$ of $\Omega_0$ is \emph{elementary} if $e_{m+1}$ is anisotropic and
\begin{equation*}
[e_i,e_{j}]=\begin{cases}
1 &\text{if $i+j=2m+2$ and }i\neq m+1,\\
0 &\text{otherwise}.
\end{cases}
\end{equation*}
\end{dfn}
By \cite[Lemma 22.3]{shi}, there is an elementary basis of $\Omega_0$. We fix such a basis $e_1,\ldots ,e_{2m+1}$ of $\Omega_0$. Put $G_0:=\SO(\Omega_0)$, which is a split reductive group over $\Fp$. We regard a subgroup of $\GL_{2m+1}$ by the basis $e_1,\ldots ,e_{2m+1}$. Let $T_0\subset G_0$ be the diagonal torus, and $B_0\subset G_0$ the upper-triangular Borel subgroup. Then the Weyl group $W$ can be identified with the subgroup of the symmetric group $\mathfrak{S}_{2m+1}$: 
\begin{equation*}
W=\{w\in \mathfrak{S}_{2m+1}\mid w(i)+w(2m+2-i)=2m+2\text{ for any }i\in \{1,\ldots ,2m+1\}\}. 
\end{equation*}
The set of all simple reflections in $W(T)$ is $\{s_1,\ldots ,s_m\}$, where
\begin{itemize}
\item $s_i=(i\ i+1)(2m+1-i\ 2m+2-i)$ for $1\leq i\leq m-1$,
\item $s_m=(m\ m+2)$. 
\end{itemize}
For $0\leq i\leq m$, we define 
\begin{equation*}
w_i:=s_m\cdots s_{m+1-i}. 
\end{equation*}
Then we have $w_0=\id$, and $w_m$ is a Coxeter element. Moreover, for $0\leq i\leq m-1$, let $I_i:=\{1,\ldots,m-1-i\}$, $W_i:=W_{I_i}$ and $P_i:=BW_iB$ the corresponding parabolic subgroup. Then we have $P_{m-1}=B$. We put $P_m:=B$ by convention. Moreover, $G/P_0$ is isomorphic to the Grassmanian of all totally isotropic subspaces of dimension $m$ of $\Omega:=\Omega_0\otimes_{\Fp}\Fpbar$. Now, we put $\overline{\Phi}:=\id_{\Omega_0}\otimes \sigma$. 

We introduce the following notation. We also use it in Section \ref{nseo}. 

\begin{dfn}\label{smgn}
Let $\overline{F} \colon \overline{N}\rightarrow \overline{N}$ be a $\sigma$-linear isomorphism, and $R$ an $\Fp$-algebra. For a $R$-submodule $\overline{N}'$ of $\overline{N}\otimes_{\Fpbar}R$, we denote by $\overline{F}_*(\overline{N}')$ the $R$-submodule of $\overline{N}\otimes_{\Fpbar}R$ generated by $\overline{F}(\overline{N}')$. Note that it is locally free of rank $\rk_R(\overline{N}')$ if $\overline{N}'$ is locally free. 
\end{dfn}

Let us consider the closed subscheme of $G/P_0$ as below: 
\begin{equation*}
S_{\Omega_0}:=\{L\in G/P_0 \mid \rk(L\cap \overline{\Phi}_*(L))\geq m-1\}. 
\end{equation*}
The following follows from \cite[Corollary 2.4.6]{hlz} and the isomorphism $X_{\Sigma_{i}^{\sharp},w_i}\xrightarrow{\cong}X_{J,w_i}$ in the proof of \cite[Proposition 2.5.1]{hlz} for $(\G,J,\mathcal{L})=(B_n,\mathbb{S}-\{s_n\},(s_1,\ldots,s_n))$ in \cite[Table 1]{hlz}. 

\begin{prop}\label{stsv}
\emph{
There is a locally closed stratification of $S_{\Omega_0}$: 
\begin{equation*}
S_{\Omega_0}=\coprod^{m}_{i=0}X_{P_i}(w_i).
\end{equation*}
Moreover, for $0\leq i\leq m$, the closure of $X_{P_i}(w_i)$ in $S_{\Omega_0}$ equals $\coprod_{j=}^{i}X_{P_i}(w_j)$. 
}
\end{prop}

\begin{prop}\label{slir}
\emph{The schemes $X_B(w_m)$ and $S_{\Omega_0}$ are irreducible of dimension $m$. }
\end{prop}

\begin{proof}
We have $\dim X_{P_i}(w_i)=i$ and $X_B(w_m)$ is irreducible by \cite[Proposition 3.2]{hp}. By Proposition \ref{stsv}, $X_B(w_m)$ is dense in $S_{\Omega_0}$, which concludes the irreducibility of $S_{\Omega_0}$. 
\end{proof}

\subsection{Relation with the case for non-split even orthogonal groups}\label{nseo}

Let $(\Omegabar_0,[\,,\,])$ be a $(2m+2)$-dimensional non-split quadratic space over $\Fp$.  Furthermore, put $\Omegabar:=\Omegabar_0\otimes_{\Fp}\Fpbar$ and $\overline{\Phi}:=\id_{\Omegabar_0}\otimes \sigma$. For $d\in \{m,m+1\}$, let $\OGr(d;\Omegabar)$ be the moduli space of totally isotropic subspaces of rank $d$ of $\Omegabar$ over $\Fpbar$. On the other hand, let $\OGr(m,m+1;\Omegabar)$ be the moduli space which parametrizes flags of $\O_S$-submodules $L_{m}\subset L_{m+1}$ of $\Omegabar \otimes_{\Fpbar}\O_S$ such that $L_{i}$ is locally free of rank $i$ for $i\in \{m,m+1\}$, for any $\Fpbar$-scheme $S$. The scheme $\OGr(m,m+1;\Omegabar)$ has exactly two connected components. Denote by $\OGr^{\pm}(m,m+1;\Omegabar)$ the connected components of $\OGr(m,m+1;\Omegabar)$. Moreover, put
\begin{equation*}
\Sbar_{\Omegabar_0}:=\{L\in \OGr(m+1;\Omegabar_0)\mid \rk(L\cap \overline{\Phi}(L))=m\}. 
\end{equation*}
We have a closed immersion 
\begin{equation*}
i_{\Omegabar_0}\colon \Sbar_{\Omegabar_0}\rightarrow \OGr(m,m+1;\Omegabar);L\mapsto (L\cap \overline{\Phi}(L)\subset L). 
\end{equation*}
Moreover, put
\begin{equation*}
\Sbar_{\Omegabar_0}^{\pm}:=\Sbar_{\Omegabar_0} \cap \OGr^{\pm}(m,m+1;\Omegabar). 
\end{equation*}

The schemes $\Sbar_{\Omegabar_0}^{\pm}$ equal the schemes $\mathscr{X}^{\pm}$ for $d=m+1$ in the sense of \cite[\S 3.2]{hp}. Hence we have the following: 

\begin{prop}\label{suir} (\cite[Proposition 3.6]{hp})
\emph{The schemes $\Sbar_{\Omegabar_0}^{\pm}$ are projective, smooth and irreducible of dimension $m$. }
\end{prop}

We relate $\Sbar_{\Omegabar_0}^{\pm}$ with $S_{\Omega_0}$. Fix an anisotropic vector $\omega \in \Omegabar_0$, and let $\pibar$ be the reflection with respect to $\omega$. Set $\Omega_0:=\Omegabar_0^{\pibar}$. It is a $(2m+1)$-dimensional quadratic space over $\Fp$. 

We endow $\OGr(d;\Omegabar)$ for $d\in \{m,m+1\}$ and $\Sbar_{\Omegabar_0}$ with an $\Fp$-structure defined by $L\mapsto (\overline{\Phi} \circ \pibar)_*(L)$. Moreover, we endow $\OGr(m,m+1;\Omegabar_0)$ with an $\Fp$-structure defined by $(L_m\subset L_{m+1})\mapsto ((\overline{\Phi} \circ \pibar)_*(L_m)\subset (\overline{\Phi} \circ \pibar)_*(L_{m+1}))$. Then $i_{\Omegabar_0}$ commutes with the $\Fp$-structures. 

\begin{prop}\label{svsu}
\emph{
\begin{enumerate}
\item We have $(\overline{\Phi} \circ \pibar)_*(\OGr^{\pm}(m,m+1;\Omegabar))\subset \OGr^{\pm}(m,m+1;\Omegabar)$ and $(\overline{\Phi} \circ \pibar)_*(\Sbar_{\Omegabar_0}^{\pm})\subset \Sbar_{\Omegabar_0}^{\pm}$. 
\item For a reduced $\Fpbar$-algebra $R$ of finite type and $L\in \OGr(m+1;\Omegabar_0)(R)$, we have $L\cap \pibar(L)\subset \Omega \otimes_{\Fpbar}R$. Moreover, $L\cap \pibar(L)$ is totally isotropic and locally free of rank $m$ over $R$. Therefore we obtain the morphism 
\begin{equation*}
q\colon \OGr(m+1;\Omegabar_0) \rightarrow G/P_0;L\mapsto L\cap \pibar(L)
\end{equation*}
over $\Fpbar$. Moreover, $q$ commutes with $\Fp$-structures. 
\item The morphism $q$ induces an isomorphism $\Sbar_{\Omegabar_0}^{\pm}\cong S_{\Omega_0}$, which commutes with the $\Fp$-structures. 
\end{enumerate}}
\end{prop}

\begin{proof}
(i): First, we prove $(\overline{\Phi} \circ \pibar)_*(\OGr^{\pm}(m,m+1;\Omegabar))\subset \OGr^{\pm}(m,m+1;\Omegabar)$. Fix an elementary basis $e'_1,\ldots,e'_{2m+1}$ of $\Omega_0$, and put $c:=-[e'_{2m+1},e'_{2m+1}]\cdot [\omega,\omega]^{-1}$ (note that $[\omega,\omega]\neq 0$ since $\omega$ is anisotropic). Then we have $c\in \Fpt \setminus (\Fpt)^2$ since $\Omegabar_0$ is non-split. Fix a square root $\overline{\varepsilon}\in \Fpst$ of $c$. Then we have $\sigma(\overline{\varepsilon})=-\overline{\varepsilon}$. Now fix a $(p+1)$-th root $a\in \Fpst$ of $2[e'_{m+1},e'_{m+1}]^{-1}$ (this is possible since $\Fpst$ is a cyclic group of order $p^2-1$ and $\# \Fpt=p-1$), and define $e_i$ and $f_i$ for $1\leq i\leq m+1$ as follow: 
\begin{equation*}
e_i:=\begin{cases}
e'_i &\text{if }1\leq i\leq m,\\
a(e'_{m+1}+\overline{\varepsilon}\omega) &\text{if } i=m+1,
\end{cases}
,\quad f_i:=\begin{cases}
e'_{2m+2-i} &\text{if }1\leq i\leq m,\\
a^{p}(e'_{m+1}-\overline{\varepsilon}\omega) &\text{if }i=m+1. 
\end{cases}
\end{equation*}
Then we have the following: 
\begin{gather*}
[e_i,f_j]=\delta_{ij}\,(i,j\in \{1,\ldots ,m+1\}),\\ 
\overline{\Phi}(e_i)=
\begin{cases}
e_i &\text{if }1\leq i\leq m,\\
f_{m+1}&\text{if } i=m+1,
\end{cases}
,\quad \overline{\Phi}(f_i)=
\begin{cases}
f_i &\text{if }1\leq i\leq m,\\
e_{m+1} &\text{if } i=m+1, 
\end{cases}
\\ \pibar(e_i)=
\begin{cases}
e_i &\text{if }1\leq i\leq m,\\
f_{m+1}&\text{if } i=m+1. 
\end{cases}
\end{gather*}
We define subspaces of $\Omegabar$ as below: 
\begin{equation*}
L:=\bigoplus_{i=1}^{m}\Fpbar e_i,\quad L^{+}:=L\oplus e_{m+1},\quad L^{-}:=L\oplus f_{m+1}. 
\end{equation*}
Then we have $(L\subset L^{\pm})\in \OGr(m,m+1;\Omegabar)(\Fpbar)$. Moreover, by after \cite[Lemma 3.4]{hp}, after relabeling the connected components of $\OGr(m,m+1;\Omegabar)$ if necessary, we may assume $(L\subset L^{\pm})\in \OGr^{\pm}(m,m+1;\Omegabar)(\Fpbar)$. On the other hand, we have $((\overline{\Phi} \circ \pibar)_*(L)\subset (\overline{\Phi} \circ \pibar)_*(L^{\pm}))=(L\subset L^{\pm})$ by the relations between $e_i,f_j$ and $\overline{\Phi},\pibar$. Hence we obtain $(\overline{\Phi} \circ \pibar)_*(\OGr^{\pm}(m,m+1;\Omegabar))\subset \OGr^{\pm}(m,m+1;\Omegabar)$ by the connectedness of $\OGr^{\pm}(m,m+1;\Omegabar)$. 

Second, we prove $(\overline{\Phi} \circ \pibar)_*(\Sbar_{\Omegabar_0}^{\pm})\subset \Sbar_{\Omegabar_0}^{\pm}$. This follows from $(\overline{\Phi} \circ \pibar)_*(\OGr^{\pm}(m,m+1;\Omegabar))\subset \OGr^{\pm}(m,m+1;\Omegabar)$ and the commutativity of $\overline{\Phi}$ and $\pibar$. 

(ii): Note that we have $L\not\subset \Omega \otimes_{\Fpbar}R$. By the similar argument to the proof of Lemma \ref{y1il}, we have $L\cap \pibar(L)=L^{\pibar}\oplus L^{-\pibar}$, where $L^{\pibar}$ and $L^{-\pibar}$ are the $\pibar$ and $-\pibar$-fixed parts in $L$ respectively. Then we have $L^{-\pibar}\subset R\omega$ and both $L^{\pibar}$ and $L^{-\pibar}$ are totally isotropic. Let us show that $L^{\pibar}$ is locally free of rank $m$ over $R$ and $L^{-\pibar}=0$. First, suppose $R=\Fpbar$. Then, $L^{-\pibar}=0$ follows since $\omega$ is anisotropic. Hence we have $\dim_{\Fpbar}L^{\pibar}=\dim_{\Fpbar}L\cap\pibar(L)$. By the similar argument to the proof of Lemma \ref{y1il}, we have $\dim_{\Fpbar}L\cap\pibar(L)=m$, which implies the assertion. Next, for general $R$, the assertion follows from the case for $R=\Fpbar$, \cite[Exercise X.16]{lan} for $\pibar-\id_{\Omega \otimes_{\Fpbar}R}$, and the property that $R$ is a Jacobson ring. Finally, the commutativity of $q$ and $\Fp$-structures is a consequence of the commutativity of $\pibar$ and $\overline{\Phi}$. 

(iii): First, note that $q\vert_{\Sbar_{\Omegabar_0}^{\pm}}$ commutes with the $\Fp$-structures by (i) and (ii). We show that $q\vert_{\Sbar_{\Omegabar_0}^{\pm}}$ are closed immersions. By (ii), we have a morphism
\begin{equation*}
i_{\pibar}\colon \Sbar_{\Omegabar_0}\rightarrow \OGr(m,m+1;\Omegabar);L\mapsto (L\cap \pibar(L)\subset L), 
\end{equation*}
which is injective as a functor. It is a closed immersion since both $\Sbar_{\Omegabar_0}$ and $\OGr(m,m+1;\Omegabar)$ are proper. See Proposition \ref{suir}. Relabeling if necessary, we may assume $i_{\pibar}(\Sbar_{\Omegabar_0}^{\pm})\subset \OGr^{\pm}(m,m+1;\Omegabar)$. Moreover, we have isomorphisms 
\begin{equation*}
c^{\pm}\colon \OGr^{\pm}(m,m+1;\Omegabar)\xrightarrow{\cong} \OGr(m;\Omegabar);(L_{m}\subset L_{m+1})\mapsto L_m
\end{equation*}
by after \cite[Lemma 3.4]{hp}. Moreover, by (ii), $c^{\pm}\circ (i_{\pibar}\vert_{\Sbar_{\Omegabar_0}^{\pm}})$ factor as $\Sbar_{\Omegabar_0}^{\pm}\xrightarrow{q} G/P_0\rightarrow \OGr(m;\Omegabar)$. Hence the assertion follows. 

Next, we show that $q\vert_{\Sbar_{\Omegabar_0}^{\pm}}$ factors through the closed immersion $S_{\Omega_0}\rightarrow G/P_0$. Take a reduced $\Fpbar$-algebra $R$ of finite type and $L\in \Sbar_{\Omegabar_0}(R)$. By (ii), we have $L\cap \pibar(L)\in (G/P_0)(R)$. On the other hand, since $L\cap \overline{\Phi}_*(L)$ is locally free over $R$, 
\begin{equation*}
(L\cap \pibar(L))\cap \overline{\Phi}_*(L\cap \pibar(L))=(L\cap \overline{\Phi}_*(L))\cap \pibar(L\cap \overline{\Phi}_*(L))
\end{equation*}
is also locally free over $R$ by the injectivity of $\pibar$ and \cite[Exercise X.16]{lan} for $\pibar-\id_{\Omega \otimes_{\Fpbar}R}$. Moreover, by the same argument as in the proof of (ii), we have 
\begin{equation*}
\rk (L\cap \overline{\Phi}_*(L))\cap \pibar(L\cap \overline{\Phi}_*(L))=
\begin{cases}
m&\text{if }L\cap \overline{\Phi}_*(L)\subset \Omega \otimes_{\Fpbar}R, \\
m-1&\text{otherwise}. 
\end{cases}
\end{equation*}
Hence the assertion follows. 

Finally, we show that $q\vert_{\Sbar_{\Omegabar_0}^{\pm}}\colon \Sbar_{\Omegabar_0}^{\pm}\rightarrow S_{\Omega_0}$ are isomorphisms. By Propositions \ref{stsv} (i), \ref{slir} and \ref{suir}, both $\Sbar_{\Omegabar_0}^{\pm}$ and $S_{\Omega_0}$ are integral of dimension $m$. Therefore the assertion follows. 
\end{proof}

\subsection{Some lower-dimensional cases}\label{lwdm}

We describe $S_{\Omega_0}$ for $m=0,1$ and $2$. 

\textbf{Case 1. $m=0$.} In this case, $S_{\Omega_0}$ is a single point defined over $\Fp$. 

\textbf{Case 2. $m=1$.} By definition, we have $P_0=P_1=B$ and $S_{\Omega_0} \cong G/B\cong \P^{1}_{\Fpbar}$. The set of all $\Fp$-rational points corresponds to $X_{B}(\id)$ in $S_{\Omega_0}$, and the complement corresponds to $X_{P_0}(w_1)$ in $S_{\Omega_0}$. Moreover, by the identity of Dynkin diagrams $B_1=A_1$ there is a surjection $\GL_2\rightarrow \SO(\Omega)$ such that the upper-triangular Borel subgroup $B'$ of $\GL_2$ surjects onto $B$. Let $w'$ be the unique non-trivial element of the Weyl group of $(B,T)$, where $T$ is the diagonal torus. Then we have $X_{B}(\id)=X_{B'}(\id)$ and $X_{B}(w_1)=X_{B'}(w')$. 

\textbf{Case 3. $m=2$.} By definition, we have $P_0\neq P_1=B$. Consider the $4$-dimensional symplectic space $\overline{V}_0$ over $\Fp$, and denote by $\GSp(\overline{V}_0)$ as the symplectic similitude group of $\overline{V}_0$. Then $\GSp(\overline{V}_0)$ is an algebraic group over $\Fp$ which is isomorphic to $\GSp_4$. By the identity of Dynkin diagrams $B_2=C_2$, we have a surjection $\GSp_4\rightarrow \SO(\Omega_0)$ such that the Klingen parabolic subgroup $P'$ of $\GSp(\overline{V}_0)$ surjects onto $P_0$. This induces an isomorphism
\begin{equation*}
S_{\Omega_0}\cong \{l\in \P(\overline{V}_0)\mid \sigma(l)\subset l^{\perp}\}. 
\end{equation*}
The right-hand side is isomorphic to the surface defined by 
\begin{equation}\label{frvr}
x_0x_3^{p}-x_0^{p}x_3+x_1x_2^{p}-x_1^{p}x_2=0. 
\end{equation}
in $\Proj \Fpbar [x_0,x_1,x_2,x_3]$. Note that Proposition \ref{svsu} (ii) gives an $\Fpbar$-isomorphism (however not defined over $\Fp$) between (\ref{frvr}) and the Fermat surface defined by
\begin{equation*}
x_0^{p+1}+x_1^{p+1}+x_2^{p+1}+x_3^{p+1}=0
\end{equation*}
in $\Proj \Fpbar [x_0,x_1,x_2,x_3]$. 

Let $B'$ be the upper-triangular Borel subgroup of $\GSp(\overline{V}_0)$. The Weyl group can be identified with a subgroup of the symmetric group $\mathfrak{S}_4$: 
\begin{equation*}
\{w'\in \mathfrak{S}_4\mid w'(i)+w'(5-i)=5\text{ for any }i\in \{1,\ldots ,4\}\}. 
\end{equation*}
Denote $s'_1$ by the simple reflection which corresponds to $(1\ 2)(3\ 4)$. Also denote $s'_2$ by the simple reflection which corresponds to $(2\ 3)$. Then we have $X_{P'}(\id)\cong X_{P_0}(\id)$, $X_{B'}(s'_1)\cong X_{P_0}(w_1)$ and $X_{B'}(s'_1s'_2)\cong X_{P_0}(w_2)$. 

\subsection{Relation with the Bruhat--Tits strata of $\M_G$}\label{relb}

In this subsection, we relate the Bruhat--Tits stratification with Deligne--Lusztig varieties by using the results in the previous subsections. 

\begin{dfn}\label{ombd}
For $\widetilde{\Lambda} \in \Vrt_H$, put 
\begin{equation*}
\Omegabar_0(\Lambda):=\widetilde{\Lambda}/\widetilde{\Lambda}^{\natural}. 
\end{equation*}
It is a $\Fp$-vector space. We endow $\Omegabar_0(\Lambda)$ with quadratic form $v\mapsto -p\varepsilon^2 Q(v)\bmod p$. 
\end{dfn}
We apply the results in Section \ref{nseo} to $\Omegabar_0:=\Omegabar_0(\varphi(\Lambda))$ and $\omega:=p^{-1}y_1$. Then we have $\pibar=\pi$ on $\Omegabar_0$, and $\Omega_0$ is isomorphic to the quadratic space $\Omega_0(\varphi(\Lambda))$; cf.~Definition \ref{omlm}. Moreover, we have $\overline{\Phi}=\Phi \bmod \varphi(\Lambda)^{\natural}$ on $\Omegabar_0(\Lambda)\otimes_{\Fp}\Fpbar$ and $\Omega_0(\Lambda)\otimes_{\Fp}\Fpbar$. The following follow from \cite[Theorem 3.9]{hp} and Propositions \ref{btsc}, \ref{svsu} (ii): 

\begin{thm}\label{vldl}
\emph{
\begin{enumerate}
\item There is an isomorphism $\M^{(0)}_{G,\Lambda}\cong S_{\Omega_0(\Lambda)}$, which commutes with the bijection in Proposition \ref{pscr}. 
\item The isomorphism in (i) induces an isomorphism
\begin{equation*}
X_{P_i}(w_i)\cong \coprod_{\Lambda' \in \Vrt(2i+1),\Lambda'\subset \Lambda}\BT_{G,\Lambda'}^{(0)}
\end{equation*}
for any $0\leq i\leq (t(\Lambda)-1)/2$. In particular the dense open subscheme $\BT_{G,\Lambda}^{(0)}$ of $\M_{G,\Lambda}^{(0)}$ is isomorphic to the Deligne--Lusztig variety for $\SO_{t(\Lambda)}$ associated with a Coxeter element. 
\end{enumerate}}
\end{thm}

\begin{proof}
(i): This is a consequence of Propositions \ref{clsb} (ii), \ref{svsu} (ii) and \cite[Theorem 3.9]{hp}. 

(ii): This follows from the same argument as in the proof of \cite[Theorem 3.11]{hp}. 
\end{proof}

Combining Theorem \ref{vldl} with Proposition \ref{odis} and the results in Section \ref{lwdm}, we obtain the following: 

\begin{cor}\label{dbst}
\emph{Let $x\in \VE$. Then $\BT_{G,x}^{(0)}$ is isomorphic to the Deligne--Lusztig variety for $\GSp_{2d(x)}$ associated with a Coxeter element, where 
\begin{equation*}
d(x)=
\begin{cases}
2 &\text{if }x\in \Vtx^{\hs},\\
1 &\text{if }x\in \Edg^{\hs},\\
0 &\text{if }x\in \Vtx^{\nsp}. 
\end{cases}
\end{equation*}}
\end{cor}

\section{Application to some Shimura varieties}\label{shvr}

In this section, we apply results for $\M_G$ to the Shimura varieties for quaternionic unitary groups of degree $2$. First, we review integral models of such Shimura varieties; cf.~\cite[Chapter 6]{rz}. Next, we recall the definition of the supersingular locus, and study the non-smooth locus of the integral model by using the local model. Finally, we show the global results by using the $p$-adic uniformization theorem and the results for $\M_G$. 

\subsection{Integral models of Shimura varieties for quaternionic unitary groups}\label{shqt}

Let $\bD$ be an indefinite quaternion algebra over $\Q$ which is ramified at $p$. Define a set $\Ram(\bD)$ as the set of all prime numbers which ramifies in $\bD$. Note that $\Ram(\bD)$ is a finite set and $2\mid \#\Ram(\bD)$. 

\begin{lem}\label{dele}
There are two elements $\delta,e\in \Q^{\times}$ satisfying the following conditions: 
\begin{itemize}
\item $\bD=\Q(\varepsilon)[\Delta]$, where $\Delta^2=\delta,\,\varepsilon^2=e$ and $\Delta \varepsilon=-\varepsilon \Delta$, 
\item $\ord_p(\delta)=1$ and $\ord_p(e)=0$, 
\item $e<0<\delta$. 
\end{itemize}
\end{lem}

\begin{proof}
Since $\bD$ is ramified at $p$, there are $\delta,e\in \Q^{\times}$ satisfying the first condition and $\ord_p(\delta)=1$, $\ord_p(e)\in \{0,1\}$. If $\ord_p(\delta)=1$, then we can replace $e$ to $-\delta^{-1}e$, and hence we may assume $\ord_p(e)=0$. Next, if $\delta<0$ then we have $e>0$ since $\bD$ is indefinite. Moreover, we may replace $\delta$ to $-e\delta>0$ since $-e\in \N_{\Q(\sqrt{e})/\Q}(\Q(\sqrt{e})^{\times})$. Therefore we may assume $\delta>0$. Finally, if $e>0$ then by $\delta>0$ there is an element $\alpha \in \N_{\Q(\sqrt{\delta})/\Q}(\Q(\sqrt{\delta})^{\times})$ such that $\alpha<0$ and $\ord_p(\alpha)=0$. Therefore we can replace $e$ to $\alpha e<0$. 
\end{proof}

Let us fix $\delta,e\in \Q^{\times}$ as in Lemma \ref{dele}. Let $*$ be the involution of $\bD$ defined by
\begin{equation*}
d^{*}:=\varepsilon \overline{d}\varepsilon^{-1}. 
\end{equation*}
It is a positive involution by the assumption on $\varepsilon$ and \cite[\S 21, Theorem 2]{mum}. Let $O_{\bD}$ be an order of $\bD$ which is stable under $*$ and $O_{\bD}\otimes_{\Z}\Zp$ is a maximal order of $\bD \otimes_{\Q}\Qp$. Put $\bV:=\bD^{\oplus 2}$ as a left $\bD$-module and let $(\,,\,)^{\sim}$ be a bilinear form on $\bV$ defined by
\begin{equation*}
((x_1,x_2),(y_1,y_2))^{\sim}:=\trd_{\bD/\Q}(\Delta^{-1}(x_1^*y_2-x_2^*y_1))
\end{equation*}
for $(x_1,x_2),(y_1,y_2)\in \bV$. Then $(\,,\,)^{\sim}$ is a non-degenerate symplectic form satisfying the following conditions: 
\begin{itemize}
\item $(dx,y)^{\sim}=(x,d^{*}y)^{\sim}$ for any $d\in \bD$ and $x,y\in \bV$. 
\item Let $\Lambdabar{}^0:=O_{\bD}^{\oplus 2}\subset \bV$ be a $\Z$-lattice. Then $(\Lambdabar{}^0\otimes_{\Z} \Z_{(p)})^{\vee}=\Lambdabar{}^0\otimes_{\Z} \Z_{(p)}$. 
\end{itemize}
These follow from the same argument as in the proof of Lemma \ref{sdlt}. By the definition of $\Delta$, the lattice chain $\{\Delta^n\Lambdabar{}^0\otimes_{\Z} \Z_{(p)}\}_{n\in \Z}$ in $\bV$ is self-dual. 

Let $\bG$ be the algebraic group over $\Q$ defined by
\begin{equation*}
\bG(R)=\{(g,c)\in \GL_{\bD \otimes_{\Q}R}(\bV \otimes_{\Q}R)\times \bG_m(R)\mid (g(v),g(w))^{\sim}
=c(v,w)^{\sim}\text{ for all }v,w\in \bV \otimes_{\Q}R\}
\end{equation*}
for any $\Q$-algebra $R$. By Corollary \ref{unhm} (ii), there are isomorphisms
\begin{equation*}
\bG \cong \GU_2(\bD),\quad \bG \otimes_{\Q} \Qp\cong \GU_2(D),\quad \bG \otimes_{\Q} \R \cong \GSp_4\otimes_{\Z}\R. 
\end{equation*}
Let $h\colon \Res_{\C/\R}\G_m \rightarrow \bG \otimes_{\Q}\R$ be homomorphism induced by
\begin{equation*}
\Res_{\C/\R}\G_m\rightarrow \GSp_4\otimes_{\Z}\R;\, a+b\sqrt{-1}\mapsto 
{\begin{pmatrix}
aE_2&-bE_2\\
bE_2&aE_2
\end{pmatrix}}
\end{equation*}
and an isomorphism $\bG \otimes_{\Q}\R \cong \GSp_4\otimes_{\Z}\R$. We denote by $\bX$ the $\bG(\R)$-conjugacy class of $\Hom(\Res_{\C/\R}\G_m,\bG \otimes_{\Q}\R)$ containing $h$ as above. Then we obtain a Shimura datum $(\bG,\bX)$. The reflex field of $(\bG,\bX)$ is $\Q$. 

Moreover, we can construct a cocharacter $\mu \colon \G_m \rightarrow \bG \otimes_{\Q}\C$ of $\bG$ over $\C$ from $h$. See \cite[6.2]{rz}. Finally, let $K^p\subset \bG(\A^p_f)$ be a compact open subgroup which is contained in the congruence subgroup of level $N\geq 3$, where $N$ is prime to $p$. Consequently we obtain a datum
\begin{equation*}
(\bD,O_{\bD},\bV,(\,,\,)^{\sim},\mu,\{\Delta^n \Lambdabar{}^0\otimes_{\Z} \Z_{(p)}\}_{n\in \Z},K^p), 
\end{equation*}
and hence we can define an integral model over $\Zp$ as in \cite[Definition 6.9]{rz}, which is denoted by $\sS_{K}$. Here $K=K^pK_p$ and $K_p$ is the stabilizer of $\Lambdabar{}^0\otimes_{\Z}\Zp$ in $\bG \otimes_{\Q} \Qp$. It is defined as the functor which parametrizes tuples $(A,\iota,\lambda,\overline{\eta}^p)$ for any connected noetherian $\Z_p$-scheme $S$, where
\begin{itemize}
\item $A$ is an abelian scheme of dimension $4$ over $S$, 
\item $\iota \colon \O_{\bD}\otimes_{\Z} \Z_{(p)}\rightarrow \End(A)\otimes_{\Z} \Z_{(p)}$ is a ring homomorphism, 
\item $\lambda \colon A \rightarrow A^{\vee}$ is a prime-to-$p$ quasi-polarization, 
\item $\overline{\eta}^p$ is a $K^p$-level structure, that is, a $\pi_1(S,\sbar)$-invariant $K^p$-orbit of an $O_{\bD}$-linear isomorphism
\begin{equation*}
\eta^p \colon H_1(A_{\sbar},\A_f^p) \xrightarrow{\cong}\bV \otimes_{\Q} \A_f^p
\end{equation*}
which respects the symplectic forms up to a constant in $(\A_f^p)^{\times}$ (here, $H_1(A_{\sbar},\A_f^p):=(\prod_{\ell \neq p}T_{\ell}A_{\sbar})\otimes_{\Z}\Q$ and $\sbar$ is a geometric point of $S$. See \cite[\S 5]{kot3}), 
\end{itemize}
satisfying the following conditions for any $d\in O_{\bD}$: 
\begin{itemize}
\item $\det(T-\iota(d) \mid \Lie(A))=(T^2-\trd_{\bD/\Q}(d)T+\nrd_{\bD/\Q}(d))^2$, 
\item $\lambda \circ \iota(d)=\iota(d^{*})^{\vee}\circ \lambda$. 
\end{itemize}
Two tuples $(A_1,\iota_1,\lambda_1,\overline{\eta}_1^p)$ and $(A_2,\iota_2,\lambda_2,\overline{\eta}_2^p)$ are equivalent if there is a prime-to-$p$ quasi-isogeny $\rho \colon A_1\rightarrow A_2$ such that $\rho^{\vee} \circ \lambda_2\circ \rho=\lambda_1$ and $\overline{\eta}_2^p\circ H_1(\rho,\A_f^p)=\overline{\eta}_1^p$. 

The functor above is representable by a quasi-projective scheme over $\Z_{(p)}$ by geometric invariant theory. See \cite[\S 5]{kot3}. 

\begin{rem}
Let us explain a symplectic form on $H_1(A,\A_f^p)$ for a polarized abelian variety $(A,\lambda)$ over an algebraically closed field of characteristic $p$. Choose an isomorphism $\A_f^p(1)\cong \A_f^p$. For any prime $\ell \neq p$, the polarization $\lambda$ induces a homomorphism $T_{\ell}A\rightarrow T_{\ell}A^{\vee}$, which induces a Weil pairing 
\begin{equation*}
T_{\ell}A\times T_{\ell}A\rightarrow \Z_{\ell}(1). 
\end{equation*}
Using this, we obtain a symplectic form over $\A_f^p$: 
\begin{equation*}
H_1(A,\A_f^p)\times H_1(A,\A_f^p)\rightarrow \A_f^p(1)\cong \A_f. 
\end{equation*}
We endow $H_1(A,\A_f^p)$ with symplectic form as above. 
\end{rem}

\subsection{Supersingular loci and non-smooth loci of the integral models}\label{ssdf}

In this section, we define the supersingular locus of the integral model $\sS_{K}$, and consider the non-smooth locus of $\sS_{K,W}:=\sS_{K}\times_{\spec \Zp}\spec W$. \emph{In this subsection, we regard $\D_{\Q}=V\otimes_{\Qp}K_0$ as an isocrystal over $\Fpbar$ by the $\sigma$-linear map $F=b\circ \sigma$. }

Let $\sS_{K}^{\si}$ be the supersingular locus of $\sS_{K}$, that is, the reduced closed subscheme of $\sS_{K,\Fpbar}:=\sS_{K}\times_{\spec \Z_{(p)}}\spec \Fpbar$ such that
\begin{equation*}
\sS_{K}^{\si}(k)=\{(A,\iota,\lambda,\overline{\eta}^p)\in \sS_{K}(k)\mid A\text{ is supersingular}\}
\end{equation*}
for any algebraically closed field $k$ of characteristic $p$. Let $\widehat{\sS}_{K}^{\si}$ be the completion of $\sS_{K,W}$ along $\sS_{K}^{\si}$. 

We also define the \emph{basic locus} $\sS_{K}^{\bs}$ of $\sS_{K,\Fpbar}$ to be the closed subscheme of $\sS_{K,\Fpbar}$ such that
\begin{equation*}
\sS_{K}^{\bs}(k)=\{(A,\iota,\lambda,\overline{\eta}^p)\in \sS_{K}(k) \mid \D(A[p^{\infty}])_{\Q}\text{ is basic (see \cite[5.1]{kot2})}\}
\end{equation*}
for any algebraically closed field $k$ of characteristic $p$. 

Now consider the datum below: 
\begin{equation*}
\mathcal{D}'=(\bD \otimes_{\Q} \Qp,*,O_{\bD}\otimes_{\Z} \Zp,\bV \otimes_{\Q}\Qp,(\,,\,)^{\sim},\{\Delta^n\Lambdabar{}^0\otimes_{\Z} \Zp\}_{n\in \Z}). 
\end{equation*}
Since $\bD$ is ramified at $p$, we have an isomorphism $\bD \otimes_{\Q} \Qp \cong D$ over $\Qp$. By the Skolem--Noether theorem, we may take an isomorphism such that $\varepsilon \in \bD$ maps to an element of $\Qps$. We identify $\bD \otimes_{\Q} \Qp$ and $D$ by the isomorphism above. Then we may assume that $\varepsilon \in \Qps$ in Section \ref{mrlv} is an element above. In this case, under the isomorphism $\bD \otimes_{\Q}\Qp \cong D$, the involution $*$ defined in Section \ref{shqt} equals the one defined in Section \ref{mrlv}. Moreover, we have $\delta=\Pi c$ for some $c\in \Zps^{\times}$. 

\begin{lem}
\emph{There is a $D$-linear isometry of symplectic spaces over $\Qp$
\begin{equation*}
\psi \colon \bV \otimes_{\Q}\Qp \xrightarrow{\cong}V
\end{equation*}
satisfying $\psi(\Lambdabar{}^0\otimes_{\Z}\Zp)=\Lambda^0$. }
\end{lem}

\begin{proof}
Let
\begin{equation*}
\psi \colon \bV \otimes_{\Q}\Qp \rightarrow V;(x_1,x_2)\otimes a\mapsto (ax_1c^{-1},ax_2). 
\end{equation*}
Then $f$ is a $D$-linear isomorphism. Moreover, we have
\begin{align*}
(\psi(x_1,x_2),\psi(y_1,y_2))=&\trd_{D/\Qp}(\Pi^{-1}((x_1c^{-1})^{*}y_2-x_2^{*}y_1c^{-1}))\\
=&\trd_{D/\Qp}(c^{-1}\Pi^{-1}(x_1^{*}y_2-x_2^{*}y_1))\\
=&((x_1,x_2),(y_1,y_2))^{\sim}
\end{align*}
for $x_1,x_2,y_1,y_2\in \bV$. The assertion $\psi(\Lambdabar{}^0\otimes_{\Z} \Zp)=\Lambda^0$ follows from the definition of $\psi$. 
\end{proof}

By the isometry above, we obtain the following: 
\begin{itemize}
\item The cocharacter $\mu$ obtained by $h$ is identical to that of Section \ref{rzdt}. 
\item The lattice chain $\{\Delta^n\Lambdabar{}^0\otimes_{\Z} \Zp\}_{n\in \Z}$ in $\bV \otimes_{\Q}\Qp$ equals $\{\Pi^n\Lambda^0\}_{n\in \Z}$.
\end{itemize}
Therefore, under the isomorphism $\bD \otimes_{\Q}\Qp \cong \Qp$ and the isometry $\psi$, the datum $\mathcal{D}'$ equal the Rapoport--Zink datum $\mathcal{D}$ defined in Section \ref{rzdt}. 

\begin{prop}\label{bsss}
\emph{We have an equality $\sS_{K}^{\bs}=\sS_{K}^{\si}$. }
\end{prop}

\begin{proof}
We follow the proof of \cite[Lemma 4.2.4]{hp2}. Let $i \colon G\rightarrow \GL_{\Qp}(V)$ be the standard representation. Then it suffices to prove that $b$ is basic if and only if $\rho(b)$ is so. This follows from that the center of $G$ equals $\G_{m}$, which is identified with the scalar matrices by $i$. 
\end{proof}

\begin{thm}\label{nsss}
\emph{
\begin{enumerate}
\item The scheme $\sS_{K,W}$ is regular and flat over $W$. 
\item Let ${\sS}_{K,W}^{\ns}$ be the set of non-smooth points in ${\sS}_{K,W}$. Then ${\sS}_{K,W}^{\ns}$ is the finite set of all $\Fpbar$-rational points such that $\iota(\Pi)=0$ on $\Lie(A)$. 
\item Let $(\widehat{\sS}_{K}^{\si})^{\nfs}$ be the set of non-formally smooth points in $\widehat{\sS}_{K}^{\si}$. Then we have an equality $\sS_{K,W}^{\ns}=(\widehat{\sS}_{K}^{\si})^{\nfs}$. 
\end{enumerate}}
\end{thm}

\begin{proof}
By \cite[Theorem 6.4]{hai} and Theorem \ref{flat}, the following hold: 
\begin{itemize}
\item $\sS_{K,W}$ is flat over $\spec W$ and regular, 
\item $\sS_{K,W}\times_{\spec W}\spec K_0$ is smooth over $\spec K_0$, 
\item $x\in \sS_{K,W}^{\ns}$ corresponds to an object $(A,\iota,\lambda,\overline{\eta}^p)\in \sS_{K}(\Fpbar)$ such that $\iota(\Pi)=0$ on $\Lie(A)$. Moreover, there is an isomorphism between the complete local ring of $x\in \sS_{K,W}$ and $W[[t_1,t_2,t_3,t_4]]/(t_1t_2+t_3t_4+p)$. 
\end{itemize}
Therefore, (i) and (ii) follow. Let us show (iii) in the sequel. We must show that $\sS_{K}^{\ns}\subset \sS_{K}^{\si}(\Fpbar)$. Take $(A,\iota,\lambda,\overline{\eta}^p)\in \sS_{K}^{\ns}$. It suffices to show that $\D(A[p^{\infty}])_{\Q}$ is isoclinic of slope $1/2$. Put $N:=\D(A[p^{\infty}])_{\Q}$ and $M:=\D(A[p^{\infty}])$. Then $M$ is a $W$-lattice in $N$ which is stable under $O_D$, $F$ and $pF^{-1}$. Furthermore, we have an isomorphism $\Lie(A)\cong M/pF^{-1}(M)$. Using the assumption that $\iota(\Pi)=0$ on $M/pF^{-1}(M)$, we have $F(M)=\Pi M$; cf.~the proof of Theorem \ref{sgvl}. Hence we obtain $F^2(M)=pM$, which means that $N$ is isoclinic of slope $1/2$. 

By above, $x\in \sS_{K}^{\si}$ belongs to $\sS_{K}^{\ns}$ if and only if the complete local ring of $\sS_{K,W}^{\si}$ at $x$ is not formally smooth over $W$, which is equivalent to $x\in (\widehat{\sS}_{K}^{\si})^{\nfs}$. 
\end{proof}

\begin{prop}\label{ssne}
\emph{The supersingular locus $\sS_{K}^{\si}$ is non-empty. }
\end{prop}

To prove Proposition \ref{ssne}, we need a comparison between $\sS_{K}$ and the integral model of $\Sh_{K}(\bG,\bX)$ considered in \cite{kmps}. To explain more precisely, let $\bG'$ be the symplectic similitude group of $\bV$ regarded as a symplectic space over $\Q$. Then there is an isomorphism $\bG'\cong \GSp_{8}\otimes_{\Z}\Q$. We denote by $\bi \colon \bG\rightarrow \bG'$ be the canonical injection, and let $\bX'$ the $\bG'(\R)$-conjugacy class of $\bi \circ h$. Then $(\bG',\bX')$ is the Shimura datum whose reflex is $\Q$. Moreover, we obtain an embedding of Shimura datum
\begin{equation*}
\bi \colon (\bG,\bX)\rightarrow (\bG',\bX'). 
\end{equation*}
For a compact open subgroup ${K'}^p$ of $\bG'(\A_{f}^{p})$, let $\sS'_{{K'}^p}$ be the $\Z_{(p)}$-scheme defined as the functor which parametrizes equivalence classes of triples $(A,\lambda,[\eta^p])$, where
\begin{itemize}
\item $A$ is an abelian scheme over $S$, 
\item $\lambda \colon A \rightarrow A^{\vee}$ is a prime-to-$p$ quasi-polarization, 
\item $[\eta^p]$ is a ${K'}^p$-level structure, that is, a $\pi_1(S,\sbar)$-invariant ${K'}^p$-orbit of an isomorphism
\begin{equation*}
\eta^p \colon H_1(A_{\sbar},\A_f^p) \xrightarrow{\cong}\bV \otimes_{\Q} \A_f^p
\end{equation*}
which respects the symplectic forms up to a constant in $(\A_f^p)^{\times}$ (here $\sbar$ is a geometric point of $S$). 
\end{itemize}
We define the notion of equivalence on triples by the same manner as in the definition of $\sS_{K}$. 

\begin{prop}\label{imci}
\emph{There is an open compact subgroup ${K'}^p$ of $\bG'(\A_{f}^{p})$ such that the canonical morphism 
\begin{equation*}
\sS_{K}\rightarrow \sS'_{{K'}^p}\times_{\spec \Z_{(p)}}\spec \Zp;(A,\iota,\lambda,\overline{\eta}^p)\mapsto (A,\lambda,[\eta^p])
\end{equation*}
is a closed immersion. }
\end{prop}

\begin{proof}
Put $\sS_{K_p}:=\plim[K^p]\sS_{K^pK_p}$ and $\sS':=\plim[{K'}^p](\sS_{{K'}^p}\times_{\spec \Z_{(p)}}\spec \Zp)$. Then $\sS_{K_p}$ is the moduli space of prime-to-$p$ isogeny classes of the quadruples as in the case of $\sS_{K}$, except that instead of being $\eta^p$ exactly an $O_{\bD}$-linear isomorphism $H_1(A_{\overline{s}},\A_{f}^{p})\xrightarrow{\cong}\bV\otimes_{\Q}\A_{f}^{p}$. The similar moduli interpretation for $\sS'$ also holds. It suffices to prove that the canonical morphism 
\begin{equation*}
\sS(\bi) \colon \sS_{K_p}\rightarrow \sS'
\end{equation*}
induced by $\bi$ is a closed immersion. We may prove that $\sS(\bi)$ is a proper monomorphism. The properness follows from the valuative criterion by using the theory of N{\'e}ron models. On the other hand, the moduli descripstion implies that $\sS(\bi)$ is a monomorphism. Hence the assertion follows. 
\end{proof}

Fix $K^p$ and ${K'}^p$ satisfying Proposition \ref{imci}. We denote by $\sS_{K}^{\KMPS}$ the integral model of $\Sh_{K}(\bG,\bX)$ with respect to $\bi$ in the sense of \cite{kmps}, that is, the normalization of the scheme-theoretic closure of $\Sh_{K}(\bG,\bX)$ in $\sS'_{K'}$. 

\begin{cor}\label{kpim}
\emph{
\begin{enumerate}
\item There is an isomorphism $\sS_{K}\cong \sS_{K}^{\KMPS}\times_{\spec \Z_{(p)}}\spec \Zp$. 
\item Under the isomorphism in (i), $\sS_{K}^{\bs}$ equals the basic locus in the sense of \cite{kmps}. 
\end{enumerate}}
\end{cor}

\begin{proof}
The assertion (i) is a consequence of Theorem \ref{nsss} (i) and Proposition \ref{imci}. The assertion (ii) follows from the definitions of $\sS_{K}$ and $\sS_{K}^{\KMPS}$. 
\end{proof}

\begin{proof}[Proof of Proposition \ref{ssne}]
By Corollary \ref{kpim} (ii), it suffices to prove that the basic locus of $\sS_{K}^{\KMPS}$ is non-empty. However, this assertion is exactly the same as \cite[Theorem 1.3.13 (2)]{kmps}. 
\end{proof}

\subsection{Proof of the global result}

We describe the scheme $\sS_{K}^{\si}$ by using the results on $\M_G$ and the $p$-adic uniformization theorem. To apply the $p$-adic uniformization theorem for $\sS_{K}^{\si}$, we need to show that $\sS_{K}^{\si}$ equals $\sS_{K}^{\bs}$, $\sS_{K}^{\si}$ is not empty, and the Hasse principle for $\bG$ holds. However, the first assertion is Proposition \ref{bsss}, the second assertion is Proposition \ref{ssne}. Moreover, the last assertion is pointed out in \cite[\S 7]{kot3}. Therefore, we obtain the following: 

\begin{thm}\label{unif} (\cite[Theorem 6.30]{rz})
\emph{For a fixed $(A_0,\iota_0,\lambda_0,\overline{\eta}^p)\in \sS_{K}^{\si}(\Fpbar)$, there is an isomorphism
\begin{equation*}
I(\Q)\backslash (\M_G\times \bG(\A^p_f)/K^p)\xrightarrow{\cong} \widehat{\sS}_{K}^{\si}
\end{equation*}
of formal schemes over $\spf W$. Here $I$ is an algebraic group over $\Q$ defined by 
\begin{equation*}
I(R)=\{(g,c)\in (\End_{O_\bD}^0(A_0)\otimes_{\Q} R)^{\times}\times \G_m(R)\mid g^{\vee}\circ \lambda_0\circ g=c\lambda_0\}
\end{equation*}
for any $\Q$-algebra $R$. }
\end{thm}
Note that $I\otimes_{\Q}\R$ is anisotropic modulo center. Moreover, for any prime number $\ell$, we have
\begin{equation*}
I\otimes_{\Q}\Q_{\ell}\cong 
\begin{cases}
\bG \otimes_{\Q}\Q_{\ell} & \text{if }\ell \neq p,\\
J & \text{if }\ell=p.
\end{cases}
\end{equation*}
We regard $I(\Q)$ as a subgroup of $J(\Qp)\times \bG(\A_f^p)$ by the diagonal embedding. Put $m:=\#I(\Q)\backslash \bG(\A_f^p)/K^p$ (note that $I(\Q)\backslash \bG(\A_f^p)/K^p$ is a finite set), and let $\{g_1,\ldots ,g_m\}$ is a set of complete representative of $I(\Q)\backslash \bG(\A_f^p)/K^p$. Put $\Gamma_i:=I(\Q)\cap(J(\Qp)\times g_iK^pg_i^{-1})$. Then we have 
\begin{equation*}
I(\Q)\backslash (\M_G\times \bG(\A^p_f)/K^p)\cong \coprod_{i=1}^m \Gamma_i\backslash (\M_G\times g_iK^p/K^p). 
\end{equation*}
Moreover, if we regard $\Gamma_i$ as a subgroup of $J(\Qp)$, then the right-hand side is of the form $\coprod_{i=1}^{m}\Gamma_i\backslash \M_G$. Note that the group $\Gamma_i$ is discrete, cocompact modulo center and torsion-free by the hypothesis on $K^p$. 

We compute the numbers of connected and irreducible components of $\sS_{K}^{\si}$, and the set $\sS_{K,W}^{\ns}$. 

\begin{thm}\label{icnn}
\emph{
\begin{enumerate}
\item There is an equality
\begin{equation*}
\#\{\text{connected components of }\sS_{K}^{\si}\}=\#(I(\Q)\backslash (J^0\backslash J(\Qp)\times \bG(\A_f^p)/K^p))
\end{equation*}
(see Definition \ref{j0df} for the definition of $J^0$). 
\item Let $K_{\max,i}$ be the stabilizer of a vertex $x \in \Vtx(i)$ for $i\in \{0,2\}$. Then there is an equality
\begin{gather*}
\#\{\text{irreducible components of }\sS_{K}^{\si}\}\\
=\#(I(\Q)\backslash (J(\Qp)/K_{\max,0}\times \bG(\A_f^p)/K^p))+\#(I(\Q)\backslash (J(\Qp)/K_{\max,2}\times \bG(\A_f^p)/K^p)). 
\end{gather*}
\item Let $K_{\min}$ be the stabilizer of a vertex lattice $x \in \Vtx^{\nsp}$. Then there is an equality
\begin{equation*}
\# \sS_{K}^{\ns}=\#(I(\Q)\backslash (J(\Qp)/K_{\min} \times \bG(\A_f^p)/K^p)). 
\end{equation*}
\end{enumerate}
All numbers above are finite. }
\end{thm}

\begin{proof}
We follow the proof of \cite[Proposition 6.3]{vol}. We have a decomposition
\begin{equation}\label{m0dc}
\Gamma_i\backslash \M_G\cong \coprod_{(\Gamma_iJ^0)\backslash J(\Qp)}(\Gamma_i\cap J^0)\backslash \M_G^{(0)}
\end{equation}
for any $i$. 

(i): We have 
\begin{equation*}
\coprod_{i=1}^m(\Gamma_iJ^0)\backslash J(\Qp)
\cong \coprod_{i=1}^m\Gamma_i \backslash (J(\Qp)/J^0)
\cong \coprod_{i=1}^m\Gamma_i \backslash (J(\Qp)/J^0\times g_iK^p/K^p)\cong I(\Q)\backslash (J(\Qp)/J^0\times \bG(\A_f^p)/K^p). 
\end{equation*}
Since $\M_G^{(0)}$ is connected by Theorem \ref{btm0} (iii), the assertion follows from (\ref{m0dc}) and Theorem \ref{unif}. 

(ii): By Theorem \ref{btm0} (iii), the number of irreducible components of $(\Gamma_i\cap J^0)\backslash \M_G$ equals $\#((\Gamma_i\cap J^0) \backslash \Vtx^{\hs})$. By Proposition \ref{trjv} (ii), the latter number equals $\#((\Gamma_i\cap J^0) \backslash J^0/K_{\max,0})+\#((\Gamma_i\cap J^0) \backslash J^0/K_{\max,2})$. Therefore the assertion follows from the same computation in (i), (\ref{m0dc}) and Theorem \ref{unif}. 

(iii): By Theorems \ref{sgvl} and \ref{nsss} (iii), the number of non-formally smooth points in $(\Gamma_i\cap J^0)\backslash \M_G$ equals $\#((\Gamma_i\cap J^0) \backslash \Vtx^{\nsp})$. By Proposition \ref{trjv} (iii), the latter number equals $\#((\Gamma_i\cap J^0) \backslash J^0/K_{\min})$. Therefore the assertion follows from the same computation in (i), (\ref{m0dc}) and Theorem \ref{unif}. 
\end{proof}

Let $\pi_i \colon \M_G \rightarrow \Gamma_i \backslash \M_G$ be the canonical morphism for each $i$. 

\begin{prop}\label{birt}
\emph{Let $x\in \Vtx^{\hs}$ and $i\in \{1,\ldots ,m\}$. 
\begin{enumerate}
\item The morphism $\pi_i \colon \M_{G,x}^{(0)}\rightarrow \pi_i(\M_{G,x}^{(0)})$ is proper. The scheme $\pi_i(\M_{G,x}^{(0)})$ is projective over $\spec \Fpbar$. 
\item The morphism $\pi_i \colon \M_{G,x}^{(0)}\rightarrow \pi_i(\M_{G,x}^{(0)})$ is  birational. 
\end{enumerate}}
\end{prop}

\begin{proof}
(i): The scheme $\M^{(0)}_{G,x}$ is projective over $\spec \Fpbar$ by Corollary \ref{redv} (ii). On the other hand, $(\Gamma_i\cap J^0)\backslash \M^{(0),\red}_{G}$ is separated over $\spec \Fpbar$ since it is a closed subscheme of $\sS_{K,\Fpbar}$, which is quasi-projective over $\spec \Fpbar$. Hence $\pi_i$ is proper. Consequently, $\pi_i(\M_{G,x}^{(0)})$ is proper over $\spec \Fpbar$ and hence it is a closed subscheme of $\sS_{K,\Fpbar}$, which quasi-projective over $\spec \Fpbar$. Therefore, the assertion follows. 

(ii): The set $\{y\in \Vtx^{\nsp}\mid y\leq x\}$ is finite by Propositions \ref{ltic} (iii) and \ref{vtvt} (i), (ii). The scheme $\M^{(0),\red}_{G,y}$ consists of single $\Fpbar$-rational point for $y\in \Vtx^{\nsp}$ by Proposition \ref{vtvt} (i) and Corollary \ref{redv} (ii). Hence $\bigcup_{y\in \Vtx^{\nsp},y\leq x}\M^{(0)}_{G,y}$ is a finite set. Put 
\begin{equation*}
U:=\M^{(0)}_{G,x}\setminus \bigcup_{y\in \Vtx^{\nsp},y\leq x}\M^{(0)}_{G,y}. 
\end{equation*}
We claim that $\gamma(U)\cap U=\emptyset$ for $\gamma \in \Gamma_i\cap J^0$. Since $\stab_{J^0}(x)$ is compact and $\Gamma_i$ is discrete, $\stab_{J^0}(x)\cap \Gamma_i$ is finite. Moreover, we have $\stab_{J^0}(x)\cap \Gamma_i=\{ \id \}$ since $\Gamma_i$ is torsion-free. Now take $\gamma \in (\Gamma_i\cap J^0)\setminus \{ \id\}$. Since $x$ and $\gamma(x)$ have the same type, Theorem \ref{mtl1} (v) implies that $\M_{G,x}\cap \M_{G,\gamma(x)}$ is not $1$-dimensional. Moreover, $\M_{G,x}\cap \M_{G,\gamma(x)}$ consists of at most single point, and is of the form $\M_{G,y}$ for some $y\in \Vtx^{\nsp}$ if non-empty. Hence the claim follows. Consequently, $\pi_i \colon U\rightarrow \pi_i(U)$ is bijective on any $k$-rational points, where $k$ is a field containing $\Fpbar$. 

Now we prove that $\pi_i \colon U\rightarrow \pi_i(U)$ is birational. Let $s\colon \spec K(\pi_i(U))\rightarrow \pi_i(U)$ be the morphism induced by the localization at the unique generic point of $\pi_i(U)$, where $K(\pi_i(U))$ is the function field of $\pi_i(U)$. By the bijectivity of $\pi_i \colon U\rightarrow \pi_i(U)$ on $K(\pi_i(U))$-rational points, there is a unique morphism $s'\colon \spec K(\pi_i(U))\rightarrow U$ such that $s=\pi_i\circ s'$. Hence we obtain a homomorphism $K(U)\rightarrow K(\pi_i(U))$ such that $K(\pi_i(U))\xrightarrow{\pi_i} K(U)\rightarrow K(\pi_i(U))$ is an identity map, where $K(U)$ is the function field of $U$. Hence $K(\pi_i(U))\xrightarrow{\pi_i} K(U)$ is an isomorphism, which implies the assertion. 
\end{proof}

Finally, we prove the the main result for $\sS_{K}^{\si}$: 

\begin{thm}\label{mtsh}
\emph{
\begin{enumerate}
\item The scheme $\sS_{K}^{\si}$ is purely $2$-dimensional. Every irreducible component is projective and birational to the Fermat surface defined by
\begin{equation*}
x_0^{p+1}+x_1^{p+1}+x_2^{p+1}+x_3^{p+1}=0
\end{equation*}
in $\Proj \Fpbar[x_0,x_1,x_2,x_3]$. 
\item Let $F$ be an irreducible component of $\sS_{K}^{\si}$. Then the following hold:
\begin{itemize}
\item There are at most $(p+1)(p^2+1)$-irreducible components of $\sS_{K}^{\si}$ whose intersections with $F$ is birational to $\P^1_{\Fpbar}$. Here we endow the intersections with reduced structures. 
\item There are at most $p(p+1)(p^2+1)$-irreducible components of $\sS_{K}^{\si}$ which intersects to $F$ at a single point. 
\item Other irreducible components of $\sS_{K}^{\si}$ do not intersect $F$. 
\end{itemize}
\item Each non-smooth point in $\sS_{K,W}$ is contained in at most $2(p+1)$-irreducible components of $\sS_{K}^{\si}$. 
\item Each irreducible component of $\sS_{K}^{\si}$ contains at most $(p+1)(p^2+1)$-non-smooth points in $\sS_{K,W}$. 
\end{enumerate}}
\end{thm}

\begin{proof}
The assertion (i) follows from Proposition \ref{birt}. In the sequel, we only prove (iv). The others are similar. 

Take $x\in \Vtx^{\hs}$ and $i$ as in Proposition \ref{birt}. By Proposition \ref{birt}, $\pi_i$ induces the following surjection: 
\begin{equation*}
\M_G^{\nfs}\cap \M_{G,x}^{(0)}=\pi_i^{-1}(\sS_{K,W}^{\ns}\cap \pi_i(\M_{G,x}^{(0)}))\rightarrow \sS_{K,W}^{\ns}\cap \pi_i(\M_{G,x}^{(0)}). 
\end{equation*}
By Propositions \ref{ltic} (iii) and \ref{odis}, the left-hand side of the surjection above has exactly $(p+1)(p^2+1)$-elements. Therefore the assertion follows. 
\end{proof}

\section{Application to arithmetic intersections}\label{arit}

For another application, we compute the intersection multiplicity of certain cycle, called the GGP cycle. We keep the notation for dual lattices in $\L_{\Q}^{\Phi}$ as Definition \ref{hvdl}. 

\subsection{Basic properties of the GGP cycle}\label{ggpd}

We define an algebraic group $J_H$ over $\Qp$ by
\begin{equation*}
J_H(R)=\{g\in H(K_0\otimes_{\Qp}R)\mid g\circ F=F\circ g\}
\end{equation*}
for any $\Qp$-algebra $R$. Let us also define an algebraic group $J_H^0$ over $\Qp$ by the same formula for $H^0$ instead of $H$. The representabilities of $J_H$ and $J_H^0$ follow from \cite[Proposition 1.12]{rz}. We have canonical injections 
\begin{equation*}
J\rightarrow J_H^0\rightarrow J_H. 
\end{equation*}
Moreover, we have an isomorphism $J_H^0\cong \GSpin(\L_{\Q}^{\Phi})$ of algebraic groups over $\Qp$. See \cite[Remark 2.8]{hp}. 

We define a left action of $J_H(\Qp)$ on $\M_H$ as 
\begin{equation*}
\M_H \rightarrow \M_H;(X,\iota,\lambda,\rho)\mapsto (X,\iota,\lambda,g\circ \rho)
\end{equation*}
for $g\in J_H(\Qp)$; cf.~Definition \ref{gpdf} for the action of $J(\Qp)$ on $\M_G$. Then, the bijection $M\mapsto L(M)$ in Theorem \ref{hppc} commutes with the actions of $J_H^0(\Qp)$. 

\begin{dfn}
\begin{itemize}
\item Let $i_{G,H}\colon\M_G\rightarrow \M_H$ be the closed immersion defined in Proposition \ref{mgeb}. Then we obtain a closed immersion
\begin{equation*}
(\id_{\M_G},i_{G,H})\colon p^{\Z}\backslash \M_G \rightarrow 
p^{\Z}\backslash (\M_G\times_{\spf W}\M_H);x\mapsto (x,i_{G,H}(x)).
\end{equation*}
We define the \emph{GGP cycle} $\Delta$ as the image of $(\id,i_{G,H})$. It is a formal scheme over $\spf W$. 
\item For $g\in J_H(\Qp)$, put $g\Delta:=(\id \times g)(\Delta)$, where
\begin{equation*}
\id \times g\colon p^{\Z}\backslash (\M_G\times_{\spf W}\M_H) \rightarrow p^{\Z}\backslash (\M_G\times_{\spf W}\M_H);(x,y)\mapsto (x,g(y)). 
\end{equation*}
\end{itemize}
\end{dfn}

\begin{lem}\label{dlgh}
\emph{For any $g\in J_H(\Qp)$, the second projection $\pr_2\colon \M_G\times_{\spf W}\M_H\rightarrow \M_H$ induces an isomorphism 
\begin{equation*}
\Delta \cap g\Delta\cong p^{\Z}\backslash(i_{G,H}(\M_G)\cap \M_H^g), 
\end{equation*}
where $\M_H^g$ is the $g$-fixed part of $\M_H$. }
\end{lem}

\begin{proof}
Take $S\in \nilp_W$ and $(x,y)\in p^{\Z}\backslash(\M_G\times_{\spf W}\M_H)(S)$. Then we have $(x,y)\in p^{\Z}\backslash(\Delta \cap g\Delta)(S)$ if and only $(x,y),(x,g^{-1}(y))\in p^{\Z}\backslash(\M_G\times_{\spf W}\M_H)(S)$. Hence we have $y=i_{G,H}(x)=g^{-1}(y)$, that is, $y\in p^{\Z}\backslash(i_{G,H}(\M_G)\cap \M_H^g)(S)$. Therefore, $\pr_2$ induces an injection 
\begin{equation*}
\Delta \cap g\Delta \rightarrow p^{\Z}\backslash(i(\M_G)\cap \M_H^g). 
\end{equation*}
To prove the surjectivity of the morphism above, take $S\in \nilp_W$ and $y\in p^{\Z}\backslash(i_{G,H}(\M_G)\cap \M_H^g)(S)$. Then there is a unique $x\in p^{\Z}\backslash \M_G(S)$ satisfying $i_{G,H}(x)=y$. Moreover, we have $g(y)=y$ since $y\in p^{\Z}\backslash \M_H^g(S)$. Therefore, we have $(x,i_{G,H}(x))\in (\Delta \cap g\Delta)(S)$ and $\pr_2(x,i_{G,H}(x))=y$. Hence the assertion follows. 
\end{proof}

We moreover rewrite the formal scheme $\Delta \cap g\Delta$ for $g\in J_H^0(\Qp)$. Put
\begin{equation*}
L(g):=\sum_{i=0}^{5}\Zp(g^i\cdot y_1), 
\end{equation*}
which is a $\Zp$-submodule of $\L_{\Q}^{\Phi}$. 

\begin{dfn}\label{spcy}
For a $\Zp$-submodule $\bv$ of $\L_{\Q}^{\Phi}$, we define a closed formal subschme $\cZ(\bv)$ of $\M_H$ to be the locus $(X,\iota,\lambda,\rho)$ satisfying $\rho^{-1}\circ \bv \circ \rho \subset \End(X)$. 
\end{dfn}

\begin{prop}\label{nhtr}
\emph{Let $g\in J_H^0(\Qp)$. 
\begin{enumerate}
\item We have an equality $i_{G,H}(\M_G)\cap \M_H^g=\cZ(L(g))^{g}$. 
\item If $p^{\Z}\backslash \cZ(L(g))^{g}$ is an artinian scheme, then we have
\begin{equation*}
\O_{\Delta}\otimes^{\mathbb{L}}\O_{g\Delta}=\O_{\Delta}\otimes \O_{g\Delta},
\end{equation*}
that is, the left-hand side is represented by the right-hand side in the derived category of sheaves on $p^{\Z}\backslash (\M_G\times_{\spf W}\M_H)$. 
\end{enumerate}}
\end{prop}

\begin{proof}
(i): Take $S\in \nilp_W$ and $(X,\iota,\lambda,\rho)\in \M_H(S)$. Then we have $(X,\iota,\lambda,\rho)\in (i_{G,H}(\M_G)\cap \M_H^g)(S)$ if and only if $\rho^{-1}\circ y_1\circ \rho \in \End(X)$ and $g\circ \rho=\rho$. Note that $y_1=\iota_0({\Pi})$. Since $g\cdot y_1=g\circ y_1\circ g^{-1}$ by definition, the condition above is equivalent to the condition that $\rho^{-1}\circ L(g)\circ \rho \subset \End(X)$ and $g\circ \rho=g$, that is, $(X,\iota,\lambda,\rho)\in \cZ(L(g))^{g}$. 

(ii): Note that $p^{\Z}\backslash \M_G$ is a regular formal scheme of dimension $4$ by Corollary \ref{rzsg}. Hence locally $\Delta$ is the intersection of $4$-regular divisors. Put $\O:=\O_{p^{\Z}\backslash (\M_G\times_{\spf W}\M_H)}$. Take $x\in p^{\Z}\backslash \cZ(L(g))^{g}$ and let $s_i\in \O_x\,(1\leq i \leq 4)$ be regular elements such that $\O_{x}/(s_1,\ldots,s_4)=\O_{\Delta,x}$. Moreover, if we put $s_i:=g(s_{i-4})\in \O_{x}$ for $5\leq i\leq 8$. Since $\O_{x}$ is Cohen-Macaulay and $\dim \O_{p^{\Z}\backslash \cZ(L(g))^{g},x}=\dim \O_{x}/(s_1,\ldots,s_8)=0$ by hypothesis, $s_1,\ldots,s_8$ form a regular sequence by \cite[Corollaire 16.5.6]{ega40} for $M=\O_x$. Therefore, for $1\leq i\leq 7$, $s_{i+1}$ is regular in $\O_{x}/(s_1,\ldots,s_i)$, and hence we have
\begin{gather*}
\Tor_1^{\O_x}(\O_{x}/(s_1,\ldots ,s_i),\O_{x}/(s_{i+1}))
=\Ker(\O_{x}/(s_1,\ldots ,s_i)\xrightarrow{s_{i+1}}\O_{x}/(s_1,\ldots ,s_i))=0. 
\end{gather*}
Therefore we have 
\begin{equation*}
\Tor_1^{\O_x}(\O_{\Delta,x},\O_{g\Delta,x})=0. 
\end{equation*}
Consequently, we obtain an equality of sheaves
\begin{equation*}
\Tor_1^{\O}(\O_{\Delta},\O_{g\Delta})=0, 
\end{equation*}
which implies the assertion. 
\end{proof}

If $g\in J_H^0(\Qp)$ and $p^{\Z}\backslash(i_{G,H}(\M_G)\cap \M_H^g)$ is an artinian scheme, then we define
\begin{equation*}
\langle \Delta,g\Delta \rangle:=\chi(\O_{\Delta}\otimes^{\mathbb{L}}\O_{g\Delta}), 
\end{equation*}
where $\chi$ is the Euler-Poincar{\'e} characteristic. Then we have the following by Proposition \ref{nhtr}: 

\begin{prop}\label{isca}
\emph{Under the hypothesis above, we have an equality
\begin{equation*}
\langle \Delta,g\Delta \rangle=\sum_{x\in p^{\Z}\backslash \cZ(L(g))^g}\length_W \O_{p^{\Z}\backslash \cZ(L(g))^g,x}\in \Z. 
\end{equation*}}
\end{prop}

\subsection{Relation with the Bruhat--Tits strata of $\M_H$}

We give a description of $p^{\Z}\backslash \M_{H,\widetilde{\Lambda}}^g(\Fpbar)$ in terms of the Bruhat--Tits stratum of $\M_H$ for certain $\widetilde{\Lambda}\in \Vrt_H$. We will use it for $\widetilde{\Lambda}=L(g)^{\natural}$ in Section \ref{ggcp}. For $\widetilde{\Lambda} \in \Vrt_H$, put 
\begin{equation*}
\BT_{H,\widetilde{\Lambda}}:=\M_{H,\widetilde{\Lambda}}\setminus \bigcup_{\Lambda'\in \Vrt_H,\widetilde{\Lambda}'\subsetneq \widetilde{\Lambda}}\M_{H,\widetilde{\Lambda}'}. 
\end{equation*}
It is the Bruhat--Tits stratum of $\M_H$ attached to $\widetilde{\Lambda}$. Moreover, the following holds: 

\begin{prop}\label{bthr}(\cite[\S 2.6]{hp})
\emph{The bijection $M\mapsto L(M)$ in Theorem \ref{hppc} induces a bijection
\begin{equation*}
p^{\Z} \backslash \BT_{H,\widetilde{\Lambda}}(\Fpbar)
=\{L\colon \text{special lattice in }\L_{\Q}\mid \widetilde{\Lambda}(L)=\widetilde{\Lambda} \}.
\end{equation*}
Here $\widetilde{\Lambda}(L)$ is the $\Phi$-fixed part of $\sum_{i\in \Znn}\Phi^i(L)$, which is an element of $\Vrt_H$. }
\end{prop}

\begin{dfn}\label{vgul}
For $g\in J_H^0(\Qp)$, we define two subsets of $\Vrt_H$ as follow: 
\begin{itemize}
\item $\Vrt_H^g:=\{\widetilde{\Lambda} \in \Vrt_H\mid g\widetilde{\Lambda}=\widetilde{\Lambda}\}$, 
\item For $\widetilde{\Lambda}\in \Vrt_H^g$, $\Vrt_H^g(\widetilde{\Lambda}):=\{\widetilde{\Lambda}' \in \Vrt_H^g\mid \widetilde{\Lambda}' \subset \widetilde{\Lambda}\}$. 
\end{itemize}
\end{dfn}

If $\widetilde{\Lambda}\in \Vrt_H^g$, then we have $g \cdot \widetilde{\Lambda}^{\natural}=\widetilde{\Lambda}^{\natural}$, and hence $g$ induces an element $\gbar_{\widetilde{\Lambda}}\in \SO(\Omegabar_0(\widetilde{\Lambda}))$ (see Definition \ref{ombd} for the definition of $\Omegabar_0(\widetilde{\Lambda})$). Moreover, $g$ induces actions of $\gbar_{\widetilde{\Lambda}}$ on $\M_{H,\widetilde{\Lambda}}$ and $\BT_{H,\widetilde{\Lambda}}$. 

\begin{prop}\label{gfbt}
\emph{There is an equality
\begin{equation*}
p^{\Z}\backslash \M_{H,\widetilde{\Lambda}^{\natural}}^g(\Fpbar)=\coprod_{\widetilde{\Lambda}' \in \Vrt_H^g(\widetilde{\Lambda})}p^{\Z}\backslash \BT_{H,\widetilde{\Lambda}'}^{\gbar_{\widetilde{\Lambda}'}}(\Fpbar). 
\end{equation*}}
\end{prop}

\begin{proof}
The bijection $M\mapsto L(M)$ in Theorem \ref{hppc} induces a bijection
\begin{align*}
p^{\Z}\backslash \M_{H,\widetilde{\Lambda}}^g(\Fpbar)
&\cong \{L\colon \text{special lattice in }\L_{\Q}\mid \widetilde{\Lambda}^{\natural}\subset L,g\cdot L=L\} \\
&=\{L\colon \text{special lattice in }\L_{\Q}\mid \widetilde{\Lambda}(L)\subset \widetilde{\Lambda},g\cdot L=L\}. 
\end{align*}
On the other hand, the bijection in Proposition \ref{bthr} induces a bijection
\begin{align*}
\coprod_{\widetilde{\Lambda}'\in \Vrt_H^g(\widetilde{\Lambda})}p^{\Z}\backslash \BT_{H,\widetilde{\Lambda}'}^{\gbar_{\widetilde{\Lambda}'}}(\Fpbar)
&\cong \coprod_{\widetilde{\Lambda}'\in \Vrt_H^g(\widetilde{\Lambda})}\{L\colon \text{special lattice in }\L_{\Q}\mid \widetilde{\Lambda}(L)=\widetilde{\Lambda}',g\cdot L=L\}\\
&=\{L\colon \text{special lattice in }\L_{\Q}\mid \widetilde{\Lambda}(L)\in \Vrt_H^g(\widetilde{\Lambda}),g\cdot L=L\}. 
\end{align*}
Note that we have $g\cdot \widetilde{\Lambda}(L)=\widetilde{\Lambda}(L)$ if $g\cdot L=L$ for a special lattice $L$. Hence, under $g\cdot L=L$, the condition $\widetilde{\Lambda}(L)\in \Vrt_H^g(\widetilde{\Lambda})$ is equivalent to the condition $\widetilde{\Lambda}(L)\subset \widetilde{\Lambda}$. Therefore, the assertion follows. 
\end{proof}

In the end of this section, we give a criterion for non-emptyness of $p^{\Z}\backslash \BT_{H,\widetilde{\Lambda}}^{\gbar_{\widetilde{\Lambda}}}$. This is proved in the proof of \cite[Theorem 3.6.4]{lz} by considering Deligne--Lusztig varieties. See also \cite[Theorem 3.5.2]{lz}. 

\begin{prop}\label{neir}
\emph{Let $g\in J_H^0(\Qp)$ and $\widetilde{\Lambda}\in \Vrt_H^g$. Let $P_{\gbar_{\widetilde{\Lambda}}}\in \Fp[T]$ be the characteristic polynomial of $\gbar_{\widetilde{\Lambda}}$ on $\Omegabar_0(\widetilde{\Lambda})$. Assume that $P_{\gbar_{\widetilde{\Lambda}}}$ equals the minimal polynomial of $\gbar_{\widetilde{\Lambda}}$ on $\Omegabar_0(\widetilde{\Lambda})$. Then, we have $p^{\Z}\backslash \BT_{H,\widetilde{\Lambda}}^{\gbar_{\widetilde{\Lambda}}}\neq \emptyset$ if and only if $P_{\gbar_{\widetilde{\Lambda}}}$ is irreducible. In this case, we have 
\begin{equation*}
\#(p^{\Z}\backslash \BT_{H,\widetilde{\Lambda}}^{\gbar_{\widetilde{\Lambda}}})=\dim_{\Fp}\Omegabar_0(\widetilde{\Lambda})=\deg P_{\gbar_{\widetilde{\Lambda}}}<\infty. 
\end{equation*}}
\end{prop}

\subsection{Intersection multiplicity of the GGP cycles}\label{ggcp}

First, we state a result which is (essentially) proved in \cite{lz}. Second, we define the notion of minusculeness for $g\in J_H^0(\Qp)$. Finally, we show that the intersection multiplicity is defined for a minuscule $g$. 

In the sequel, we fix $g\in J_H^0(\Qp)$ satisfying $\M_H^g\neq \emptyset$. 

\begin{lem}\label{lggs}
\emph{We have $\Vrt_H^g\neq \emptyset$ and $g\cdot L(g)=L(g)$. }
\end{lem}

\begin{proof}
By the hypothesis $\M_H^g\neq \emptyset$, there is a special lattice $L$ in $\L_{\Q}$ satisfying $g\cdot L=L$. Then, we have $g\cdot \widetilde{\Lambda}(L)=\widetilde{\Lambda}(L)$, which concludes the first assertion. Moreover, if $\widetilde{P}_g(T)\in \Qp[T]$ is the characteristic polynomial of $g$ on $\L_{\Q}^{\Phi}$, then we have $\widetilde{P}_g(T)\in \Zp[T]$. Therefore the second assertion follows. 
\end{proof}

For $\widetilde{\Lambda}\in \Vrt_H^g$, we have $\cZ(\widetilde{\Lambda}^{\natural})=\M_{H,\widetilde{\Lambda}}$ by definition. Note that the right-hand side is a reduced scheme of characteristic $p$ by \cite[Proposition 4.2.11]{lz}. Moreover, $g$ induces an element $\gbar_{\widetilde{\Lambda}} \in \SO(\Omegabar_0(\widetilde{\Lambda}))$, as explained after Definition \ref{vgul}. Let $P_{\gbar_{\widetilde{\Lambda}}}\in \Fp[T]$ be the characteristic polynomial of $\gbar_{\widetilde{\Lambda}}$ on $\Omegabar_0(\widetilde{\Lambda})$.

\begin{dfn}
Let $R\in \Fp[T]$ be a non-zero polynomial. 
\begin{itemize}
\item We define the reciprocal of $R$ by $R^{*}(T):=T^{\deg(R)}R(T^{-1})$. 
\item The polynomial $R$ is self-reciprocal if $R^{*}=R$. 
\end{itemize}
\end{dfn}

\begin{lem}\label{chsr}
For $\widetilde{\Lambda}\in \Vrt_H^g$, the polynomial $P_{\gbar_{\widetilde{\Lambda}}}$ is self-reciprocal. 
\end{lem}

\begin{proof}
The assertion follows from $\gbar_{\widetilde{\Lambda}} \in \SO(\Omegabar_0(\widetilde{\Lambda}))$. 
\end{proof}

For a self-reciprocal $P\in \Fp[T]$, let $\Irr(P)$ be the set of all monic irreducible factors of $P$, and $\Irr^{\sr}(P)$ the set of all self-reciprocal elements of $\Irr(P)$. Furthermore, put $\Irr^{\nsr}(P):=\Irr(P)\setminus \Irr^{\sr}(P)$. We introduce an equivalence relation $\sim$ on $\Irr(P)$ and $\Irr^{\nsr}(P)$ by $R\sim cR^*$ for some $c\in \Fpt$. For $R\in \Irr(P)$, let $[R]$ be the image of $R$ under the canonical map $\Irr(P)\rightarrow \Irr(P)/\sim$, and $m(R)$ the multiplicity of $R$ in $P$. Then we have $m(R)=m(R^*)$ since $P$ is self-reciprocal. 

\begin{prop}\label{iart}
\emph{Let $\widetilde{\Lambda}\in \Vrt_H^g$. Assume that $P_{\gbar_{\widetilde{\Lambda}}}$ equals the minimal polynomial of $\gbar_{\widetilde{\Lambda}}$ on $\Omegabar_0(\widetilde{\Lambda})$. Then the following are equivalent: 
\begin{enumerate}
\item We have $p^{\Z}\backslash \M_{H,\widetilde{\Lambda}}^g\neq \emptyset$. 
\item There is a unique $Q_{\gbar_{\widetilde{\Lambda}}}\in \Irr^{\sr}(P_{\gbar_{\widetilde{\Lambda}}})$ such that $m(Q_{\gbar_{\widetilde{\Lambda}}})$ is odd.
\end{enumerate}
If the conditions above hold, then we have
\begin{equation*}
\#(p^{\Z}\backslash \M_{H,\widetilde{\Lambda}}^g(\Fpbar))=\deg Q_{\gbar_{\widetilde{\Lambda}}}\prod_{[R]\in \Irr^{\nsr}(P_{\gbar_{\widetilde{\Lambda}}})/\sim}(1+m(R))<\infty. 
\end{equation*}
In particular, $p^{\Z}\backslash \M_{H,\widetilde{\Lambda}}^g$ is an artinian scheme. }
\end{prop}

The proof below gives a more precise proof of \cite[Theorem 3.6.4]{lz}. See also \cite[Proposition 8.1]{rtz}. We prepare some lemmas for the assertion above. For $\widetilde{\Lambda}\in \Vrt_H^g$ and an $\Fp[\gbar_{\widetilde{\Lambda}}]$-submodule $W$ of $\Omegabar_0(\widetilde{\Lambda})$, let $R_{W}\in \Fp[T]$ be the characteristic polynomial of $\gbar_{\widetilde{\Lambda}}$ on $W$. 

\begin{lem}\label{bjsm}
\emph{Let $\widetilde{\Lambda}\in \Vrt_H^g$. 
\begin{enumerate}
\item For $\widetilde{\Lambda}'\in \Vrt_H^g(\widetilde{\Lambda})$, we have an equality
\begin{equation*}
P_{\gbar_{\widetilde{\Lambda}}}=R_{(\widetilde{\Lambda}')^{\natural}/\widetilde{\Lambda}^{\natural}}
P_{\gbar_{\widetilde{\Lambda}'}}R_{(\widetilde{\Lambda}')^{\natural}/\widetilde{\Lambda}^{\natural}}^{*}. 
\end{equation*}
Hence we obtain a map
\begin{equation*}
\beta \colon \Vrt_H^g(\widetilde{\Lambda})\rightarrow \{R\in \Fp[T]\mid P_{\gbar_{\widetilde{\Lambda}}}\in (RR^*)\Fp[T]\};\widetilde{\Lambda}'\mapsto R_{(\widetilde{\Lambda}')^{\natural}/\widetilde{\Lambda}^{\natural}}. 
\end{equation*}
\item Let $Q\in \Fp[T]$ satisfying $Q\mid P_{\gbar_{\widetilde{\Lambda}}}$ and $Q^{*}=Q$. Then the preimage of $\{R\in \Fp[T]\mid P_{\gbar_{\widetilde{\Lambda}}}=RQR^{*}\}$ under $\beta$ equals $\{\widetilde{\Lambda}'\in \Vrt_H^g(\widetilde{\Lambda})\mid P_{\gbar_{\widetilde{\Lambda}'}}=Q\}$. 
\end{enumerate}}
\end{lem}

\begin{proof}
(i): We have inclusions of $\Fp[\gbar_{\widetilde{\Lambda}}]$-submodules
\begin{equation*}
0\subset (\widetilde{\Lambda}')^{\natural}/\widetilde{\Lambda}^{\natural}\subset \widetilde{\Lambda}'/\widetilde{\Lambda}^{\natural}\subset \Omegabar_0(\widetilde{\Lambda}). 
\end{equation*}
Let $R'$ be the characteristic polynomial of $\gbar_{\widetilde{\Lambda}'}$ on $\widetilde{\Lambda}/\widetilde{\Lambda}'$. Then we have $R'=R_{(\widetilde{\Lambda}')^{\natural}/\widetilde{\Lambda}^{\natural}}^{*}$ and
\begin{equation*}
P_{\gbar_{\widetilde{\Lambda}}}=R_{(\widetilde{\Lambda}')^{\natural}/\widetilde{\Lambda}^{\natural}}P_{\gbar_{\widetilde{\Lambda}'}}R'. 
\end{equation*}
Hence the assertion follows. 

(ii): Take $\widetilde{\Lambda}'\in \Vrt_H^g(\widetilde{\Lambda})$. Then it suffices to show that $P_{\gbar_{\widetilde{\Lambda}'}}=Q$ if and only if $P_{\gbar_{\widetilde{\Lambda}}}=R_{(\widetilde{\Lambda}')^{\natural}/\widetilde{\Lambda}^{\natural}}Q
R_{(\widetilde{\Lambda}')^{\natural}/\widetilde{\Lambda}^{\natural}}^{*}$. This follows from (i). 
\end{proof}

\begin{lem}\label{cmbt}
\emph{Let $\widetilde{\Lambda}\in \Vrt_H^g$. Assume that the assumption on $P_{\gbar_{\widetilde{\Lambda}}}$ in Proposition \ref{iart} holds. 
\begin{enumerate}
\item The polynomial $P_{\gbar_{\widetilde{\Lambda}'}}$ equals the minimal polynomial of $\gbar_{\widetilde{\Lambda}'}$ on $\Omegabar_0(\widetilde{\Lambda}')$. 
\item The map $\beta$ in Lemma \ref{bjsm} is bijective. 
\end{enumerate}}
\end{lem}

\begin{proof}
(i): This follows from Lemma \ref{bjsm} (i) and the assumption on $P_{\gbar_{\widetilde{\Lambda}}}$. 

(ii): First, note that we have a bijection
\begin{equation*}
\Vrt_H^g(\widetilde{\Lambda})\xrightarrow{\cong} \{W\subset \Omegabar_0(\widetilde{\Lambda})\colon \Fp\text{-subspace}\mid W\subset W^{\perp},\gbar_{\widetilde{\Lambda}}(W)=W\};\widetilde{\Lambda}'\mapsto (\widetilde{\Lambda}')^{\natural}/\widetilde{\Lambda}^{\natural}. 
\end{equation*}
Hence it suffices to show that the map
\begin{equation*}
\{W\subset \Omegabar_0(\widetilde{\Lambda})\colon \Fp \text{-subspace}\mid W\subset W^{\perp},\gbar_{\widetilde{\Lambda}}(W)=W\}
\rightarrow \{R\in \Fp[T]\mid P_{\gbar_{\widetilde{\Lambda}}}\in (RR^*)\Fp[T]\};W\mapsto R_W
\end{equation*}
is bijective. We have a decomposition of $\Omegabar(\widetilde{\Lambda})$ to generalized eigenspaces
\begin{equation*}
\Omegabar(\widetilde{\Lambda})=\bigoplus_{Q\in \Irr(P_{\gbar_{\widetilde{\Lambda}}})}\Omegabar_0(\widetilde{\Lambda})_{Q},\quad
\Omegabar_0(\widetilde{\Lambda})_{Q}:=\Ker(Q(\gbar_{\widetilde{\Lambda}})^{m(Q)}). 
\end{equation*}
Note that $\Omegabar_0(\widetilde{\Lambda})_{Q}$ is $\gbar_{\widetilde{\Lambda}}$-invariant for any $Q\in \Irr(P_{\gbar_{\widetilde{\Lambda}}})$. Since $P_{\gbar_{\widetilde{\Lambda}}}$ equals the minimal polynomial of $\gbar_{\widetilde{\Lambda}}$ on $\Omegabar_0(\widetilde{\Lambda})$, a $\gbar_{\widetilde{\Lambda}}$-invariant subspaces of $\Omegabar_0(\widetilde{\Lambda})_Q$ is of the form 
\begin{equation*}
W_{Q,a}:=\Ker(Q(\gbar_{\widetilde{\Lambda}})^{a})=\Ima(Q(\gbar_{\widetilde{\Lambda}})^{m(Q)-a}), 
\end{equation*}
where $0\leq a\leq m(Q)$. Therefore, $W_{Q,a}$ is a unique subspace of $\Omegabar_0(\widetilde{\Lambda})$ whose characteristic polynomial of $\gbar_{\widetilde{\Lambda}}$ equals $Q^a$. Moreover, we have $\dim_{\Fp}W_{Q,a}=a\cdot \deg Q$. On the other hand, for $Q\in \Irr(P_{\gbar_{\widetilde{\Lambda}}})$, $0\leq a\leq m(Q)$, $v\in W_{Q,a}$ and $w\in \Omegabar_0(\widetilde{\Lambda})$, we have
\begin{equation}\label{ipvn}
0=[Q(\gbar_{\widetilde{\Lambda}})^a(v),w]=[v,Q(\gbar_{\widetilde{\Lambda}}^{-1})^a(w)]
=[v,Q^*(\gbar_{\widetilde{\Lambda}})^a\gbar_{\widetilde{\Lambda}}^{-a\deg Q}(w)]. 
\end{equation}
Here the first equality follows from $v\in W_{Q,a}$. Hence we obtain the following: 
\begin{itemize}
\item If $Q,Q'\in \Irr(P_{\gbar_{\widetilde{\Lambda}}})$ and $Q\neq Q'$, then $\Omegabar_0(\widetilde{\Lambda})_{Q}\perp \Omegabar_0(\widetilde{\Lambda})_{Q'}$ (apply (\ref{ipvn}) to $a=m(Q)$ and $w\in \Omegabar_0(\widetilde{\Lambda})_{Q'}$). 
\item If $Q\in \Irr^{\sr}(P_{\gbar_{\widetilde{\Lambda}}})$ satisfies $2\nmid m(Q)$ and $0\leq a\leq (m(Q)-1)/2$, then $W_{Q_{\gbar_{\widetilde{\Lambda}}},a}$ is a totally isotropic subspace in $\Omegabar_0(\widetilde{\Lambda})_{Q_{\gbar_{\widetilde{\Lambda}}}}$ (apply (\ref{ipvn}) to $w\in \Omegabar_0(\widetilde{\Lambda})_{Q_{\gbar_{\widetilde{\Lambda}}}}$). 
\item If $Q\in \Irr^{\sr}(P_{\gbar_{\widetilde{\Lambda}}})$ satisfies $2\mid m(Q)$ and $0\leq a\leq m(Q_{\gbar_{\widetilde{\Lambda}}})/2$, then $W_{Q,a}$ is a totally isotropic subspace in $\Omegabar_0(\widetilde{\Lambda})_Q$ (apply (\ref{ipvn}) to $w\in W_{Q,a}$). 
\item If $Q\in \Irr^{\nsr}(P_{\gbar_{\widetilde{\Lambda}}})$, $0\leq a\leq m(Q)$ and $0\leq a'\leq m(Q)-a$, then $W_{Q,a}\oplus W_{Q^*,a'}$ is a totally isotropic subspace in $\Omegabar_0(\widetilde{\Lambda})_Q\oplus \Omegabar_0(\widetilde{\Lambda})_{Q^*}$ (apply (\ref{ipvn}) to $w\in W_{Q^*,a'}$). 
\end{itemize}

Now, take $R\in \Fp[T]$ satisfying $P_{\gbar_{\widetilde{\Lambda}}}=(RR^*)\Fp[T]$. If we write $R=\prod_{Q\in \Irr(P_{\gbar_{\widetilde{\Lambda}}})}Q^{m_{Q}}$, then we have
\begin{equation*}
\Ker(R(\gbar_{\widetilde{\Lambda}}))=\bigoplus_{Q\in \Irr(P_{\gbar_{\widetilde{\Lambda}}})}W_{Q,m_{Q}},
\end{equation*}
and it is a unique $\gbar_{\widetilde{\Lambda}}$-invariant totally isotropic subspace in $\Omegabar_0(\widetilde{\Lambda})$ satisfying $R_{\Ker(R(\gbar_{\widetilde{\Lambda}}))}=R$. Therefore, we obtain the bijectivity of $\beta$. 
\end{proof}

\begin{cor}\label{glbn}
\emph{Under the hypothesis in Lemma \ref{cmbt}, we have $p^{\Z}\backslash \BT_{H,\widetilde{\Lambda}'}^{\gbar_{\widetilde{\Lambda}'}}\neq \emptyset$ if and only if $P_{\gbar_{\widetilde{\Lambda}'}}\in \Irr(P_{\gbar_{\widetilde{\Lambda}}})$ and $m(P_{\gbar_{\widetilde{\Lambda}'}})$ is odd. In this case, we have 
\begin{equation*}
\#(p^{\Z}\backslash \BT_{H,\widetilde{\Lambda}}^{\gbar_{\widetilde{\Lambda}}})=\deg P_{\gbar_{\widetilde{\Lambda}'}}. 
\end{equation*}}
\end{cor}

\begin{proof}
By Lemma \ref{cmbt} (i) and Proposition \ref{neir}, it suffices to show that $P_{\gbar_{\widetilde{\Lambda}'}}$ is irreducible if and only if $P_{\gbar_{\widetilde{\Lambda}'}}\in \Irr^{\sr}(P_{\gbar_{\widetilde{\Lambda}}})$ and $m(P_{\gbar_{\widetilde{\Lambda}'}})$ is odd. Suppose that $P_{\gbar_{\widetilde{\Lambda}'}}$ is irreducible. Then we have $P_{\gbar_{\widetilde{\Lambda}'}}\in \Irr^{\sr}(P_{\gbar_{\widetilde{\Lambda}}})$ by Lemma \ref{chsr}. On the other hand, $2\nmid m(Q)$ follows from Lemma \ref{bjsm}. 
\end{proof}

\begin{proof}[Proof of Proposition \ref{iart}]
First, assume that (i) holds. We prove (ii). By Proposition \ref{gfbt}, there is $\widetilde{\Lambda}'\in \Vrt_H^g(\widetilde{\Lambda})$ satisfying $p^{\Z}\backslash \BT_{H,\widetilde{\Lambda}'}^{\gbar_{\widetilde{\Lambda}'}}\neq \emptyset$. Hence we have $P_{\gbar_{\widetilde{\Lambda}'}}\in \Irr^{\sr}(P_{\gbar_{\widetilde{\Lambda}}})$ and $m(P_{\gbar_{\widetilde{\Lambda}'}})$ by Corollary \ref{glbn}. Moreover, by Lemma \ref{bjsm}, $P_{\gbar_{\widetilde{\Lambda}'}}$ is a unique element of $\Irr(P_{\gbar_{\widetilde{\Lambda}}})$ such that the multiplicity in $P_{\gbar_{\widetilde{\Lambda}}}$ is odd. Therefore we obtain (ii). 

Next, assume that (ii) holds. We prove (i), and compute $\#(p^{\Z}\backslash \M_{H,\widetilde{\Lambda}}^g(\Fpbar))$. Let $Q_{\gbar_{\widetilde{\Lambda}}}\in \Irr^{\sr}(P_{\gbar_{\widetilde{\Lambda}}})$ be as in (ii). By Proposition \ref{gfbt} and Corollary \ref{glbn}, the following hold for $\widetilde{\Lambda}'\in \Vrt_H^g(\widetilde{\Lambda})$: 
\begin{itemize}
\item we have $p^{\Z}\backslash \BT_{H,\widetilde{\Lambda}'}^{\gbar_{\widetilde{\Lambda}'}}\neq \emptyset$ if and only if $P_{\gbar_{\widetilde{\Lambda}'}}=Q_{\gbar_{\widetilde{\Lambda}}}$, 
\item if $\widetilde{\Lambda}'$ satisfies the equivalent conditions above, then we have $\#(p^{\Z}\backslash \BT_{H,\widetilde{\Lambda}'}^{\gbar_{\widetilde{\Lambda}'}})=\deg Q_{\gbar_{\widetilde{\Lambda}}}$. 
\end{itemize}
Moreover, $\beta$ induces a bijection
\begin{equation*}
\{\widetilde{\Lambda}'\in \Vrt_H^g(\widetilde{\Lambda})\mid P_{\gbar_{\widetilde{\Lambda}'}}=Q_{\gbar_{\widetilde{\Lambda}}} \} \xrightarrow{\cong}\{R\in \Fp[T]\mid P_{\gbar_{\widetilde{\Lambda}}}=RQ_{\gbar_{\widetilde{\Lambda}}}R^*\}. 
\end{equation*}
by Lemmas \ref{cmbt} (ii) and \ref{bjsm} (ii). Hence it suffices to show the equality 
\begin{equation*}
\#\{R\in \Fp[T]\mid P_{\gbar_{\widetilde{\Lambda}}}=RQ_{\gbar_{\widetilde{\Lambda}}}R^*\}=\prod_{[R]\in \Irr^{\nsr}(P_{\gbar_{\widetilde{\Lambda}}})/\sim}(1+m(R))\in \Zpn, 
\end{equation*}
The set $\#\{R\in \Fp[T]\mid P_{\gbar_{\widetilde{\Lambda}}}=RQ_{\gbar_{\widetilde{\Lambda}}}R^*\}$ consists of elements of the form
\begin{equation*}
\prod_{[R]\in \Irr^{\nsr}(P_{\gbar_{\widetilde{\Lambda}}})/\sim}(R^{a_{[R]}}R^{*m(R)-a_{[R]}})\cdot \prod_{[R']\in \Irr^{\sr}(P_{\gbar_{\widetilde{\Lambda}}})\setminus \{Q_{\gbar_{\widetilde{\Lambda}}}\}}R'^{m(R')/2}, 
\end{equation*}
where $0\leq a_{[R]}\leq m(R)$ for any $[R]\in \Irr^{\nsr}(P_{\gbar_{\widetilde{\Lambda}}})$. In particular, it is non-empty. Hence we obtain the desired equality. 
\end{proof}

Next, we compute $\length_{\Fpbar}\O_{p^{\Z}\backslash \M_{H,\widetilde{\Lambda}}^g}$ for $\widetilde{\Lambda}\in \Vrt_H^g$ and $x\in p^{\Z}\backslash \M_{H,\widetilde{\Lambda}}^g$. It contributes to computation of the arithmetic intersection number. See Proposition \ref{isca}. 
We recall the action of $\gbar_{\widetilde{\Lambda}}$ on $\Sbar_{\Omegabar_0(\widetilde{\Lambda})}$. It is given by $\gbar_{\widetilde{\Lambda}} \cdot(\cL',\cL)=(\gbar_{\widetilde{\Lambda}}(\cL'),\gbar_{\widetilde{\Lambda}}(\cL))$ for $(\cL',\cL)\in \Sbar_{\Omegabar_0(\widetilde{\Lambda})}$. 

\begin{prop}\label{mhso}
\emph{There is an isomorphism $p^{\Z}\backslash \M_{H,\widetilde{\Lambda}}^g\cong \Sbar_{\Omegabar_0(\widetilde{\Lambda})}^{\gbar_{\widetilde{\Lambda}}}$. }
\end{prop}

\begin{proof}
There is an isomorphism $p^{\Z}\backslash \M_{H,\widetilde{\Lambda}}\cong \Sbar_{\Omegabar_0(\widetilde{\Lambda})}$. See \cite[Theorem 3.9]{hp}. Moreover, the isomorphism above is compatible with the actions of $g$ and $\gbar_{\widetilde{\Lambda}}$ by definition. Hence the assertion follows. 
\end{proof}

By Proposition \ref{mhso}, to compute $\length_{\Fpbar}\O_{p^{\Z}\backslash \M_{H,\widetilde{\Lambda}}^g}$, it suffices to compute $\length_{\Fpbar}\O_{\Sbar_{\Omegabar_0(\widetilde{\Lambda})}^{\gbar_{\widetilde{\Lambda}}},x}$ for $\widetilde{\Lambda}\in \Vrt_H^g$ and $x\in \Sbar_{\Omegabar_0(\widetilde{\Lambda})}^{g_{\widetilde{\Lambda}}}$. However, it is already computed by \cite[5.2, 5.3, 5.4]{lz}. 

Let us state the result of computation in \cite{lz}. Let $x_0=(\cL'_{0},\cL_{0})\in \Sbar_{\Omegabar_0(\widetilde{\Lambda})}^{\gbar_{\widetilde{\Lambda}}}(\Fpbar)$. Then $\Phi(\cL_0)$ is stable under $\gbar_{\widetilde{\Lambda}}$ since it commutes with $\Phi$. Define $\lambda_{\widetilde{\Lambda}} \in \Fpbar$ and $c_{\widetilde{\Lambda}} \in \Znn$ as follows: 
\begin{itemize}
\item $\lambda_{\widetilde{\Lambda}} \in \Fpbar$ is the (non-zero) unique eigenvalue of $\gbar_{\widetilde{\Lambda}}$ on $\Phi(\cL_{0})/\cL'_{0}$, 
\item $c_{\widetilde{\Lambda}}$ is the size of the Jordan block of $\gbar_{\widetilde{\Lambda}} \vert_{\Phi(\cL_{0})}$of the eigenvalue $\lambda_{\widetilde{\Lambda}}$ (note that $1\leq c_{\widetilde{\Lambda}}\leq d$). 
\end{itemize}

The following follows from exactly the same argument as the proof of \cite[Corollary 5.4.2]{lz}: 

\begin{thm}\label{lzmr}
\emph{
\begin{enumerate}
\item Assume $p>c_{\widetilde{\Lambda}}$. Then there is an isomorphism between the complete local ring of $\Sbar_{\widetilde{\Lambda}}^{\gbar_{\widetilde{\Lambda}}}$ at $x_0$ and $\Fpbar[X]/(X^{c_{\widetilde{\Lambda}}})$. 
\item Assume that $P_{\gbar_{\widetilde{\Lambda}}}$ equals the minimal polynomial of $\gbar_{\widetilde{\Lambda}}$ on $\Omegabar_0(\widetilde{\Lambda})$. Then we have $c_{\widetilde{\Lambda}}=(m(Q_{\gbar_{\widetilde{\Lambda}}})+1)/2$. Here $Q_{\gbar_{\widetilde{\Lambda}}}$ is the polynomial appeared in Proposition \ref{iart} (ii). 
\end{enumerate}}
\end{thm}

Now we define the notion of minusculeness for $g$, and compute the intersection multiplicity for minuscule $g$. 

\begin{dfn}\label{rsmd}
Let $g\in J_H^0(\Qp)$. 
\begin{itemize}
\item An element $g$ is \emph{regular semi-simple} if $\L_{\Q}^{\Phi}=\bigoplus_{i=0}^{5}\Qp(g^i\cdot y_1)$. 
\item A regular semi-simple element $g$ is \emph{minuscule} if $L(g)^{\natural}\in \Vrt_H$. 
\end{itemize}
\end{dfn}

We suppose that $g$ is regular semi-simple and minuscule. Then we have $L(g)^{\natural}\in \Vrt_H^g(L(g)^{\natural})$ by Lemma \ref{lggs}. Moreover, put $\Omegabar_0(g):=\Omegabar_0(L(g)^{\natural})$, $\gbar:=\gbar_{L(g)^{\natural}} \in \SO(\Omegabar_0(g))$, $P_g:=P_{\gbar_{L(g)^{\natural}}}\in \Fp[T]$, $\lambda_g:=\lambda_{L(g)^{\natural}}$ and $c_g:=c_{L(g)^{\natural}}$. 

\begin{lem}\label{chmp}
\emph{The polynomial $P_g$ equals the minimal polynomial of $\gbar$. }
\end{lem}

\begin{proof}
Note that we have $L(g)=\bigoplus_{i=0}^{5}\Zp(g^i\cdot y_1)$ since $g$ is regular semi-simple and minuscule. It suffices to show that $L(g)^{\natural}$ is $g$-cyclic, that is, there is $u\in L(g)^{\natural}$ such that $L(g)^{\natural}=\bigoplus_{i=0}^{5}\Zp (g^i\cdot u)$. This follows from the same argument as in the proof of \cite[Lemma 5.3]{rtz}. 
\end{proof}

Summarizing the results above, we obtain the following: 

\begin{thm}\label{mnin}
\emph{Assume $g\in J_H^0(\Qp)$ is regular semi-simple, minuscule and satisfies $\M_H^g\neq \emptyset$. 
\begin{enumerate}
\item The following are equivalent: 
\begin{itemize}
\item We have $\Delta \cap g\Delta \neq \emptyset$. 
\item There is a unique $Q_{g}\in \Irr^{\sr}(P_{g})$ such that $m(Q_{g})$ is odd.
\end{itemize}
If the conditions above hold, then we have
\begin{equation*}
\#(\Delta \cap g\Delta)(\Fpbar)=\deg Q_{g}\prod_{[R]\in \Irr^{\nsr}(P_{g})/\sim}(1+m(R))<\infty. 
\end{equation*}
\item If $g$ satisfies the condition (i), then we have an equality
\begin{equation*}
\langle \Delta,g\Delta \rangle=\deg Q_g\frac{m(Q_g)+1}{2}\prod_{[R]\in \Irr^{\nsr}(P_g)/\sim}(1+m(R)). 
\end{equation*}
\end{enumerate}}
\end{thm}

\begin{proof}
(i): This follows from Proposition \ref{iart} and Lemma \ref{chmp}. 

(ii): By Lemma \ref{dlgh}, Proposition \ref{nhtr} (i) and $\cZ(L(g))=\M_{G,L(g)^{\natural}}$, we have 
\begin{equation*}
\Delta \cap g\Delta \cong p^{\Z}\backslash \M_{H,L(g)^{\natural}}^{g}. 
\end{equation*}
Hence $\langle \Delta,g\Delta \rangle$ is defined by (i) and Proposition \ref{nhtr} (ii). Moreover, by Proposition \ref{mhso}, we have
\begin{equation*}
\langle \Delta,g\Delta \rangle=\sum_{x\in \Sbar_{\Omegabar_0(g)}^{\gbar}(\Fpbar)} \length_{\Fpbar}\O_{\Sbar_{\Omegabar_0(g)}^{\gbar},x}\in \Z. 
\end{equation*}
On the other hand, by Theorem \ref{lzmr} (ii) and Lemma \ref{chmp}, we have $c_g=(m(Q_g)+1)/2$. We claim that $c_g\leq 2$. Indeed, since $\deg Q_g\in \{2,4,6\}$ by the proof of Proposition \ref{iart}, we have $m(Q_g)\leq 6/2=3$. Hence we have 
\begin{equation*}
c_g=(m(Q_g)+1)/2\leq (3+1)/2=2
\end{equation*}
as claimed. 

We compute the intersection multiplicity. Since $\Sbar_{\Omegabar_0(g)}^{\gbar}$ is an artinian scheme, $\length_{\Fpbar}\O_{\Sbar_{\Omegabar_0(g)}^{\gbar},x}$ equals the length of the complete local ring at $\Sbar_{\Omegabar_0(g)}^{\gbar}$ at $x$ for any $x\in \Sbar_{\Omegabar_0(g)}^{\gbar}$, which equals $c_g=(m(Q_g)+1)/2$ by Theorem \ref{lzmr} (ii). On the other hand, by Proposition \ref{iart}, we have 
\begin{equation*}
\# \Sbar_{\Omegabar_0(g)}^{\gbar}(\Fpbar)=\deg Q_g\prod_{[R]\in \Irr^{\nsr}(P_g)/\sim}(1+m(R)). 
\end{equation*}
Therefore the assertion follows. 
\end{proof}


\appendix
\section{Complement of the proof of \cite[Corollary 2.14]{hp}}\label{hpcr}

In this appendix, we provide a complement of the proof of \cite[Corollary 2.14]{hp} (it is introduced as Theorem \ref{hppc} in Section \ref{rtpt}). We use the notation in Sections \ref{excp} and \ref{rtpt}. To prove the assertion, they asserted in \cite[Lemma 2.15]{hp} the bijectivity of the map
\begin{equation*}
p^{\Z}\backslash \{\text{nearly self-dual lattices in }\D_{\Q,K}\} \rightarrow \{\text{self-dual lattices in }\L_{\Q,K}\};M\mapsto L^{\sharp}(M), 
\end{equation*}
where $L^{\sharp}(M):=\{v\in \L_{\Q,K} \mid v(M)\subset M\}$. We claim that the lattice $L^{\sharp}(M)$ in $\L_{\Q,K}$ is not necessarily self-dual even when $M$ is a nearly self-dual lattice in $\D_{\Q,K}$. To avoid this possibility, we restrict the left-hand side of the map above. 

We also denote an element of $H(K)$ as $(h_0,h_1)$, where $h_i\in \GL(\D_{\Q,i}\otimes_{K_0}K)$ for $i\in \{0,1\}$ by the same manner as in Section \ref{excp}. 

\begin{lem}\label{gu0d}
\emph{Denote by $C_K$ the stabilizer of $\D \otimes_{W}W(k)$ in $H(K)$. Then we have an equality
\begin{equation*}
H^0(K)C_K=\{h=(h_0,h_1)\in H(K)\mid \ord_p\det(h_0)=\ord_p\det(h_1)\}. 
\end{equation*}}
\end{lem}

\begin{proof}
We have an isomorphism
\begin{equation*}
H(K)\xrightarrow{\cong} \GL(\D_{\Q,0}\otimes_{K_0} K)\times K^{\times};h\mapsto (h_0,\sml(h)). 
\end{equation*}
The proof is the same as the isomorphy of the first homomorphism as in Lemma \ref{htis} (i). 
It induces the following isomorphisms: 
\begin{gather*}
H^0(K)\cong \{(h_0,c)\in \GL(\D_{\Q,0}\otimes_{K_0} K)\times K^{\times}\mid c^2=\det(h_0)\},\\
C_K \cong \GL(\D_0\otimes_W W(k))\times W(k)^{\times}. 
\end{gather*}
Therefore we obtain
\begin{equation*}
H^0(K)C_K\cong \{(h_0,c)\in \GL(\D_{\Q,0}\otimes_{K_0} K)\times K^{\times}\mid 2\ord_p (c)=\ord_p(\det(h_0))\}, 
\end{equation*}
which is equivalent to the assertion. 
\end{proof}

\begin{ex}
For $i\in \Z$, let 
\begin{equation*}
M_i:=(p^i\D_0\oplus p^{-i}\D_1)\otimes_W W(k). 
\end{equation*}
Then $M_i$ is self-dual. On the other hand, we have $M_i=h^i(\D) \otimes_W W(k)$, where $h:=(p\id_{\D_{\Q,0}},p^{-1}\id_{\D_{\Q,1}})$. Then $h^i\in H^0(K_0)C$ if and only if $i=0$. Put $\L:=\{v\in \L_{\Q}\mid v(\D)\subset \D \}$. Then we have
\begin{equation*}
L^{\sharp}(M_i)=p^{2\left|i\right|}\L \otimes_W W(k),
\end{equation*}
which is self-dual if and only if $i=0$. 
\end{ex}

Next, we modify the bijection $M\mapsto L^{\sharp}(M)$. Write $Z_K$ for the center of $H(K)$. Then we have an injection
\begin{equation*}
\widetilde{\lambda}\colon H^0(K)C_K Z_K/C_K Z_K\rightarrow p^{\Z}\backslash \{\text{nearly self-dual lattices in $\D_{\Q,K}$}\}
;h\mapsto h(\D \otimes_W W(k)). 
\end{equation*}

\begin{dfn}
A \emph{special nearly self-dual lattice} is a lattice $M$ in $\D_{\Q,K}$ such that the homothety class of $M$ is in the image of the map $\widetilde{\lambda}$. 
\end{dfn}
The definition amounts to saying that there is an element $h$ in $H^0(K)$ (or $H^0(K)C_K$) such that $M=h(\D \otimes_{W}W(k))$. By definition, special nearly self-dual lattices are $\Zps$-invariant and nearly self-dual in the sense of \cite[\S 2.5]{hp}. 

We state the modified version of \cite[Lemma 2.15]{hp}. The proof is the same as the original one. 

\begin{prop}\label{sdmd}
\emph{The map
\begin{equation*}
p^{\Z}\backslash \{\text{special nearly self-dual lattices in }\D_{\Q,K}\} \rightarrow \{\text{self-dual lattices in }\L_{\Q,K}\};
M\mapsto L^{\sharp}(M)
\end{equation*}
is bijective. }
\end{prop}

By the modification above, \cite[Lemma 2.16]{hp} holds only for special nearly self-dual lattices. The precise statement is as follows: 

\begin{lem}\label{sdlg}
\emph{Let $M$ be a special nearly self-dual lattice in $\D_{\Q,K}$, and put $L^{\sharp}:=L^{\sharp}(M)$, where $L^{\sharp}(M)$ is the self-dual lattice in $\L_{\Q,K}$ corresponding to the bijection $M\mapsto L^{\sharp}(M)$ in Proposition \ref{sdmd}. Let $i\in \Z$ satisfying $M^{\vee}=p^iM$, and we regard $M$ as a self-dual symplectic space over $W(k)$ by $p^i(\,,\,)$. 
\begin{enumerate}
\item For an isotropic line $l$ in $L^{\sharp}\otimes_{W(k)}k$, 
\begin{equation*}
\sM(l):=\{x\in M\otimes_{W(k)}k\mid v(x)=0\text{ for any }v\in l\}
\end{equation*}
is a $\Zps$-stable Lagrangian subspace in $M\otimes_{W(k)}k$. Hence we obtain a map
\begin{equation*}
\{\text{isotropic lines in }L^{\sharp}\otimes_{W(k)}k\} \rightarrow 
\{\Zps \text{-stable Lagrangian subspaces in }M\otimes_{W(k)}k\};l\mapsto \sM(l). 
\end{equation*}
\item The map $l\mapsto \sM(l)$ in (i) is bijective. The inverse map is given by 
\begin{equation*}
\sM \mapsto l(\sM):=\{v\in L^{\sharp}\otimes_{W(k)}k\mid v(\sM)=0\}. 
\end{equation*}
Moreover, we have 
\begin{equation*}
l(\sM)^{\perp}=\{v\in L^{\sharp}\otimes_{W(k)}k\mid v(\sM)\subset \sM \}. 
\end{equation*}
\end{enumerate}}
\end{lem}

We can also prove the following by the same proof of \cite[Lemma 2.16]{hp}: 

\begin{lem}\label{tilg}
\emph{Let $K'/K$ be a field extension. 
\begin{enumerate}
\item For an isotropic line $l$ in $\L_{\Q,K'}:=\L_{\Q,K}\otimes_{K}K'$, 
\begin{equation*}
\sM(l):=\{x\in \D_{\Q,K'}\mid v(x)=0\text{ for any }v\in l\}
\end{equation*}
is a $\Zps$-stable Lagrangian subspace in $\D_{\Q,K'}$. Hence we obtain a map
\begin{equation*}
\{\text{isotropic lines in }\L_{\Q,K'}\} \rightarrow 
\{\Zps \text{-stable Lagrangian subspaces in }\D_{\Q,K'}\};l\mapsto \sM(l). 
\end{equation*}
\item The map $l\mapsto \sM(l)$ in (i) is bijective. The inverse map is given by 
\begin{equation*}
\sM \mapsto l(\sM):=\{v\in \L_{\Q,K'}\mid v(\sM)=0\}. 
\end{equation*}
Moreover, we have 
\begin{equation*}
l(\sM)^{\perp}=\{v\in \L_{\Q,K'}\mid v(\sM)\subset \sM \}. 
\end{equation*}
\end{enumerate}}
\end{lem}

Lastly, we check that \cite[Corollary 2.14]{hp} is true. 

\begin{prop}\label{dsns}
\emph{For a Dieudonn{\'e} lattice $M$ in $\D_{\Q,K}$, $F^{-1}(pM)$ is special nearly self-dual. }
\end{prop}

\begin{proof}
Put $\D'_{W(k)}:=b(\D \otimes_{W}W(k))$, where $b\in G(K_0)$ is defined as in Section \ref{rzdt}. Write $F^{-1}(pM)=h(\D'_{W(k)})=hb(\D \otimes_{W}W(k))$ where $h\in H(K)$. Since $b\in H^0(K)$ by definition, it suffices to prove $h\in H^0(K)C_K$. Put $g:=h^{-1}b\sigma(h)b^{-1}\in H(K)$. Then $g$ induces an isomorphism 
\begin{equation*}
F_{*}(\D'_{W(k)})\xrightarrow{\cong} (h^{-1}\circ F\circ h)_{*}(\D'_{W(k)})=h^{-1}(F_{*}(h(\D'_{W(k)}))). 
\end{equation*}
On the other hand, for $i\in \{0,1\}$, put $\D'_{W(k),i}:=\varepsilon_{i}\D'_{W(k)}$. Then we have
\begin{equation*}
\length_{W(k)}(\D'_{W(k),1-i}/F_{*}(\D'_{W(k),i}))=\length_{W(k)}(\D'_{W(k),1-i}/h^{-1}(F_{*}(h(\D'_{W(k),i}))))=2
\end{equation*}
by the Kottwitz condition. See also the proof of \cite[Proposition 2.10]{hp}. Therefore we have $\ord_p\det(g_i)=0$ if we write $g=(g_0,g_1)$. Taking determinant of the defining equation of $g$, we obtain
\begin{equation*}
(\det(g_0),\det(g_1))=(\det(h_0),\det(h_1))^{-1}\sigma(\det(h_0),\det(h_1))=(\det(h_0)^{-1}\sigma(\det(h_1)),\det(h_1)^{-1}\sigma(\det(h_0)))
\end{equation*}
in $\Qps \otimes_{\Qp} K_0\cong K_0\times K_0$, which concludes that $\ord_p\det(h_0)=\ord_p\det(h_1)$. Therefore the assertion follows from Lemma \ref{gu0d}. 
\end{proof}

Now we can give a proof of \cite[Corollary 2.14]{hp}. Take a Dieudonn{\'e} lattice in $\D_{\Q,K}$. Then Proposition \ref{dsns} implies that $F^{-1}(pM)$ is special nearly self-dual. We apply Lemma \ref{sdlg} to $F^{-1}(pM)$, and the assertion therefore follows from the same proof as the original one. 

\section{Reducedness of the Bruhat--Tits strata of $\M_{H}$}\label{rdhp}

In this appendix, we prove the following: 

\begin{thm}\label{redh}
\emph{For $\widetilde{\Lambda}\in \Vrt_H$, the formal scheme $\M_{H,\widetilde{\Lambda}}$ is a reduced scheme of characteristic $p$. }
\end{thm}

To prove the assertion above, we introduce some notations. 

\begin{dfn}(cf.~\cite[Definition 4.1.6]{lz}) 
Let $\sC$ be the category defined as follows: 
\begin{itemize}
\item objects in $\sC$ are triples $(\O,s,\gamma)$, where $\O$ is a local artinian $W$-algebra, $s$ is a $W$-homomorphism $\O \rightarrow \Fpbar$, and $\gamma$ is a divided power structure on $\Ker(s)$ (see \cite[Definitions 3.1, Definition 3.27]{bo}), 
\item a morphism of objects in $\sC$ from $(\O_1,s_1,\gamma_1)$ to $(\O_2,s_2,\gamma_2)$ is a $W$-homomorphism $g\colon \O_1\rightarrow \O_2$ such that $s_1=s_2\circ g$ and $(g\vert_{\Ker(s_1)})\circ \gamma_1=\gamma_2 \circ (g\vert_{\Ker(s_1)})$. 
\end{itemize}
\end{dfn}

\begin{rem}
\begin{enumerate}
\item For an object $(\O,s,\gamma)$ in $\sC$, the homomorphism $s$ is surjective since it is a $W$-homomorphism. Hence $\Ker(s)$ is the unique maximal (or prime) ideal of $\O$, and $s$ induces an isomorphism between the residue field of $\O$ and $\Fpbar$. 
\item For a morphism $g\colon (\O_1,s_1,\gamma_1)\rightarrow(\O_2,s_2,\gamma_2)$ in $\sC$, we have $g(\Ker(s_1))=\Ker(s_2)$, that is, $g$ is a local homomorphism. 
\end{enumerate}
\end{rem}

We simply denote $(\O,s,\gamma)\in \sC$ by $\O$ as long as there is no confusion. 

\begin{dfn}
Let $(L,Q)$ be a quadratic space over $W$. 
\begin{itemize}
\item Let $\Isot_{L}$ be the functor which parametrizes all isotropic lines in $L\otimes_{W}R$ for any $W$-algebra $R$. Here, an isotropic line $l$ in $L\otimes_{W}R$ is an $R$-submodule of $L\otimes_{W}R$ which is locally a direct summand of rank $1$ and $Q(l)=0$. Note that the functor $\Isot_{L}$ is representable by a projective scheme. 
\item Suppose that $R$ is a local ring whose residue field is $\Fpbar$, and let $s\colon R\rightarrow \Fpbar$ be the canonical homomorphism. Then, for an isotropic line $l$ in $L\otimes_{W}\Fpbar$, put
\begin{equation*}
\Lift_{R}(l):=\Isot_L(s)^{-1}(l)\subset \Isot_{L}(R). 
\end{equation*}
\end{itemize}
\end{dfn}

Under the notations above, we prove the following: 

\begin{prop}\label{dmil}
\emph{Let $x\in \M_H(\Fpbar)$. We denote by $\widehat{\M}_{H,x}$ the completion of $\M_H$ at $x$. Let $M$ be the Dieudonn{\'e} lattice in $\D_{\Q}$ corresponding to $x$ under the bijection in Theorem \ref{rtdu}, and $L$ the image of $M$ under the bijection $M\mapsto L(M)$ in Theorem \ref{hppc}. Moreover, put $\Fil^1\Phi(L):=p(L+\Phi(L))/p\Phi(L)$ (it is an isotropic line in $\Phi(L)\otimes_{W}\Fpbar \cong \Phi(L)/p\Phi(L)$ by Lemma \ref{sdlg}). Then, for $\O \in \sC$, there is a bijection
\begin{equation*}
f_{\O}\colon \widehat{\M}_{H,x}(\O)\xrightarrow{\cong}\Lift_{\O}(\Fil^1\Phi(L))
\end{equation*}
satisfying the following properties. 
\begin{enumerate}
\item For a morphism $g\colon \O \rightarrow \O'$, we have 
\begin{equation*}
f_{\O'}\circ \widehat{\M}_{H,x}(g)=\Isot_g(\Phi(L))\circ f_{\O}\colon \widehat{\M}_{H,x}(\O)\rightarrow \Lift_{\O'}(\Fil^1\Phi(L)). 
\end{equation*}
\item Let $\bv$ be a $\Zp$-submodule in $\L_{\Q}^{\Phi}$ satisfying $x\in \cZ(\bv)$, where $\cZ(\bv)$ is a closed defined in Definition \ref{spcy}. Denote by $\widehat{\cZ}(\bv)_x$ the completion of $\cZ(L)$ at $x$. Then the bijection $f_{\O}$ induces a bijection
\begin{equation*}
\widehat{\cZ}(\bv)_x(\O)\xrightarrow{\cong}\{\widetilde{l}\in \Lift_{\O}(\Fil^1\Phi(L))\mid \widetilde{l}\perp \bv_{\O}\text{ in }\Phi(L)\otimes_{W}\O \}, 
\end{equation*}
where $\bv_{\O}$ is the $\O$-submodule of $\Phi(L)\otimes_{W}\O$ generated by the image of $\bv$ under $\Phi(L)\rightarrow \Phi(L)\otimes_{W}\O$. 
\end{enumerate}}
\end{prop}

The assertion above is a variant of \cite[Theorem 4.1.7]{lz}. Hence, if Proposition \ref{dmil} is true, then Theorem \ref{redh} follows from the same argument as the proof of \cite[Theorem 4.2.11]{lz}. 

To prove Proposition \ref{dmil}, we generalize Lemma \ref{sdlg}. For this, we moreover prepare some notations. 

\begin{dfn}
Let $M$ and $i$ be as in Lemma \ref{sdlg}. 
\begin{itemize}
\item Let $\Lag_M^{\Zps}$ be a functor which parametrizes $\Zps$-stable Lagrangian subspaces in $M\otimes_{W}R$ for any $W$-algebra $R$. Here, a Lagrangian subspace $\sM$ in $M\otimes_{W}R$ is an $R$-submodule of $M\otimes_{W}R$ which is locally a direct summand of rank $4$ and $p^i(\sM,\sM)=0$. Note that the functor $\Lag_M^{\Zps}$ is representable by a projective scheme. 
\item Suppose that $R$ is a local ring whose residue field is $\Fpbar$, and let $s\colon R\rightarrow \Fpbar$ be the canonical homomorphism. Then, for an isotropic line $\sM$ in $M\otimes_{W}\Fpbar$, put
\begin{equation*}
\Lift_{R}^{\Zps}(\sM):=\Lag_{M}^{\Zps}(s)^{-1}(\sM)\subset \Lag_{M}^{\Zps}(R). 
\end{equation*}
\end{itemize}
\end{dfn}

\begin{lem}\label{liat}
\emph{Let $M$ and $L$ be as in Lemma \ref{sdlg}. 
\begin{enumerate}
\item For a $W$-algebra $R$ and $l\in \Isot_{R}(L^{\sharp})$, 
\begin{equation*}
\sM(l):=\{x\in M\otimes_{W}R \mid v(x)=0\text{ for any }v\in l\}
\end{equation*}
is a $\Zps$-stable Lagrangian subspace in $M\otimes_{W}R$. Hence $l\mapsto \sM(l)$ induces a morphism of $W$-schemes
\begin{equation*}
f\colon \Isot_{L^{\sharp}}\rightarrow \Lag_{M}^{\Zps}. 
\end{equation*}
\item The morphism $f$ in (i) is an isomorphism. The inverse is given by 
\begin{equation*}
\Lag_{M}^{\Zps}(R)\rightarrow \Isot_{L^{\sharp}}(R);\sM \mapsto l(\sM):=\{v\in L^{\sharp}\otimes_{W}R \mid v(\sM)=0\}
\end{equation*}
for any $W$-algebra $R$. Moreover, we have
\begin{equation*}
l(\sM)^{\perp}=\{v\in L^{\sharp}\otimes_{W}R \mid v(\sM)\subset \sM \}. 
\end{equation*}
\end{enumerate}}
\end{lem}

To prove Lemma \ref{liat}, we need a preparation on certain $C(L^{\sharp})$-modules. As pointed out in the proof of \cite[Lemma 2.16]{hp}, the inclusion $\L_{\Q}\subset \End(\D_{\Q})$ induces an isomorphism 
\begin{equation*}
i'\colon C(L^{\sharp})\xrightarrow{\cong} \End(M). 
\end{equation*}
We regard $M$ as a left $C(L^{\sharp})$-module under the isomorphism above. On the other hand, we regard $C(L^{\sharp})$ as a left $C(L^{\sharp})$-module by the left multiplication. 

\begin{lem}\label{mreq}
\emph{Under the notation in Lemma \ref{liat}, there is an isomorphism $C(L^{\sharp})\cong M^{\oplus 8}$ of $C(L^{\sharp})$-modules. }
\end{lem}

\begin{proof}
Fix a $W$-basis $e''_1,\ldots,e''_8$ of $M$, and let $\varepsilon_1\in C(L^{\sharp})$ be the element corresponding to the endomorphism
\begin{equation*}
M\rightarrow M;e''_i\mapsto \delta_{i1}e''_i
\end{equation*}
under the isomorphism $i'\colon C(L^{\sharp})\xrightarrow{\cong} \End(M)$. Then, both $\varepsilon_1C(L^{\sharp})$ and $M$ are free $W$-modules of rank $8$. Hence there is an isomorphism $\varepsilon_1C(L^{\sharp})\cong M$ of $W$-modules, which associates a desired isomorphism by the Morita equivalence. 
\end{proof}

\begin{proof}[Proof of Lemma \ref{liat}]
(i): By localizing $R$, we may assume that $l$ is free over $R$. Take $v\in l$ so that $l=Rv$. The stability of $\sM(l)$ under the $\Zps$-action is a consequence of the equality $\iota_0(a)\circ v=v\circ \iota_0(a^{*})$ for any $a\in \Zps$. Next, we prove the equality
\begin{equation}\label{kieq}
\sM(l)=v(M \otimes_W R),
\end{equation}
and that it is a totally isotropic $R$-direct summand of $\D \otimes_{W}R$ of rank $4$. Since $M$ is special nearly self-dual, there is $h\in H^0(K_0)$ such that $M=h(\D)$. Hence we may assume $M=\D$. Write 
\begin{equation*}
v=\sum_{i=1}^{6}a_ix_i
\end{equation*}
for $a_i\in R$ (see Definition for the definition of $x_i$). Then we have 
\begin{equation}\label{aimp}
a_1a_2+a_3a_4+a_5a_6=0. 
\end{equation}
On the other hand, since $l$ is an $R$-direct summand of $\L \otimes_W R$, we have $a_i\in R^{\times}$ for some $i$. Here, we suppose that $a_1\in R^{\times}$. Other cases are similar. We may assume $a_1=1$. Then, by using (\ref{aimp}), we obtain
\begin{align*}
\sM(l)&=v(\D \otimes_W R) \\
&=R(e_2+a_3e_3+a_5e_4)\oplus R(-e_1+a_6e_3-a_4e_4)\oplus R(-a_4f_1-a_5f_2+f_4)\oplus R(-a_6f_1+a_3f_2-f_3),
\end{align*}
which is an $R$-direct summand of $\D \otimes_{W}R$ of rank $4$. Hence the assertion follows. 

(ii): The injectivity follows from Lemmas \ref{sdlg} (ii), \ref{tilg} (ii) and the Nakayama's lemma. Next, we prove the surjectivity. Put 
\begin{equation*}
\GU(M)(R):=\{(h,c)\in \GL_{\Zps \otimes_{\Zp}R}(M\otimes_W R)\mid p^i(h(v),h(w))=cp^i(v,w)\text{ for all }v,w\in M\otimes_W R \}. 
\end{equation*}
Then we have an isomorphism
\begin{equation*}
\GU(M)(R)\xrightarrow{\cong} \GL(\varepsilon_0M\otimes_{W}R)\times R^{\times};h\mapsto (h_0,\sml(h))
\end{equation*}
(see Definition \ref{e0e1} for the definition of $\varepsilon_0$). The proof is the same as the isomorphy of the first homomorphism as in Lemma \ref{htis} (i). Moreover, if we set
\begin{equation*}
\GU^0(M)(R):=\{(h,c)\in \GU(M)(R)\mid c^2=\det{}_{\Zps \otimes_{\Zp}R}(h)\}, 
\end{equation*}
then the isomorphism above induces the following isomorphism: 
\begin{equation*}
\GU^0(M)(R)\cong \{(h_0,c)\in \GL(\varepsilon_0M\otimes_{W}R)\times R^{\times}\mid c^2=\det(h_0)\}. 
\end{equation*}
On the other hand, we have a bijection
\begin{equation*}
\Lag_{M}^{\Zps}(R)\xrightarrow{\cong}\{ R\text{-direct summands of }\varepsilon_0M\otimes_{W}R\text{ of rank }2\};\sM \mapsto \varepsilon_0 \sM. 
\end{equation*}
which commutes with the canonical actions of $\GU(M)(R)\cong \GL(\varepsilon_0M\otimes_{W}R)\times R^{\times}$. Since $\GU^0(M)(R)$ acts transitively on the set of all $R$-direct summands of $M\otimes_WR$ of rank $2$, so does $\Lag_{M}^{\Zps}(R)$. Hence the surjectivity follows. 

Next, we give an inverse of $f$. Take $\sM \in \Lag_{M}^{\Zps}(R)$. By localizing $R$, we may assume that $\sM$ is free over $R$. Take $v\in L^{\sharp}\otimes_W R$ such that $Rv\in \Isot_{L^{\sharp}}(R)$ and $\sM=\sM(Rv)$ (this is possible by the proof of the bijectivity of $f$). Then, by (\ref{kieq}) we have 
\begin{equation*}
l(\sM)=l(v(M\otimes_W R))=Rv, 
\end{equation*}
which implies the assertion for $l$. On the other hand, for $w\in L^{\sharp}$, we have $v\perp w$ if and only if $wvC(L^{\sharp})\subset vC(L^{\sharp})$ by the same argument as in the proof of \cite[Lemma 2.16]{hp}. Moreover, the condition $wvC(L^{\sharp})\subset vC(L^{\sharp})$ is interpreted as $w(\sM) \subset \sM$ by Lemma \ref{mreq}. Hence the equality for $l(\sM)^{\perp}$ follows. 
\end{proof}

\begin{proof}[Proof of Proposition \ref{dmil}]
Put $\Fil^1_x:=pF^{-1}(M)/pM$. Then it is a $\Zps$-stable Lagrangian subspace of $M\otimes_W\Fpbar \cong M/pM$. Let $\O \in \sC$. By the Grothendieck-Messing theory (see \cite[Chapter V, Theorem (1.6)]{mes}), we have a bijection
\begin{equation*}
\GM_{\O}\colon \widehat{\M}_{H,x}(\O)\xrightarrow{\cong}\Lift_{\O}^{\Zps}(\Fil^1_x). 
\end{equation*}
On the other hand, the bijection $\sM \mapsto l_{\O}(\sM)$ in Lemma \ref{liat} induce a bijection
\begin{equation*}
l_{x,\O}\colon \Lift_{\O}^{\Zps}(\Fil^1_{x})\xrightarrow{\cong}\Lift_{\O}(\Fil^1\Phi(L)). 
\end{equation*}
Now, put $f_{\O}:=l_{x,\O}\circ \GM_{\O}$. Then, $f_{\O}$ is bijective and satisfies (i) by Proposition \ref{dmil} (i), (ii) and (iv). Next, we prove (ii). For $\widetilde{x}\in \widehat{\M}_{H,x}(\O)$, we have $\widetilde{x}\in \widehat{\cZ}(\bv)_{x}(\O)$ if and only if $v(\GM_{\O}(\widetilde{x}))\subset \GM_{\O}(\widetilde{x})$ for any $v\in \bv$. By Proposition \ref{dmil} (iii), it is equivalent to the condition that $\bv_{\O}\subset l_{\O}(\GM_{\O}(\widetilde{x}))^{\perp}=f_{\O}(\widetilde{x})^{\perp}$. 
\end{proof}


\begin{thebibliography}{99}
\bibitem[Bas74]{bas}
H.~Bass, \emph{Clifford algebras and spinor norms over a commutative ring}, Amer.~J.~Math.~\textbf{96} (1974), 156--206. 
\bibitem[BO78]{bo}
P.~Berthelot, A.~Ogus, \emph{Notes on crystalline cohomology}, Princeton University Press, 1978. 
\bibitem[Cho18]{cho}
S.~Cho, \emph{The basic locus of the unitary Shimura variety with parahoric level structure, and special cycles}, preprint, arXiv:1807.09997, 2018. 
\bibitem[Fan11]{fan}
Y.~Fang, \emph{Zeta functions of complexes from $\PGSp(4)$}, Ph.~D.~thesis, The Pennsylvania State of University, 2011. 
\bibitem[Fu15]{fu}
L.~Fu, \emph{Etale cohomology theory, revised edition}, Nankai Tracts in Mathematics, vol.~14, 2015. 
\bibitem[Gar97]{gar}
P.~Garrett, \emph{Buildings and classical groups}, Chapman \& Hall, London, 1997. 
\bibitem[Hai05]{hai}
T.~J.~Haines, \emph{Introduction to Shimura varieties with bad reduction of parahoric type}, in \emph{Harmonic analysis, the trace formula, and Shimura varieties, Proc.~Clay Mathematics Institute 2003 Summer School, The Fields Institute Toronto, 2-27, June 2003}, Clay Mathematical Proceedings, vol.~4, eds J.~Arthur, D.~Ellwood and R.~E.~Kottwitz (American Mathematical Society/Clay Mathematics Institute, Providence, RI/Cambridge, MA, 2005), 583--658. 
\bibitem[HLZ19]{hlz}
X.~He, C.~Li, Y.~Zhu, \emph{Fine Deligne--Lusztig varieties and arithmetic fundamental lemma}, Forum Math.~Sigma \textbf{7} (2019), e47. 
\bibitem[HPR18]{hpr}
X.~He, G.~Pappas, M.~Rapoport, \emph{Good and semi-stable reductions of Shimura varieties}, preprint, arXiv:1804.09615, 2018. 
\bibitem[HP14]{hp}
B.~Howard, G.~Pappas, \emph{On the supersingular locus of $\GU(2,2)$ Shimura variety}, Algebra Number Theory \textbf{8} (2014), 1659--1699. 
\bibitem[HP17]{hp2}
B.~Howard, G.~Pappas, \emph{Rapoport--Zink spaces for spinor groups}, Compos.~Math.~\textbf{153} (2017), 1050--1118. 
\bibitem[KO87a]{ko}
T.~Katsura, F.~Oort, \emph{Families of supersingular abelian surfaces}, Compos.~Math.~\textbf{62} (1987) no.~2, 107--167. 
\bibitem[KO87b]{ko2}
T.~Katsura, F.~Oort, \emph{Supersingular abelian varieties of dimension two or three and class numbers}, in \emph{Algebraic geometry, Sendai, 1985}, Adv.~Stud.~Pure~Math.~\textbf{10} 253--281, North-Holland, Amsterdam, 1987. 
\bibitem[KMPS]{kmps}
M.~Kisin, K.~Madapusi Pera, S.~W.~Shin, \emph{Honda--Tate theory for Shimura varieties}, preprint, \texttt{https://math.berkeley.edu/\~{}swshin/HT.pdf}. 
\bibitem[Kot85]{kot2}
R.~E.~Kottwitz, \emph{Isocrystals with additional structure}, Compos.~Math.~\textbf{56} (1985), 201--220. 
\bibitem[Kot92]{kot3}
R.~E.~Kottwitz, \emph{Points on some Shimura varieties over finite fields}, J.~Amer.~Math.~Soc.~\textbf{5} (1992) no.~2, 373--444. 
\bibitem[Kud02]{kud}
S.~Kudla, \emph{Derivatives of Eisenstein series and generating functions for arithmetic cycles}, S{\'e}minaire Bourbaki, 52 ann{\'e}e, 1999--2000, no.~876.~Ast{\'e}risque \textbf{276} (2002), 341--368. 
\bibitem[KR00]{kr}
S.~Kudla, M.~Rapoport, \emph{Cycles on Siegel threefolds and derivatives of Eisenstein series}, Ann. Sci. {\'E}cole Norm.~Sup.~(4) \textbf{33} (2000), no.~5, 695--756. 
\bibitem[Lan02]{lan}
S.~Lang, \emph{Algebra}, Graduate Texts in Mathematics 211, Springer-Verlag, New York, third edition, 2002. 
\bibitem[LO98]{lo}
K.-Z.~Li, F.~Oort, \emph{Moduli of supersingular abelian varieties}, Lecture Notes in Mathematics 1680, Springer-Verlag, Berlin, 1998. 
\bibitem[LZhu18]{lz}
C.~Li, Y.~Zhu, \emph{Arithmetic intersection on GSpin Rapoport--Zink spaces}, Compos. Math.~\textbf{154} (2018), 1407--1440. 
\bibitem[LZha19]{Li2019}
C.~Li, W.~Zhang, \emph{Kudla--Rapoport cycles and derivatives of local densities}, preprint, arXiv:1908.01701, 2019. 
\bibitem[Mes72]{mes}
W.~Messing, \emph{The crystals associated to Barsotti-Tate groups: with application to abelian schemes}, Lecture Notes in Mathematics 264, Springer-Verlag, Berlin-New York, 1972. 
\bibitem[Mie20]{mie}
Y.~Mieda, \emph{On irreducible components of Rapoport--Zink spaces}, Int.~Math.~Res.~Not., IMRN 2020, no.~8, 2361--2407. 
\bibitem[Mum70]{mum}
D.~Mumford, \emph{Abelian varieties}, Tata Institute of Fundamental Research Studies in Mathematics, no.~5, Published for the Tata Institute of Fundamental Research, Bombay, 1970. 
\bibitem[RSZ18]{rsz}
M.~Rapoport, B.~Smithling, W.~Zhang, \emph{Regular formal moduli spaces and arithmetic transfer conjectures}, Math.~Ann.~\textbf{370} (2018), no. 3--4, 1079-–1175. 
\bibitem[RTW14]{rtw}
M.~Rapoport, U.~Terstiege, S.~Wilson, \emph{The supersingular locus of the Shimura variety for $\GU(1,n-1)$ over a ramified prime}, Math.~Z.~\textbf{276} (2014) no.~3--4, 1165--1188. 
\bibitem[RTZ13]{rtz}
M.~Rapoport, U.~Terstige, W.~Zhang, \emph{On the arithmetic fundamental lemma in the minuscule case}, Compos.~Math.~\textbf{149} (2013), 1631--1666. 
\bibitem[RZ96]{rz}
M.~Rapoport, Th.~Zink, \emph{Period spaces for $p$-divisible groups}, Annals of Mathematics Studies, vol.~141, Princeton University Press, Princeton, NJ, 1996. 
\bibitem[Sch85]{sch}
W.~Scharlau, \emph{Quadratic and hermitian forms}, Grundlehren der mathematischen Wissenschaften, A series of Comprehensive Studies in Mathematics, vol.~270, Springer-Verlag, Berlin, 1985. 
\bibitem[Shi10]{shi}
G.~Shimura, \emph{Arithmetic of quadratic forms}, Springer Monographs in Mathematics, Springer, New York, 2010. 
\bibitem[Ter11]{ter}
U.~Terstiege, \emph{Intersections of arithmetic Hirzeburch-Zagier cycles}, Math.~Ann.~\textbf{349} (2011), 161--213. 
\bibitem[Vol10]{vol}
I.~Vollaard, \emph{The supersingular locus of the Shimura variety for $\GU(1,s)$}, Canad.~J.~Math.~\textbf{62} (2010) no.~3, 668--720. 
\bibitem[VW11]{vw}
I.~Vollaard, T.~Wedhorn, \emph{The supersingular locus of the Shimura variety for $\GU(1,n-1)$ II}, Invent.~Math.~\textbf{184} (2011) no.~3, 591--627. 
\bibitem[Wil00]{wil}
R.~A.~Wilson, \emph{Finite simple groups}, Graduate Texts in Mathematics 251, Springer, 2000. 
\bibitem[Wu16]{wu}
H.~Wu, \emph{The supersingular locus of the Shimura varieties with exotic good reduction}, Ph.~D.~thesis, University of Duisburg-Essen, arXiv:1609.08775, 2016. 
\bibitem[Wan19]{wan}
H.~Wang, \emph{On the Bruhat--Tits stratification of a quaternionic unitary Rapoport--Zink space}, preprint, arXiv:1907.05999, 2019. 
\bibitem[Yu11]{yu}
C.~F.~Yu, \emph{Geometry of the Siegel modular threefold with paramodular level structure}, Proc.~Amer.~Math.~Soc.~\textbf{139} (2011), no.~9, 3181–-3190. 
\bibitem[Zin01]{zin}
Th.~Zink, \emph{Windows for displays of $p$-divisible groups}, pp.~491--518 in \emph{Moduli of abelian varieties} (Texel, 1999), edited by C.~Faber et al., Progr.~Math.~\textbf{195}, Birkh{\"a}user, Basel, 2001. 
\bibitem[EGA~IV-0]{ega40}
J.~Dieudonn{\'e}, A.~Grothendieck, \emph{{\'E}l{\'e}ments de g{\'e}om{\'e}trie alg{\'e}brique, IV.~{\'E}tude locale des sch{\'e}mas et des morphismes de sch{\'e}mas, Premi{\`e}re partie}, Inst.~Hautes {\'E}tudes Sci.~Publ.~Math.~no.~20. 
\end{thebibliography}
\end{document}